\documentclass[11pt,reqno]{amsart}

\usepackage[letterpaper, margin=1in]{geometry}
\usepackage{amsfonts}
\usepackage{amssymb}
\usepackage{amsmath}
\usepackage{mathrsfs}
\usepackage{color}
\usepackage[nospace,noadjust]{cite}
\usepackage{verbatim}
\usepackage{dsfont}
\usepackage{bbm}
\usepackage{enumerate}
\usepackage{graphicx}
\usepackage{xcolor} 
\usepackage{slashed}
\usepackage[titletoc,title]{appendix}
\usepackage{wrapfig}
\usepackage{marginnote}
\usepackage{bm}
\usepackage{enumitem}
\usepackage{hyperref}

\newcommand{\R}{{\mathbb R}}

\newcommand{\C}{{\mathbb C}}

\newcommand{\N}{{\mathbb N}}

\newcommand{\eps}{{\varepsilon}}

\newcommand{\vertiii}[1]{{\left\vert\kern-0.25ex\left\vert\kern-0.25ex\left\vert #1 
\right\vert\kern-0.25ex\right\vert\kern-0.25ex\right\vert}}

\DeclareMathOperator*{\supp}{supp}

\DeclareMathOperator*{\Span}{span}

\definecolor{deepgreen}{cmyk}{1,0,1,0.5}

\newcommand{\Del}[1]{}

\numberwithin{equation}{section}

\newtheorem{theorem}{Theorem}[section]
\newtheorem{lemma}[theorem]{Lemma}
\newtheorem{proposition}[theorem]{Proposition}
\newtheorem{assumption}[theorem]{Assumption}
\newtheorem{remark}[theorem]{Remark}
\newtheorem{definition}[theorem]{Definition}

\newcommand{\mwhere}{{\ \ \text{where} \ \ }}
\newcommand{\mand}{{\ \ \text{and} \ \  }}

\newcommand{\mfor}{{\ \ \text{for} \ \ }}

\newcommand{\jap}[1]{\langle #1\rangle}

\renewcommand{\hbar}{{\underline h}}

\newcommand{\calC}{\mathcal C}
\newcommand{\calD}{\mathcal D}
\newcommand{\calE}{\mathcal E}
\newcommand{\calF}{\mathcal F}
\newcommand{\calG}{\mathcal G}
\newcommand{\calH}{\mathcal H}

\newcommand{\calJ}{\mathcal J}

\newcommand{\calL}{\mathcal L}
\newcommand{\calM}{\mathcal M}
\newcommand{\calN}{\mathcal N}
\newcommand{\calO}{\mathcal O}

\newcommand{\calQ}{\mathcal Q}
\newcommand{\calR}{\mathcal R}
\newcommand{\calS}{\mathcal S}

\newcommand{\calV}{\mathcal V}
\newcommand{\calW}{\mathcal W}

\newcommand{\frakL}{\mathfrak L}

\newcommand{\wtilg}{{\widetilde{g}}}

\newcommand{\tilr}{{\tilde{r}}}

\newcommand{\wtilB}{{\widetilde{B}}}

\newcommand{\wtilM}{{\widetilde{M}}}

\newcommand{\wtilPhi}{{\widetilde{\Phi}}}

\newcommand{\hatf}{{\hat{f}}}

\newcommand{\checkf}{{\check{f}}}

\newcommand{\vecf}{{\vec f}}
\newcommand{\vecg}{{\vec g}}

\newcommand{\vecr}{{\vec r}}

\newcommand{\vecu}{{\vec u}}
\newcommand{\vecv}{{\vec v}}

\newcommand{\vecPsi}{\vec{\Psi}}

\newcommand{\ud}{\mathrm{d}}

\DeclareMathOperator{\sech}{sech}

\DeclareMathOperator{\spec}{spec}
\DeclareMathOperator{\Ran}{Ran}
\DeclareMathOperator{\Tr}{Tr}

\AtBeginDocument{
	\label{CorrectFirstPageLabel}
	
}

\begin{document}
\title{Dispersive estimates for 1D matrix Schr\"odinger operators with threshold resonance}
	\author{Yongming Li}
	\address{Department of Mathematics \\ Texas A\&M University \\ College Station, TX 77843, USA}
	\email{liyo0008@tamu.edu}
   \thanks{The author was partially supported by NSF grants DMS-1954707
	and DMS-2235233.}
	\begin{abstract}
We establish dispersive estimates and local decay estimates for the time evolution of non-self-adjoint matrix Schr\"odinger operators with threshold resonances in one space dimension. In particular, we show that the decay rates in the weighted setting are the same as  in the regular case after subtracting a finite rank operator corresponding to the threshold resonances. Such matrix Schr\"odinger operators naturally arise from linearizing a focusing nonlinear Schr\"odinger equation around a solitary wave. It is known that the linearized operator for the 1D focusing cubic NLS equation exhibits a threshold resonance. We also include an observation of a favorable structure in the quadratic nonlinearity of the evolution equation for perturbations of solitary waves of the 1D focusing cubic NLS equation. 
	\end{abstract}
\maketitle 
	\tableofcontents
%
\section{Introduction}
In this article, we establish dispersive estimates and local decay estimates  for the (non-self-adjoint) matrix Schr\"odinger operators
\begin{equation}\label{eqn: matrix operator H}
	\calH = {\calH}_0 + \calV = \begin{bmatrix}
		-\partial_x^2 + \mu & 0 \\ 0 & \partial_x^2 - \mu
	\end{bmatrix} + \begin{bmatrix}
		-V_1 & -V_2 \\ V_2 & V_1
	\end{bmatrix}\quad \text{on $L^2(\R) \times L^2(\R)$},
\end{equation}
where $\mu$ is a positive constant and $V_1$, $V_2$ are real-valued sufficiently decaying potentials. The operator $\calH$ is closed on the domain $D(\calH) = H^2(\R) \times H^2(\R)$. 

These matrix operators arise when linearizing a focusing nonlinear Schr\"odinger equation around a solitary wave. By our assumptions on $V_1$ and $V_2$, Weyl's criterion implies that the essential spectrum of $\calH$ is the same as that of $\calH_0$, given by $(-\infty,-\mu] \cup [\mu,\infty)$. As a core assumption in this paper, we suppose that the edges $\pm \mu$ of the essential spectrum are irregular in the sense of Definition~\ref{def: regular points}. This implies that there exist non-trivial bounded solutions to the equation $\calH  \vec\Psi_\pm = \pm \mu \vec\Psi_\pm$, see Lemma~\ref{lemma: threshold resonance characterization}. The dispersive estimates for $\calH$ when the thresholds $\pm \mu$ are regular have been obtained in Section 7-8 of the paper by Krieger-Schlag \cite{06KriegerSchlag}, building on the scattering theory developed by Buslaev-Perel'man \cite{95BusleavPerelman}. See also the recent work of Collot-Germain \cite{23CollotGermain}. Our proof is instead based on the unifying approach to resolvent expansions first initiated by Jensen-Nenciu \cite{00JensenNenciu}, and then further refined in Erdogan-Schlag \cite{06ErdoganSchlag} for matrix Schr\"odinger operators. We also adopt techniques from Erdogan-Green \cite{22ErdoganGreen}, where the authors prove  similar dispersive estimates for one-dimensional Dirac operators. 

\subsection{Motivation} Our interest in developing dispersive estimates for \eqref{eqn: matrix operator H} stems from the asymptotic stability problem for solitary wave solutions to nonlinear Schr\"odinger (NLS) equations. The NLS equation
\begin{equation}\label{eqn: NLS1}
	\begin{split}
		i\partial_t \psi + \partial_x^2 \psi + F(\vert \psi \vert^2)\psi &= 0, \mwhere \psi\colon \R_t \times \R_x \rightarrow \C,
	\end{split}
\end{equation}
appears in many important physical contexts such as the propagation of a laser beam, the envelope description of water waves in an ideal fluid, or the propagation of light waves in nonlinear optical fibers. See, e.g., Sulem-Sulem \cite{sulem2007nonlinear} for physics background. 

Under certain general conditions on the nonlinearity $F(\cdot)$ (see, e.g., \cite{83BerestyckiLions}),  the equation \eqref{eqn: NLS1} admits a parameterized family of localized, finite energy, traveling solitary waves of the form $\psi(t,x) = e^{it\alpha^2}\phi(x;\alpha)$, where $\phi(\cdot;\alpha)$ is a ground state, i.e., a positive, decaying, real-valued solution to the (nonlinear) elliptic equation
\begin{equation}
	- \partial_x^2 \phi + \alpha^2 \phi = F(  \phi^2)\phi.
\end{equation}
The existence and uniqueness of these ground state solutions are well-understood, see, e.g., \cite{83BerestyckiLions}, \cite{89Kwong}.

The solitary wave solutions (or simply, \textit{solitons}) are of importance due to the special role they play for the long-time dynamics of the Cauchy problem \eqref{eqn: NLS1}. Consequently, over the last few decades there has been a significant interest in the study of stability (or instability) of such solitary waves under small perturbations. The primary notion of stability is that of orbital stability, and it is by now well-understood for the NLS equation. The pioneering works in this direction were due to Cazenave-Lions \cite{82CazenaveLions}, Shatah-Strauss \cite{85ShatahStrauss}, and Weinstein \cite{85Weinstein}; see also \cite{87Grillakis} for the general theory. On the other hand, a stronger notion of stability is that of asymptotic stability. There are two general approaches for the asymptotic stability problem. The first approach is to use integrability techniques, when the underlying partial differential equation is completely integrable and inverse scattering is available. A second approach is perturbative, which means that one studies the dynamics of the nonlinear flow in the neighborhood of the solitary wave, on a restricted set of the initial data. Generally, one starts by decomposing the perturbed solution into a sum of a solitary wave and a dispersive remainder term. For the perturbative approach, dispersive estimates for the linear flow are key. 

Let us briefly describe the perturbative approach for the NLS equation. To keep our exposition short, we will not take into account any modulation aspects related to the Galilean invariance of the equation. For small $\alpha>0$, consider the perturbation ansatz $\psi(t,x) = e^{it\alpha^2}(\phi(x) + u(t,x))$ with the ground state $\phi(\cdot) = \phi(\cdot;\alpha)$ and the dispersive remainder term $u(t,x)$. The linearization of \eqref{eqn: NLS1} around the solitary wave $e^{it\alpha^2}\phi(x)$ then leads to the following nonlinear partial differential equation 
\begin{equation*}
	i \partial_t u = (- \partial_x^2  + \alpha^2 - V)u + W \overline{u}+ N,
\end{equation*}
where $N = N(\phi,u,\overline{u})$ is nonlinear in the variables $(u,\overline{u})$, and $V = F(\phi^2) + F'(\phi^2)\phi^2$ and $W = F'(\phi^2)\phi^2$ are real-valued potentials related to the ground state $\phi$. Equivalently, the above equation can be recast as a system for the vector $U := (u,\overline{u})^\top$, which is given by 
\begin{equation}\label{eqn: linearized NLS}
	i\partial_t U - \calH U = \calN, 
\end{equation}
where $\calN$ is a nonlinear term, and $\calH$ is a matrix Schr\"odinger operator of the form \eqref{eqn: matrix operator H} with the parameters $\mu = \alpha^2$, $V_1 = V$, and $V_2 = W$. 

For the study of asymptotic stability of solitary waves for NLS, it is thus crucial to fully understand the spectral properties of the matrix operator $\calH$ as well as to derive dispersive estimates for the linear evolution operator $e^{it\calH}$. One of the key steps in a perturbative analysis is to prove that the dispersive remainder \eqref{eqn: linearized NLS} decays to zero in a suitable topology. Let us consider for example, the 1D focusing NLS with a pure power nonlinearity, i.e.
\begin{equation}\label{eqn: NLS}
	i\partial_t \psi + \partial_x^2 \psi + \vert \psi \vert^{2\sigma}\psi = 0,\mwhere \sigma>0.
\end{equation}
The ground state $\phi(x;1)$ has an explicit formula for all $\sigma > 0$ given by
\begin{equation}
	\phi(x;1) = (\sigma + 1)^{\frac{1}{2\sigma}}\sech^{\frac{1}{\sigma}}(\sigma x),
\end{equation}
and the linearized operator around $e^{it}\phi(x;1)$ takes the form 
\begin{equation*}
	\calH_\sigma = \begin{bmatrix}
		-\partial_x^2 - (\sigma+1)^2\sech^2(\sigma x) + 1 & - \sigma(\sigma+1)\sech^2(\sigma x) \\ \sigma(\sigma+1)\sech^2(\sigma x) & \partial_x^2 + (\sigma+1)^2\sech^2(\sigma x) - 1
	\end{bmatrix}.
\end{equation*}
For monomial nonlinearities, we may obtain $\phi(x;\alpha)$ from rescaling by $\phi(x;\alpha) = \alpha^{\frac{1}{\sigma}}\phi(\alpha x,1)$. The matrix operators when linearizing around $e^{it\alpha^2}\phi(x;\alpha)$ are also equivalent to the matrix operator $\calH_\sigma$ by rescaling. The spectra for these matrix operators were investigated in \cite{07ChangGustafsonNakanishiTsai}; see also Section 9 of \cite{06KriegerSchlag}. For $\sigma \geq 2$, Krieger-Schlag \cite{06KriegerSchlag} were able to construct finite co-dimensional center-stable manifolds around the solitary waves and prove asymptotic stability using dispersive and Strichartz estimates developed for the evolution operator $e^{it\calH}$. However, for the completely integrable case ($\sigma =1$), it was shown in \cite{07ChangGustafsonNakanishiTsai} that the matrix operator $\calH_1$ exhibits the threshold resonance $\Psi(x) = \big(\tanh^2(x),-\sech^2(x) \big)^\top$ at $\lambda = 1$. The dispersive estimates developed in \cite{06KriegerSchlag} do not apply in this case. Furthermore, we note that a key assumption in the papers \cite{95BusleavPerelman}, \cite{05GangSigal}, \cite{06KriegerSchlag}, \cite{23CollotGermain} is that the linearized matrix operator $\calH$ does not possess threshold resonances at the edges of the essential spectrum. In these ``generic" (regular) cases, it can be shown that the evolution operator enjoy improved decay estimates in weighted spaces; see, e.g., Proposition~8.1 in \cite{06KriegerSchlag}. Thus, a meaningful motivation for this paper is to prove dispersive estimates in the presence of threshold resonances under some general spectral assumptions on the matrix operator $\calH$, which are applicable to the 1D cubic NLS case ($\sigma=1$). We will discuss this particular case briefly in Section~\ref{section: cubic linearized sect}.

\subsection{Main result} We are now in the position to state the main result of this paper. We begin by specifying some spectral assumptions on $\calH$. 
\begin{assumption} \label{assumption: spectral}\ 
	\begin{enumerate}
		\item [(A1)] $-\sigma_3 \calV$ is a positive matrix, where $\sigma_3$ is one of the Pauli matrices (c.f.~\eqref{eqn: pauli matrices}), \label{assumption (A1)}
		\item [(A2)]$L_- := -\partial_x^2 + \mu - V_1 + V_2$ is non-negative,\label{assumption (A2)}
		\item [(A3)] there exists $\beta>0$ such that $\vert V_1 (x) \vert + \vert V_2 (x) \vert \lesssim e^{-(\sqrt{2\mu}+\beta)\vert x \vert} $ for all $x \in \R$,\label{assumption (A3)}
		\item [(A4)] there are no embedded eigenvalues in $(-\infty,-
		\mu)\cup(\mu,\infty)$. \label{assumption (A4)}
	\end{enumerate}
\end{assumption}
Under these assumptions, we recall the general spectral theory for $\calH$ from \cite{06ErdoganSchlag}.\footnote{The results in Section 2 of \cite{06ErdoganSchlag} are stated for dimension 3, but they in fact hold for all dimensions. Moreover, only a polynomial decay on $V_1$ and $V_2$ is assumed in \cite{06ErdoganSchlag}. See also \cite[Theorem 1.3]{07HundertmarkLee}.}
\begin{lemma}\cite[Lemma 3]{06ErdoganSchlag}\label{lemma: 06ESLem3}
	Suppose Assumption \ref{assumption: spectral} holds. The essential spectrum of $\calH$ equals $(-\infty,-\mu] \cup [\mu,\infty)$. Moreover, 
	\begin{equation}\label{eqn: symmetry of H}
		\spec(\calH) = -\spec(\calH) = \overline{\spec(\calH)} = \spec(\calH^*),
	\end{equation}
and $\spec(\calH) \subset \R \cup i\R$. The discrete spectrum of $\calH$ consists of eigenvalues $\{z_j\}_{j=1}^N$, $0\leq N < \infty$, of finite multiplicity. For each $z_j \neq 0$, the algebraic and geometric multiplicities coincide and $\Ran(\calH-z_j)$ is closed. The zero eigenvalue has finite algebraic multiplicity, i.e., the generalized eigenspace $\cup_{k=1}^\infty \ker(\calH^k)$ has finite dimension. In fact, there is a finite $m \geq 1$ so that $\ker(\calH^k) = \ker(\calH^{k+1})$ for all $k \geq m$. 
\end{lemma}
The symmetry \eqref{eqn: symmetry of H} is due to the following commutation properties of $\calH$,
\begin{equation}\label{eqn: symmetry property}
	\calH^* = \sigma_3 \calH \sigma_3, \qquad -\calH =\sigma_1 \calH \sigma_1,
\end{equation}
with the Pauli matrices
\begin{equation}\label{eqn: pauli matrices}
	\sigma_1 = \begin{bmatrix}
		0 & 1 \\ 1 & 0
	\end{bmatrix},\quad \sigma_2 = \begin{bmatrix}
		0 & -i \\ i & 0
	\end{bmatrix},\quad\sigma_3 = \begin{bmatrix}
		1 & 0 \\ 0 & -1
	\end{bmatrix}.
\end{equation}
As a core assumption in this paper, we impose that the thresholds $\pm \mu$ of the essential spectrum are irregular.  
\begin{assumption}\label{assumption: threshold}\ 
	\begin{enumerate}
		\item [(A5)] The thresholds $\pm \mu$ are irregular in the sense of Definition~\ref{def: regular points}. This implies that there exist non-trivial  bounded solutions $\vec{\Psi}_\pm = (\Psi_1^\pm,\Psi_2^\pm)^\top$ to the equation $\calH \vec{\Psi}_\pm = \pm \mu \vec{\Psi}_\pm$. \label{assumption (A5)}
		\item [(A6)] The vanishing (bilateral)-Laplace transform condition holds \label{assumption (A6)}
		\begin{equation}
			\frakL[V_2\Psi_1^+ + V_1 \Psi_2^+](\pm \sqrt{2\mu}) = \int_{-\infty}^\infty e^{\mp \sqrt{2\mu}} (V_2 \Psi_1^+ + V_1 \Psi_2^+)(y)\, \ud y = 0.
		\end{equation}
	\end{enumerate}
\end{assumption}
For details about the characterization of the threshold functions $\vec{\Psi}$, we refer the reader to Definition~\ref{def: regular points} and Lemma~\ref{lemma: threshold resonance characterization} in Section 4. Due to the commutation identity \eqref{eqn: symmetry property}, we have the relation $\vec\Psi_+ = \sigma_1 \vec\Psi_-$. We emphasize that assumption (A6) is used to infer that (non-trivial) bounded solutions $\vec{\Psi}_\pm = (\Psi_1^\pm,\Psi_2^\pm)$ to the equation $\calH\vec\Psi_\pm = \pm \mu \vec\Psi_\pm$ satisfy $\Psi_1^+ = \Psi_2^- \in L^\infty(\R)\setminus L^2(\R)$. 

Let $P_\mathrm{d}\colon L^2(\R)\times L^2(\R) \to L^2(\R) \times L^2(\R)$ be the Riesz projection corresponding to the discrete spectrum of $\calH$, and define $P_{\mathrm{s}} := I - P_\mathrm{d}$. We now state the main theorem of this article.
\begin{theorem}\label{theorem: local decay estimate} Suppose assumptions (A1) -- (A6) hold, and let $\vec{\Psi} = (\Psi_1,\Psi_2)$ be the $L^\infty(\R)\times L^\infty(\R)\setminus L^2(\R) \times L^2(\R)$ distributional solution to
	\begin{equation}
		\calH\vec{\Psi} = \mu \vec{\Psi},
	\end{equation}
	with the normalization 
	\begin{equation}\label{eqn: normalization condition}
		\lim_{x \to \infty}\left( \vert \Psi_1(x)\vert^2 + \vert \Psi_1(-x) \vert^2 \right) = 2.
	\end{equation}
	Then, for any $\vec{f}=(f_1,f_2) \in \calS(\R) \times \calS(\R)$, we have
	\begin{enumerate}
		\item the unweighted dispersive estimate 
		\begin{equation}
			\left\Vert e^{it\calH}P_{\mathrm{s}} \vec{f}\, \right\Vert_{L^\infty(\R)\times L^\infty(\R)} \lesssim \vert t \vert^{-\frac{1}{2}} \left\Vert \vec{f} \, \right\Vert_{L^1(\R) \times L^1(\R)}, \quad \forall\, \vert t \vert \geq 1,
		\end{equation}
		\item and the weighted dispersive estimate
		\begin{equation}\label{eqn: theorem local decay estimate}
			\left\Vert \langle x \rangle^{-2} (e^{it\calH}P_{s} - F_t)\vec{f}\, \right\Vert_{L^\infty(\R)\times L^\infty(\R)} \lesssim \vert t \vert^{-\frac{3}{2}} \left\Vert \langle x \rangle^{2} \vec{f} \,\right\Vert_{L^1(\R) \times L^1(\R)}, \quad \forall \, \vert t\vert  \geq 1,
		\end{equation}
		where
		\begin{equation}\label{eqn: def of F_t}
			F_t\vec{f} := \frac{ e^{it\mu}}{\sqrt{-4 \pi i t}} \langle \sigma_3\vec{\Psi},\vec{f}\, \rangle \vec{\Psi} - \frac{e^{-it\mu}}{\sqrt{4\pi i t}}\langle \sigma_3 \sigma_1\vec{\Psi}, \vecf\, \rangle \sigma_1\vec{\Psi}.
		\end{equation}
	\end{enumerate}
\end{theorem}
We proceed with some remarks on the main theorem:
 \begin{enumerate}
		\item The estimate \eqref{eqn: theorem local decay estimate} is an analogue of the weighted dispersive estimates obtained by Goldberg \cite{07Goldberg} for the scalar Sch\"odinger operator $H = -\partial_x^2 + V$ on the real line for non-generic potentials $V$; see \cite[Theorem~2]{07Goldberg}. The local decay estimate \eqref{eqn: theorem local decay estimate} shows that the bulk of the free wave $e^{it\calH}P_\mathrm{s}$ enjoys improved local decay at the integrable rate $\calO(\vert t \vert^{-\frac{3}{2}})$, and that the slow $\calO(\vert t \vert^{-\frac{1}{2}})$ local decay can be pinned down to the contribution of the finite rank operator $F_t$. Such sharp information can be useful for nonlinear asymptotic stability problems, see also Section~\ref{section: cubic linearized sect} below. 

		\item We make some comments on the spectral hypotheses. The assumptions (A1)--(A4) are known to be satisfied by the linearized operator around the solitary wave for the 1D focusing power-type NLS \eqref{eqn: NLS}. In the case of the 1D focusing cubic NLS ($\sigma = 1$),  the linearized operator $\calH_1$ satisfies the assumptions (A1)--(A6); see Section~\ref{section: A6 for cubic} below. More generally, in Lemma~\ref{lemma: threshold resonance characterization}, we show for matrix operators $\calH$ of the form \eqref{eqn: matrix operator H} satisfying assumptions (A1)--(A6) that the edges $\pm \mu$ of the essential spectrum of $\calH$ cannot be eigenvalues, and that the non-trivial bounded solutions $\vecPsi_\pm = (\Psi_1^\pm,\Psi_2^\pm)^\top$ to $\calH \vec\Psi_\pm = \pm \mu \vec \Psi_\pm$ belong to $L^\infty \setminus L^2$ since $\Psi_1(x)$ has a non-zero limit as $x \to \pm \infty$. In this sense, we characterize the solutions $\vec\Psi_\pm$ as threshold resonances. However, it is not yet clear to the author whether  assumption (A6) is strictly needed to show that non-trivial bounded solutions $\vec\Psi_\pm$ to $\calH\vec\Psi_\pm = \pm \mu \vec\Psi_\pm$ cannot be eigenfunctions. Moreover, an inspection of the proof of Lemma~\ref{lemma: threshold resonance characterization} reveals that the strong exponential decay assumption (A3) and the vanishing condition assumption (A6) are only used in a Volterra integral equation argument. In all other proofs, we only use some polynomial decay  of the potentials $V_1$ and $V_2$.

		\item It might be possible to prove  Theorem~\ref{theorem: local decay estimate} using the scattering theory developed by \cite{95BusleavPerelman}. However, one major difficulty for this approach is due to the fact that the matrix Wronskian associated with the vector Jost solutions is not invertible at the origin for  cases where the matrix operators $\calH$ exhibit threshold resonances. Hence, the vector-valued distorted Fourier basis functions are not immediately well-defined at zero frequency. See Corollary~5.21 and Section~6 in \cite{06KriegerSchlag} for further details.
\end{enumerate}  

\subsection{Previous works}
In this subsection, we collect references related to dispersive estimates for Schr\"odinger operators and to the study of the stability of solitary waves.

For dispersive estimates for the matrix Schr\"odinger operator $\calH$, we refer to Section~5-9 of \cite{06KriegerSchlag} in dimension 1, and to \cite{06ErdoganSchlag, 09Marzuola, 12Green, 13ErdoganGreen, 17Toprak} in higher dimensions. A comprehensive study on the spectral theory for the matrix operator arising from pure-power type NLS is given in \cite{07ChangGustafsonNakanishiTsai}. See also \cite{10Vougalter, 11CostinHuangSchlag, 20Masaki} for related analytical and numerical studies. For dispersive estimates for the scalar Schr\"odinger operators, pioneering works include \cite{82Murata, 91Journe, 00Weder}, and we refer to \cite{04ErdoganSchlag, 04GoldbergSchlag, 05Schlag, 06Schlag, 07Goldberg, 10Goldberg, 11Mizutani, 12Green, 14ErdoganGoldbergGreen, 15GoldbergGreen, 16Beceanu, 17GoldbergGreen}
for a sample of recent works. Finally, we mention the papers \cite{bolle1985complete, 00JensenNenciu} on resolvent expansions for the scalar Schr\"odinger operator. 

On the general well-posedness theory for the NLS Cauchy problem \eqref{eqn: NLS1}, we refer to the pioneering works \cite{79GinibreVelo, 87Kato, 87Tsutsumi}. Results on the orbital stability (or instability) of solitary waves for the NLS equation were first obtained by \cite{81BerestckiCazenave, 82CazenaveLions, 85ShatahStrauss, 85Weinstein, 86Weinstein}, and a general theory was established in \cite{87Grillakis}. Subsequent developments for general nonlinearities were due to \cite{88Grillakis, 90Grillakis, 95Ohta, 03ComechPelinovsky, 08Maeda}. Regarding the asymptotic stability of solitary waves, the first results were due to Buslaev-Perel'man \cite{92BuslaevPerelman, 95BusleavPerelman}. Subsequent works in this direction were due to \cite{05GangSigal, 06KriegerSchlag, 08Beceanu, 09Schlag, 14Cuccagna, 21Martel, 23CollotGermain}. For  surveys on the stability of solitary waves, we refer to the reviews \cite{17KMM, 21CuccagnaMaeda} and the monographs \cite{cazenave2003semilinear, sulem2007nonlinear}.

\subsection{On the solitary wave for the 1D focusing cubic NLS}\label{section: cubic linearized sect}
In this subsection, we present two observations related to the asymptotic stability problem for the solitary wave of the 1D focusing cubic NLS. First, we verify that the assumption (A6) holds for the linearized operator around the solitary wave of the 1D focusing cubic NLS. Second, we use the local decay estimate \eqref{eqn: theorem local decay estimate} to shed some light on the leading order structure of the quadratic nonlinearity in the perturbation equation for the solitary wave of the 1D focusing cubic NLS. 

We note that a proof for the asymptotic stability problem has been given by Cuccagna-Pelinovsky \cite{14CuccagnaPelinovsky} via inverse scattering techniques. On the other hand, a perturbative proof that does not explicitly rely on the integrable structure has not yet appeared in the literature to the best of the author's knowledge. We now briefly discuss the evolution equation for perturbations of the solitary wave for the 1D focusing cubic NLS. To keep our exposition short, we do not discuss the modulation aspects for the solitary wave. For simplicity, consider the perturbation ansatz 
$$\psi(t,x) = e^{it}(Q(x)+u(t,x))$$ 
for the equation \eqref{eqn: NLS} ($\sigma = 1$). The ground state has the explicit formula
\begin{equation*}
	Q(x) := \phi(x;1) = \sqrt{2}\sech(x).
\end{equation*}
The evolution equation for the perturbation in vector form $\vec{u} = (u_1,u_2) :=(u,\bar{u})$ is given by 
\begin{equation}\label{eqn: linearized cubic NLS in complex-conjugate coordinates}
	\begin{split}
	&i \partial_t \vec{u} - \mathcal{H}_1\vec{u} = \calQ(\vecu) + \calC(\vec u),		
	\end{split}
\end{equation}
where 
\begin{equation}
	\calH_1 = \mathcal{H}_0 + \mathcal{V}_1 = \begin{bmatrix}-\partial_x^2 + 1 & 0 \\ 0 & \partial_x^2 - 1\end{bmatrix} + \begin{bmatrix}
		-4\sech^2(x) & - 2\sech^2(x) \\ 2\sech^2(x) & 4\sech^2(x)
	\end{bmatrix}, 
\end{equation}
and
\begin{equation}
	\calQ(\vecu) := \begin{bmatrix}
		- Qu_1^2 -  2Qu_1u_2 \\ Qu_2^2 + 2Qu_1u_2
	\end{bmatrix}, \mand \calC(\vecu) :=  \begin{bmatrix}
		- u_1^2u_2 \\  u_1u_2^2
	\end{bmatrix}.
\end{equation}
Recall from \cite{07ChangGustafsonNakanishiTsai} that the matrix operator $\calH_1$ has the essential spectrum $(-\infty,-1]\cup[1,\infty)$, and a four-dimensional generalized nullspace
\begin{equation}\label{eqn: generalized kernel}
	\calN_{\mathrm{g}}(\calH_1) = \Span\left\{\begin{bmatrix}Q \\ - Q\end{bmatrix}, \begin{bmatrix}
		(1+x\partial_x)Q \\ (1+x\partial_x)Q
	\end{bmatrix}, \begin{bmatrix}
	\partial_x Q \\ \partial_x Q
\end{bmatrix}, \begin{bmatrix}
x Q \\ - xQ
\end{bmatrix}\right\},
\end{equation} 
as well as a threshold resonance at $+1$ given by 
\begin{equation}
	\vec\Psi \equiv \vec\Psi_+ := \begin{bmatrix}
		\Psi_1 \\ \Psi_2
	\end{bmatrix} =  \begin{bmatrix}
		1-\frac{1}{2}Q^2\\ -\frac{1}{2}Q^2
	\end{bmatrix} = \begin{bmatrix}
	\tanh^2(x) \\ -\sech^2(x)
\end{bmatrix}.
\end{equation}
By symmetry, there is also a  threshold resonance function at $-1$ given by 
\begin{equation}
	\vec\Psi_- = \sigma_1\vec\Psi_+ = \begin{bmatrix}
		-\sech^2(x) \\ \tanh^2(x)
	\end{bmatrix}.
\end{equation}
The eigenfunctions listed in \eqref{eqn: generalized kernel} are related to the underlying symmetries for the NLS equation. Note that we have normalized the resonance function $\vec\Psi$ to satisfy the condition \eqref{eqn: normalization condition} stated in Theorem~\ref{theorem: local decay estimate}. 
\subsubsection{On assumption (A6) for the 1D focusing cubic NLS}\label{section: A6 for cubic}
Our first observation is that the assumption (A6) is satisfied by the matrix operator $\calH_1$.
\begin{lemma}\label{lemma: vanishing laplace transform} Let $V_1(x) = 4\sech^2(x)$, $V_2(x) = 2\sech^2(x)$, and $(\Psi_1(x),\Psi_2(x)) = (\tanh^2(x),-\sech^2(x))$. Then, we have
	\begin{equation}\label{eqn: cubic nls c_2}
		\int_{\R} e^{\pm \sqrt{2}y} \big(V_2(y) \Psi_1(y) + V_1(y)\Psi_2(y)\big) \, \ud y = 0.
	\end{equation}
	\begin{proof}
	We denote the (two-sided) Laplace transform by
	\begin{equation}
		\frakL[f](s) = \int_{-\infty}^\infty e^{-sy}f(y) \,\ud y, \quad s \in \C,
	\end{equation}
	which is formally related to the Fourier transform by 
	\begin{equation*}
		\frakL[f](s) = \sqrt{2\pi} \calF[f](is).
	\end{equation*}
	By direct computation, 
	\begin{equation*}
		(V_1\Psi_2+V_2\Psi_1)(x)  = 2\sech^2(x) - 6 \sech^4(x),
	\end{equation*}
	and
	\begin{equation}\label{identity: sech4}
		\sech^4(x) = \frac{2}{3}\sech^2(x)-\frac{1}{6}\partial_x^2(\sech^2(x)).
	\end{equation}
	Recall from \cite[Corollary 5.7]{21LS} that as equalities in $\calS(\R)$,
	\begin{equation}
		\calF[\sech^2](\xi) = \sqrt{\frac{\pi}{2}}\frac{\xi}{\sinh(\tfrac{\pi}{2}\xi)}.
	\end{equation}
	Hence, using the basic property $\calF[-\partial_x^2 f](\xi) = \xi^2 \calF[f](\xi)$ and \eqref{identity: sech4}, we obtain  
	\begin{equation}
		\calF[\sech^4](\xi) = \frac{1}{6}\sqrt{\frac{\pi}{2}} \frac{\xi(4+\xi^2)}{\sinh(\tfrac{\pi}{2}\xi)}.
	\end{equation}
	As complex functions, we recall that $\sinh(iz) = i \sin(z)$ and that $z \mapsto \frac{z}{\sin(z)}$ is analytic\footnote{to be pedantic, there is a removable singularity at $z=0$ which we can remove by setting the function $\frac{z}{\sin(z)}$ equal to $1$ at $z=0$.} in the strip $\{s+i\sigma: s \in (-\pi,\pi), \,\sigma \in \R\}$. Thus, by analytic continuation,
	\begin{equation*}
		\frakL[V_1 \Psi_2 + V_2 \Psi_1](s) = \sqrt{2\pi}\left(2\, \calF[\sech^2](is) - 6 \calF[\sech^4](is) \right) =  \frac{\pi s(-2+s^2)}{\sin(\tfrac{\pi s}{2})},
	\end{equation*}	
	for any $s \in \C$ with $\Re (s) \in (-2,2)$, which in particular proves the vanishing condition \eqref{eqn: cubic nls c_2}.
	\end{proof}
\end{lemma}
The other assumptions (A1)--(A5) for $\calH_1$ are also satisfied by either checking directly or invoking the results from Section~9 in \cite{06KriegerSchlag}. 

\subsubsection{Null structure for perturbations of the solitary wave of the 1D focusing cubic NLS}
Due to the slow local decay of the Schr\"odinger waves in the presence of a threshold resonance, the spatially localized quadratic nonlinearity in  \eqref{eqn: linearized cubic NLS in complex-conjugate coordinates} may pose significant difficulties for proving decay of small solutions to \eqref{eqn: linearized cubic NLS in complex-conjugate coordinates}. The weighted dispersive estimate \eqref{eqn: theorem local decay estimate} shows that the slow local decay is only due to the finite rank projection $F_t$. To shed some light on the expected leading order behavior of the quadratic nonlinearity $\calQ(\vecu)$ in \eqref{eqn: linearized cubic NLS in complex-conjugate coordinates}, it is instructive to insert a free Schr\"odinger wave 
\begin{equation*}
	\vecu_\mathrm{free}(t) := e^{-it\calH}P_\mathrm{s}\vecf,
\end{equation*}
for some fixed $\vecf \in \calS(\R) \times \calS(\R)$. By Theorem~\ref{theorem: local decay estimate}, we have
\begin{equation}\label{eqn: ufree}
	\vecu_\mathrm{free}(t) = c_- \frac{e^{-it}}{\sqrt{t}}\begin{bmatrix}
		\Psi_1 \\ \Psi_2
	\end{bmatrix} + c_+ \frac{e^{it}}{\sqrt{t}} \begin{bmatrix}
		\Psi_2 \\ \Psi_1
	\end{bmatrix} + \vecr(t),
\end{equation}
with
\begin{equation}
	c_- = \frac{1}{\sqrt{-4 \pi i}} \langle \sigma_3 \vec\Psi,\vecf\,\rangle, \quad c_+ = -\frac{1}{\sqrt{4\pi i}} \langle \sigma_3 \sigma_1 \vec\Psi,\vecf \,\rangle,
\end{equation}
and where the remainder $\vecr(t)$ satisfies
\begin{equation}
	\left \Vert \langle x \rangle^{-2} \vecr(t) \right \Vert_{L_x^\infty(\R) \times L_x^\infty(\R)} \lesssim \vert t \vert^{-\frac{3}{2}} \left \Vert \langle x \rangle^2 \vecf\, \right \Vert_{L_x^1(\R) \times L_x^1(\R)}.
\end{equation}
Thus, owing to the spatial localization of the quadratic nonlinearity, we have 
\begin{equation}\label{eqn: quadratic resonant term}
	 \calQ(\vecu_\mathrm{free}(t)) = \frac{c_+^2e^{2it}}{t} \mathcal{Q}_{1}(\vec{\Psi}) + \frac{c_+c_-}{t} \mathcal{Q}_{2}(\vec{\Psi}) + \frac{c_-^2e^{-2it}}{t} \mathcal{Q}_{3}(\vec{\Psi}) + \calO_{L^\infty}(\vert t \vert^{-2}),
\end{equation}
where
\begin{align}
	\mathcal{Q}_{1}(\vec{\Psi}) &=  \begin{bmatrix}
		-Q\Psi_2^2 - 2Q\Psi_1\Psi_2 \\ Q \Psi_1^2 + 2Q\Psi_1\Psi_2 
	\end{bmatrix},\\
\mathcal{Q}_{2}(\vec{\Psi}) &= \begin{bmatrix}
	- 2Q\Psi_1\Psi_2 - 2Q(\Psi_1^2+\Psi_2^2)\\
	2Q\Psi_1\Psi_2 + 2Q(\Psi_1^2+\Psi_2^2)
\end{bmatrix},\\
	\mathcal{Q}_{3}(\vec{\Psi}) &= -\sigma_1\calQ_1(\vec{\Psi}) = \begin{bmatrix}
	-Q\Psi_1^2 - 2Q\Psi_1\Psi_2 \\ Q \Psi_2^2 + 2Q\Psi_1\Psi_2 
\end{bmatrix}\label{calQ_3}.
\end{align}
Due to the critical $\calO(\vert t \vert^{-1})$ decay of the leading order terms on the right-hand side of \eqref{eqn: quadratic resonant term}, it is instructive to analyze the long-time behavior of small solutions to the inhomogeneous matrix Schr\"odinger equation with such a source term 
\begin{equation}\label{eqn: inhomg eqn}\left\{
	\begin{aligned}
		i\partial_t \vecu_{\mathrm{src}}- \calH_1 \vecu_{\mathrm{src}} &= P_\mathrm{s}\left(\frac{c_+^2e^{2it}}{t} \mathcal{Q}_{1}(\vec{\Psi}) + \frac{c_+c_-}{t} \mathcal{Q}_{2}(\vec{\Psi}) + \frac{c_-^2e^{-2it}}{t} \mathcal{Q}_{3}(\vec{\Psi}) \right), \quad t \geq 1,\\
		\vecu_{\mathrm{src}}(1) &= \vec0.
	\end{aligned}\right.
\end{equation}
To this end, it will be useful to exploit a special conjugation identity for the matrix Schr\"odinger operator $\calH_1$. It was recently pointed out by Martel, see \cite[Section~2.3]{21Martel}, that the matrix operator $\calH_1$ can be conjugated to the flat matrix Schr\"odinger operator $\calH_0$. By first conjugating $\calH_1$ with the unitary matrix $\calJ= \frac{1}{\sqrt{2}} \begin{bmatrix}
	1 & i \\ 1 & - i
\end{bmatrix}$,
we obtain the equivalent matrix Schr\"odinger operator
\begin{equation*}
	\calL_1 = -i \calJ^{-1}\calH_1 \calJ := \begin{bmatrix}
		0 & L_- \\ -L_+ & 0
	\end{bmatrix} = \calL_0 + \calW := \begin{bmatrix}
	0 & -\partial_x^2 + 1 \\ \partial_x^2 -1 & 0 
\end{bmatrix} + \begin{bmatrix}
0 & - 2\sech^2(x) \\ 6\sech^2(x) & 0
\end{bmatrix}.
\end{equation*}
Introducing the operator
\begin{equation}
	\calD := \begin{bmatrix}
		0 & (-\partial_x^2+1)S^2 \\ -S^2L_+ & 0
	\end{bmatrix}, \mwhere S := Q \cdot \partial_x \cdot Q^{-1} = \partial_x + \tanh(x),
\end{equation}
one has the conjugation identity (see also \cite[Section~3.4]{07ChangGustafsonNakanishiTsai})
\begin{equation}
	\calD \calL_1 = \calL_0 \calD.
\end{equation}
We then transfer the above identity to the matrix operator $\calH$ by setting $\widetilde\calD := \calJ \calD \calJ^{-1}$ to obtain the conjugation identity 
\begin{equation}\label{eqn: mirror conjugate identity}
	\widetilde\calD \calH_1 = \calH_0 \widetilde\calD.
\end{equation}
Moreover, it can be checked directly that $\widetilde\calD \vec\eta = 0$ for any generalized eigenfunction $\vec\eta \in \calN_\mathrm{g}(\calH_1)$, and this implies that $\widetilde\calD P_\mathrm{d} \equiv 0$, which is equivalent to saying that $\widetilde\calD = \widetilde\calD P_\mathrm{s}$. Hence, by applying the transformation $\widetilde\calD$ to the equation \eqref{eqn: inhomg eqn}, we obtain the transformed equation
\begin{equation}
	i\partial_t \vecv_{\mathrm{src}} - \calH_0 \vecv_{\mathrm{src}} = \widetilde\calD\left(\frac{c_+^2e^{2it}}{t} \mathcal{Q}_{1}(\vec{\Psi}) + \frac{c_+c_-}{t} \mathcal{Q}_{2}(\vec{\Psi}) + \frac{c_-^2e^{-2it}}{t} \mathcal{Q}_{3}(\vec{\Psi}) \right),
\end{equation}
where $\vecv_{\mathrm{src}} :=   \widetilde\calD \vecu_\mathrm{src}$ is the transformed variable. Note that the above equation features the flat operator $\calH_0$ on the left. The Duhamel formula for $\vecv_{\mathrm{src}}(t)$ at times $t \geq 1$ reads
\begin{equation}
	\vecv_{\mathrm{src}}(t) =  -i\int_1^t e^{-i(t-s)\calH_0}\widetilde\calD\left(\frac{c_+^2e^{2is}}{s} \mathcal{Q}_{1}(\vec{\Psi}) + \frac{c_+c_-}{s} \mathcal{Q}_{2}(\vec{\Psi}) + \frac{c_-^2e^{-2is}}{s} \mathcal{Q}_{3}(\vec{\Psi})\right) \,\ud s.
\end{equation}

The flat, self-adjoint, matrix operator $\calH_0$ has the benefit that the semigroup $e^{-it\calH_0}$ can be represented in terms of the standard Fourier transform by the formula
\begin{equation}
	\left(e^{-it\calH_0}\vecg\right) (x) = \frac{1}{\sqrt{2\pi}} \int_{\R} e^{-it(\xi^2+1)}\widehat{g_1}(\xi)e^{ix\xi}\,\ud \xi\,  \underline{e}_1 + \frac{1}{\sqrt{2\pi}} \int_{\R} e^{it(\xi^2+1)}\widehat{g_2}(\xi)e^{ix\xi}\,\ud\xi\,  \underline{e}_2,
\end{equation}
where $\vecg = (g_1,g_2)^\top$ and $\underline{e}_1,\underline{e}_2$ are the standard unit vectors in $\R^2$. The profile of $\vecv_{\mathrm{src}}(t)$ is given by
\begin{equation}
	\vecf_{\mathrm{src}}(t) := e^{it\calH_0}\vecv_{\mathrm{src}}(t).
\end{equation}
Setting $$\widetilde{\calD}\calQ_j(\vec{\Psi}) =: (G_{j,1},G_{j,2})^\top \mfor 1\leq j \leq 3,$$
we have for times $t \geq 1$ that 
\begin{equation}
	\begin{split}
		&\mathcal{F}[\vecf_\mathrm{src}(t)](\xi)\\
		&=c_+^2 \int_1^t  \frac{e^{is(\xi^2+3)}}{s} \widehat{G_{1,1}}(\xi) \,\ud s \,\underline{e}_1 + c_+c_- \int_1^t  \frac{e^{is(\xi^2+1)}}{s} \widehat{G_{2,1}}(\xi) \,\ud s \,\underline{e}_1 + c_-^2 \int_1^t  \frac{e^{is(\xi^2-1)}}{s} \widehat{G_{3,1}}(\xi)  \,\ud s \,\underline{e}_1\\
		&\quad +  c_+^2 \int_1^t \frac{e^{-is(\xi^2-1)}}{s} \widehat{G_{1,2}}(\xi)  \,\ud s \,\underline{e}_2  +  c_+c_-  \int_1^t  \frac{e^{-is(\xi^2+1)}}{s} \widehat{G_{2,2}}(\xi)   \,\ud s\, \underline{e}_2 + c_-^2  \int_1^t  \frac{e^{-is(\xi^2+3)}}{s} \widehat{G_{3,2}}(\xi)  \,\ud s \, \underline{e}_2.
	\end{split}
\end{equation}
The uniform-in-time boundedness in $L_\xi^\infty$ of the Fourier transform of the profile $\mathcal{F}[\vecf_\mathrm{src}(t)](\xi)$ is related to recovering the free decay rate for $\vecv_{\mathrm{src}}(t)$. However, in view of the critical decay of the integrand, this requires favorable time oscillations. Observe that the above terms with time phases $e^{\pm is(\xi^2+1)}$, $e^{\pm is(\xi^2+3)}$ are non-stationary for any $s \in \R$ which implies that they have a better decay rate using integration by parts in the variable $s$. On the other hand, the terms with the phases  $e^{\pm is(\xi^2-1)}$ are stationary at the points $\xi = \pm 1$. Thus, it is important to know if the Fourier coefficients $\widehat{G_{3,1}}(\pm1)$ and $\widehat{G_{1,2}}(\pm 1)$ vanish. Indeed, this is true due to the following lemma.
\begin{lemma}\label{lemma: null structure}	It holds that
	\begin{equation}\label{eqn: non-resonance in time}
		\widehat{G_{3,1}}(\pm1) = \widehat{G_{1,2}}(\pm1)= 0.	
	\end{equation}
	\begin{proof}
First, to ease notation, we write 
\begin{equation}
	\widetilde{\calD} = \frac{i}{2} \begin{bmatrix}
		(-D_1  - D_2)  &  (D_1-D_2)   \\  (-D_1 + D_2) & (D_1 + D_2)
	\end{bmatrix},
\end{equation}
where
\begin{equation}
	\begin{split}
		&D_1 := (-\partial_x^2+1)S^2 = (-\partial_x^2+1)(\partial_x + \tanh(x))(\partial_x + \tanh(x)), \\
		&D_2 := S^2L_+ =  (\partial_x + \tanh(x))(\partial_x + \tanh(x))(-\partial_x^2 - 6\sech^2(x) + 1).
	\end{split}
\end{equation}
Since $\sigma_1 \widetilde{\calD} = - \widetilde{\calD}\sigma_1$ and $\calQ_3(\vec\Psi)= -\sigma_1 \calQ_1(\vec\Psi)$ (c.f. \eqref{calQ_3}), it follows that $G_{3,1} \equiv G_{1,2}$ as functions. Note that
\begin{equation}
	G_{3,1} = \frac{i}{2}\left(   D_1(Q\Psi_1^2) + D_1(Q\Psi_2^2) + 2D_1(2Q\Psi_1\Psi_2) + D_2(Q\Psi_1^2) - D_2(Q\Psi_2^2)\right),
\end{equation}
where 
\begin{equation*}
	\begin{split}
		(Q\Psi_1^2)(x) &= \sqrt{2} \sech(x)\tanh^4(x),\\
		(Q\Psi_1\Psi_2)(x) &= -\sqrt{2}\sech^3(x)\tanh^2(x),\\
		(Q\Psi_2^2)(x) &= \sqrt{2}\sech^5(x).
	\end{split}
\end{equation*}
By using the trigonometric identity $\sech^2(x) + \tanh^2(x) =1$, we may simplify the expression for $G_{3,1}$ into
\begin{equation*}
	G_{3,1}(x) =  \frac{i \sqrt{2}}{2}\Big(D_1\big(\sech(x)-6\sech^3(x)+6\sech^5(x)\big) + D_2\big(\sech(x) - 2\sech^3(x)\big) \Big).
\end{equation*}
By patient direct computation, we find
\begin{equation}\label{D1 term1}
\begin{split}
	F_1(x) &:= D_1\big(\sech(x)-6\sech^3(x)+6\sech^5(x)\big) 	 \\ &=192\sech^3(x) - 3456\sech^5(x) + 9720 \sech^7(x) - 6720\sech^9(x)
\end{split}
	\end{equation}
and
\begin{equation}\label{D2 term1}
	\begin{split}
		F_2(x) :=D_2\big(\sech(x)-2\sech^3(x)\big) = 48\sech^3(x)  -264\sech^5(x) + 240\sech^7(x).
	\end{split}
\end{equation}
Moreover, using  the identities
\begin{equation}
	\begin{split}
		(\partial_x^2\sech)(x) &= \sech(x) - 2\sech^3(x),\\
		(\partial_x^4\sech)(x) &= \sech(x) - 20\sech^3(x)+24\sech^5(x),\\
		(\partial_x^6\sech)(x) &= \sech(x) -182\sech^3(x)+840\sech^5(x)-720\sech^7(x),\\
		(\partial_x^8\sech)(x) &= \sech(x) - 1640\sech^3(x) +23184\sech^5(x)-60480\sech^7(x) + 40320\sech^9(x),\\
	\end{split}
\end{equation}
we obtain
\begin{equation}\label{D1 term2}
	 F_1(x) = - \frac{1}{6}\left(-\partial_x^2 + 3\partial_x^4 - 3\partial_x^6 + \partial_x^8\right)\sech(x)= -\frac{1}{6}(-\partial_x^2+1)^3(-\partial_x^2)\sech(x),
\end{equation}
and
\begin{equation}\label{D2 term2}
	F_2(x) = \frac{1}{3}\left(-\partial_x^2 + 2\partial_x^4 - \partial_x^6\right)\sech(x) = \frac{1}{3}(-\partial_x^2+1)^2(-\partial_x^2)\sech(x).
\end{equation}
Thus, using the property $\calF[-\partial_x^2 f] = \xi^2 \calF[f](\xi)$ and the fact that
\begin{equation*}
	\widehat{\sech}(\xi)= \sqrt{\frac{\pi}{2}} \sech\left(\frac{\pi \xi}{2}\right),
\end{equation*}
we compute that
\begin{equation}
	\widehat{G_{3,1}}(\xi) = \frac{i\sqrt{2}}{{2}}\left(\widehat{F_1}(\xi)+\widehat{F_2}(\xi)\right)= -\frac{i \sqrt{\pi} }{12 } (\xi^2-1)\xi^2(\xi^2+1)^2\sech\left(\frac{\pi\xi}{2}\right),
\end{equation}
which implies \eqref{eqn: non-resonance in time} as claimed.
\end{proof}
\end{lemma}
\begin{remark}
	We determined the identities \eqref{D1 term1} -- \eqref{D2 term2} with the aid of the Wolfram Mathematica software.
\end{remark} 
The above lemma shows that the localized quadratic resonant terms are well-behaved for the nonlinear perturbation equation \eqref{eqn: linearized cubic NLS in complex-conjugate coordinates}. The presence of this null structure is potentially a key ingredient for a perturbative proof of the asymptotic stability of the solitary wave solutions to the 1D focusing cubic NLS. We end this subsection with the following closing remark.
\begin{remark}
The motivation for analyzing the quadratic nonlinearity in the perturbation equation \eqref{eqn: linearized cubic NLS in complex-conjugate coordinates} and for uncovering the null structure for the localized quadratic resonant terms in Lemma~\ref{lemma: null structure} is due to the recent work by L\"uhrmann-Schlag \cite{21LS}, where the authors investigate the asymptotic stability of kink solutions to the 1D sine-Gordon equation under odd perturbations. In \cite{21LS}, the authors employ a similar conjugation identity like the one we used in \eqref{eqn: mirror conjugate identity} to transform the scalar Schr\"odinger operator $H_1 := -\partial_x^2 -2\sech^2(x)$ to the flat operator $H_0 := -\partial_x^2$ for the perturbation equation. In fact, it is easy to check that one has the conjugation identity $SH_1 = H_0 S$, where $S = \partial_x + \tanh(x)$. Moreover, an analogue of Lemma~\ref{lemma: null structure} on the non-resonant property for the localized quadratic resonant terms in the perturbation equation for the sine-Gordon kink was first obtained in \cite[Remark 1.2]{23LLSS}. This remarkable null structure for the sine-Gordon model played a key role in the asymptotic stability proof in \cite{21LS}. In \cite{23LS}, the same authors obtained long-time decay estimates for even perturbation of  the soliton of the  1D focusing  cubic Klein-Gordon equation. The absence of the null structure in the nonlinearity of the perturbation equation in the focusing cubic Klein-Gordon model is a major  obstruction to full co-dimension one asymptotic stability result under even perturbations.
 
Our short discussion on the effects of the threshold resonance on the quadratic term for \eqref{eqn: linearized cubic NLS in complex-conjugate coordinates} suggests that the localized quadratic resonant terms are well-behaved for the perturbation equation in the 1D cubic NLS model. However, note that a full perturbative proof of the asymptotic stability problem for this model has to encompass the modulation theory associated to the moving solitary wave, and take into account the long-range (modified) scattering effects due to the non-localized cubic nonlinearities in the perturbation equation. We point out that Collot-Germain \cite{23CollotGermain} recently obtained general such asymptotic stability results for solitary waves for 1D nonlinear Schr\"odinger equations under the assumption that the linearized matrix Schr\"odinger operator does not exhibit threshold resonances. 
\end{remark}

\subsection{Organization of the article}
The remaining sections of this paper are devoted to the proof of Theorem~\ref{theorem: local decay estimate}. In Section~2, we state a few stationary phase lemmas, which will be heavily utilized in Sections~5 and 6, and we will also provide an analogue of Theorem~\ref{theorem: local decay estimate} for the free matrix operator $\calH_0$. In Section~3, we employ the symmetric  resolvent expansion following the framework in  \cite{06ErdoganSchlag}, and in Section~4, we carefully extract the leading operators for these resolvent expansions. A characterization of the threshold resonance is stated in Lemma~\ref{lemma: threshold resonance characterization} under the spectral assumptions (A1)--(A6). Then, in Section~5, we prove dispersive estimates for the evolution operator $e^{it\calH}$ in the low energy regime. The approach taken in Section~5 largely follows the techniques employed in \cite{22ErdoganGreen} for one-dimensional Dirac operators. In Section 6, we prove dispersive estimates for the remaining energy regimes and finish the proof of Theorem~\ref{theorem: local decay estimate}.

\subsection{Notation}
For any $\vecf = (f_1,f_2)^\top,\ \vecg = (g_1,g_2)^\top \in L^2(\R) \times L^2(\R)$, we use the inner product
\begin{equation}
	\langle \vecf,\vecg \rangle := \int_{\R}\vec{f}^* \vec{g}\ \ud x =   \int_{\R} \left(\bar{f}_1g_1 + \bar{f}_2 g_2\right) \ud x,\mwhere \vec{f}^* := (\bar{f}_1,\bar{f}_2).
\end{equation}
The Schwartz space is denoted by $\calS(\R)$ and we use the weighted $L^2$-spaces
\begin{equation}
	X_\sigma := \langle x \rangle^{-\sigma}L^2(\R) \times \langle x \rangle^{-\sigma}L^2(\R), \quad \Vert \vecf \Vert_{X_\sigma} := \Vert \jap{x}^\sigma \vecf \Vert_{L^2(\R)\times L^2(\R)}, \mwhere \sigma \in \R.
\end{equation}
Note that for any $\alpha > \beta > 0$, one has the continuous inclusions
\begin{equation}
	X_\alpha \subset X_\beta \subset X_0 = L^2(\R)\times L^2(\R) \subset X_{-\beta} \subset X_{-\alpha},
\end{equation}
and the duality $X_\alpha^* = X_{-\alpha}$. 
Our convention for the Fourier transform is
\begin{equation*}
	\calF[f](\xi) = \hatf(\xi) =  \frac{1}{\sqrt{2\pi}} \int_\R e^{-ix\xi}f(x)\ud x, \quad 	\calF^{-1}[f](x) = \checkf(x) = \frac{1}{\sqrt{2\pi}} \int_\R e^{ix\xi}f(\xi) \ud \xi.
\end{equation*}
We denote by $C>0$ an absolute constant whose value is allowed to change from line to line. In order to indicate that the constant depends on a parameter, say $\theta$, we will use the notation $C_\theta$ or $C(\theta)$. For non-negative $X$, $Y$ we write $X \lesssim Y$ if $X \leq CY$. We use the Japanese bracket notation $\langle x \rangle = (1+x^2)^{\frac{1}{2}}$ for $x \in \R$. The standard tensors on $\R^2$ are denoted by
\begin{equation}\label{eqn: def e1e2}
	\begin{split}
		\underline{e}_{1} = \begin{bmatrix}1 \\ 0\end{bmatrix}, \quad \underline{e}_2 = \begin{bmatrix}0 \\ 1 \end{bmatrix},\quad  \underline{e}_{11} = \underline{e}_1\underline{e}_1^\top =\begin{bmatrix}
			1 & 0 \\ 0 & 0
		\end{bmatrix},\quad \underline{e}_{22} = \underline{e}_2\underline{e}_2^\top= \begin{bmatrix}
			0 & 0 \\ 0 & 1
		\end{bmatrix}.		
	\end{split}
\end{equation}
\textit{Acknowledgments.} The author would like to thank his Ph.D. advisor Jonas L\"uhrmann for suggesting the problem and patiently checking the manuscript. The author is grateful to Andrew Comech, Wilhelm Schlag, Gigliola Staffilani, and Ebru Toprak for helpful discussions.

\section{Free matrix Schr\"odinger estimates}
In this section, we derive dispersive estimates for the free evolution semigroup $e^{it\calH_0}$. We recall that the free matrix Sch\"odinger operator
\begin{equation*}
	\calH_0 = \begin{bmatrix}
		-\partial_x^2 + \mu & 0 \\ 0 & \partial_x^2 - \mu
	\end{bmatrix},
\end{equation*}
has a purely continuous spectrum
$$\spec(\calH_0) = \sigma_{\mathrm{ac}}(\calH_0) = (-\infty,-\mu] \cup [\mu,\infty),$$
and the resolvent operator of $\calH_0$ is given by 
\begin{equation}\label{eqn: R_0 1}
	(\calH_0 - \lambda)^{-1} = \begin{bmatrix}
		R_0(\lambda - \mu) & 0 \\ 0 & -R_0(-\lambda-\mu)
	\end{bmatrix}, \quad \lambda \in \C \setminus (-\infty,-\mu] \cup [\mu,\infty),
\end{equation}
where $R_0$ is the resolvent operator for the one-dimensional Laplacian, with an integral kernel given by
\begin{equation}\label{eqn: R_0 2}
	R_0(\zeta^2)(x,y) :=  (-\partial^2 - \zeta^2)^{-1}(x,y) = \frac{-e^{i \zeta \vert x - y\vert }}{2i \zeta}, \quad \zeta \in \C_+,
\end{equation}
where $\C_+$ is the upper half-plane. We obtain from the scalar resolvent theory due to Agmon \cite{agmon1975spectral} that the limiting resolvent operators
\begin{equation*}
	\big(\calH_0 - (\lambda \pm i0) \big)^{-1} = \lim_{\eps \downarrow  0}\, \big(\calH_0 - (\lambda \pm i\eps)\big)^{-1}, \quad \lambda \in (-\infty,-\mu) \cup (\mu,\infty),
\end{equation*}
are well defined as operators from $X_\sigma  \to X_{-\sigma} $ for any $\sigma > \frac{1}{2}$. Here, the matrix operator $\calH_0$ is self-adjoint and Stone's formula applies: 
\begin{equation}\label{eqn: free evolution1}
	e^{it\calH_0} = \frac{1}{2 \pi i}\int_{\vert \lambda \vert \geq \mu} e^{it\lambda}\left[\big(\calH_0 - (\lambda + i0)\big)^{-1} - \big(\calH_0 - (\lambda - i0)\big)^{-1}\right]\,\ud \lambda.
\end{equation}
Let us focus on the spectrum on the positive semi-axis $[\mu,\infty)$, as the negative part can be treated using the symmetric properties of $\calH$ (c.f. Remark~\ref{remark: Ps^-}). By invoking the change of variables $\lambda \mapsto \lambda = \mu + z^2$ with $0< z <\infty$, the kernel of $e^{it\calH_0}P_{\mathrm{s}}^+$ is then given by
\begin{equation*}
	e^{it\calH_0}P_{\mathrm{s}}^+(x,y) = \frac{e^{it\mu}}{\pi i} \int_0^\infty e^{itz^2}z \left[\big(\calH_0 - (\mu+z^2+ i0)\big)^{-1} - \big(\calH_0 - (\mu+z^2- i0)\big)^{-1} \right](x,y)\,\ud z.
\end{equation*}
Here, the notation $P_\mathrm{s}^+$ means that we restrict the free evolution $e^{it\calH_0}$ to  the positive semi-axis in the integral representation \eqref{eqn: free evolution1}. By \eqref{eqn: R_0 1} and \eqref{eqn: R_0 2}, we have
\begin{equation}\label{eqn: def resolvent H_0}
	\big(\calH_0 - (\mu+z^2\pm i0)\big)^{-1}(x,y) = \begin{bmatrix}
		\frac{\pm ie^{\pm i z \vert x - y \vert}}{2 z} & 0 \\ 0 & -\frac{e^{-\sqrt{z^2+2\mu}\vert x-y\vert}}{2\sqrt{z^2+2\mu}}
	\end{bmatrix}, \quad 0 < z < \infty,
\end{equation}
and thus,
\begin{equation}\label{eqn: exp itH_0 1}
	e^{it \calH_0}P_{\mathrm{s}}^+(x,y) = \frac{e^{it \mu}}{2\pi}\int_\R e^{it z^2} e^{i z \vert x - y \vert}\underline{e}_{11}\, \ud z.
\end{equation}
Note that the above integral is to be understood in the principal value sense, due to the pole in \eqref{eqn: def resolvent H_0}.  To this end, we recall the following standard stationary phase results. The first lemma is a direct consequence of the classic van der Corput lemma.
\begin{lemma}\label{lemma: van der corput estimate}
Let $r \in \R$, and let $\psi(z)$ be a compactly supported smooth function. Then for any $\vert t\vert > 0$,
	\begin{equation}\label{eqn: VDC1}
		\left \vert \int_{\R} e^{itz^2 + i z r}\psi (z)\, \ud z \right \vert \leq C    \vert t \vert^{-\frac{1}{2}}\Vert \partial_{z} \psi\Vert_{L_z^1(\R)}.
	\end{equation}
	Moreover, if $\psi(z)$ is supported away from zero, then for all $\vert t \vert > 0 $, 
	\begin{equation}\label{eqn: VDC2}
		\left \vert \int_{\R} e^{itz^2 + i z r}\psi (z)\  \ud z \right \vert \leq C    \vert t  \vert^{-\frac{3}{2}} \left\Vert [\partial_z^2 + i r \partial_{z}]\big(\tfrac{\psi }{z}\big) \right\Vert_{L_z^1(\R)}.
	\end{equation}
	\begin{proof}
		The bound \eqref{eqn: VDC1} follows from the van der Corput lemma (see e.g. \cite[VIII Proposition~2]{stein1993harmonic}) by observing that the phase $\phi(z) = z^2 + \frac{zr}{t}$ satisfies $\vert \partial_z^2 \phi(z)\vert = 2>0$.  The last bound follows by first integrating by parts 
		\begin{equation*}
			\int_{\R} e^{itz^2}e^{i z r}\psi (z) \,\ud z  = -\frac{1}{2 i t}\int_\R e^{itz^2}\partial_{z}\left [e^{iz r}\frac{\psi(z)}{z}\right] \ud z  =-\frac{1}{2 i t}\int_\R e^{itz^2 + iz r}[ir + \partial_{z}]\left[\frac{\psi (z)}{z}\right] \ud z, 
		\end{equation*}
		and then invoking the van der Corput lemma.
	\end{proof}
\end{lemma}
We will also need the following sharper stationary phase lemma, which may be found in many text on oscillatory integrals with a Fresnel phase.
\begin{lemma}\label{lemma: stationary phase fresnel} 
	Let $\chi(z)$ be a smooth, non-negative, even cut-off function such that $\chi(z) = 1$ for $z \in [-1,1]$ and $\chi(z) = 0$ for $\vert z \vert \geq 2$. For $r, t \in \R $, define
	\begin{equation}
		G_t(r) := 	\int_\R e^{itz^2+izr}\chi(z^2)\ \ud z. 
	\end{equation}
	Then there exists $C = C\big(\Vert \chi(z^2)\Vert_{W^{4,1}(\R)}\big) >0$ such that for any $r \in \R$ and for any $ \vert t \vert > 0 $,
	\begin{equation}\label{eqn: G_t(r) error1}
		\left \vert G_t(r) - \frac{\sqrt{\pi}}{{\sqrt{-it}}}  e^{-i\frac{r^2}{4t}} \right \vert \leq C \vert t \vert^{-\frac{3}{2}}\langle r \rangle.
	\end{equation}
	Moreover, if $r_1, r_2 \geq 0$, then
	\begin{equation}\label{eqn: G_t(r) error2}
		\left \vert G_t(r_1+r_2) - \frac{\sqrt{\pi}}{{\sqrt{-it}}} e^{-i\frac{r_1^2}{4t}}e^{-i\frac{r_2^2}{4t}} \right \vert \leq C \vert t \vert^{-\frac{3}{2}}\langle r_1 \rangle \langle r_2 \rangle.
	\end{equation}
\end{lemma}
\begin{proof}
	First, the phase $\phi(z) := z^2 + \frac{zr}{t}$ has a critical point at $z_* = -\frac{r}{2t} \in \R$ with $\phi''(z) = 2 > 0$. We use Taylor expansion of $\phi(z)$ and shift the integral by the change of variables $z \mapsto z + z^*$ to obtain
	\begin{equation}\label{eqnproof: G_t(r) 1}
		G_t(r) = \int_{\R} e^{it\phi(z)}\chi(z^2)\,\ud z = \int_R e^{it\phi(z^*)+ \phi''(z_*)(z-z_*)^2}\chi(z^2) \,\ud z = e^{-i\frac{r^2}{4t}}\int_\R e^{itz^2}\chi\big((z+z_*)^2\big) \,\ud z.
	\end{equation}
	Using the Fourier transform of the free Schr\"odinger group and the Plancherel's identity, we have
	\begin{equation*}
		\begin{split}
		\int_\R e^{itz^2}\chi\big((z+z_*)^2\big) \,\ud z 	&= \frac{1}{\sqrt{-2 i t}}\int_\R  e^{-i\frac{\xi^2}{4t}} \calF_{z \to \xi}\left[\chi\big((z+z_*)^2\big)\right](\xi) \,\ud \xi \\
		&=	\frac{1}{\sqrt{-2 i t}}\int_\R \calF_{z \to \xi}\left[\chi\big((z+z_*)^2\big)\right](\xi) \,\ud\xi \\
		&\qquad + \frac{1}{\sqrt{-2 i t}}\int_\R  \Big(e^{-i\frac{\xi^2}{4t}}-1\Big) \calF_{z \to \xi}\left[\chi\big((z+z_*)^2\big)\right](\xi) \,\ud\xi \\
		&= \frac{\sqrt{2\pi}}{\sqrt{-2it}}\chi(z_*^2) + 		\frac{1}{\sqrt{-2 i t}}\int_\R  \Big(e^{-i\frac{\xi^2}{4t}}-1\Big) e^{iz_*\xi} \calF \left[\chi\big((z+z_*)^2\big)\right](\xi) \,\ud\xi.
		\end{split}
	\end{equation*}
	Using the bound $\vert e^{i\frac{\xi^2}{4t}}-1\vert \leq C\vert t \vert^{-1}\xi^2 $ and the H\"older's inequality, we bound the remainder term by
	\begin{equation*}
		\begin{split}
		\left \vert \frac{1}{\sqrt{-2 i t}}\int_\R  \Big(e^{-i\frac{\xi^2}{4t}}-1\Big) e^{iz_*\xi} \calF  \left[\chi\big((z+z_*)^2\big)\right](\xi) \,\ud\xi \right \vert   &\leq C \vert t\vert^{-\frac{3}{2}} \int_\R \vert \xi^2 \calF[\chi(z^2)](\xi) \vert \,\ud\xi \\
		& \leq C \vert t \vert^{-\frac{3}{2}}\Vert \chi(z^2)\Vert_{W^{4,1}(\R)} \leq C \vert t \vert^{-\frac{3}{2}}.
		\end{split}
	\end{equation*}
	Next, we use the fact that $\vert 1 - \chi(z^2) \vert \leq C \vert z \vert$ for all $z \in \R$ and  for some $C>0$ large enough so that 
	\begin{equation}\label{eqnproof: G_t(r) 4}
		\vert 1-\chi(z_*^2)\vert \leq C \vert z_* \vert \leq C \vert t \vert^{-1}\langle r \rangle.
	\end{equation}
	Then \eqref{eqn: G_t(r) error1} follows \eqref{eqnproof: G_t(r) 1}--\eqref{eqnproof: G_t(r) 4}. Finally, we use the estimate \eqref{eqn: G_t(r) error1} to obtain
	\begin{equation*}
		\left \vert G_t(r_1+r_2)  - \frac{\sqrt{2\pi}}{{\sqrt{-2it}}} e^{-i\frac{(r_1-r_2)^2}{4t}} \right \vert \leq C \vert t \vert^{-\frac{3}{2}}\langle r_1 - r_2 \rangle \leq C \vert t \vert^{-1} \langle r_1 \rangle \langle r_2 \rangle.
	\end{equation*}
	Thus, by the triangle inequality and the bound
	\begin{equation*}
		\left \vert e^{-i\frac{(r_1-r_2)^2}{4t}} - e^{-i\frac{r_1^2}{4t}}e^{-i\frac{r_2^2}{4t}} \right \vert = \left\vert e^{-i\frac{r_1^2}{4t}}e^{-i\frac{r_2^2}{4t}} \right\vert \left\vert e^{i\frac{r_1 r_2}{2t}} - 1 \right\vert \leq C \vert t \vert^{-1} \langle r_1 \rangle \langle r_2 \rangle,
	\end{equation*}
	we conclude \eqref{eqn: G_t(r) error2}.
\end{proof}
Next, we prove the analogue of Theorem~\ref{theorem: local decay estimate} for the free evolution. We emphasize that the free matrix Schr\"odinger operator $\calH_0$ has threshold resonances $\calH_0 \underline{e}_{1} = \mu \underline{e}_1$ and $\calH_0 \underline{e}_2 = -\mu \underline{e}_2$.
\begin{proposition}\label{prop: free estimate}
	For any $\vec{u} = (u_1,u_2) \in \calS(\R) \times \calS(\R)$ and for any $\vert t\vert  \geq 1$, we have
	\begin{equation}\label{eqn: free unweighted estimate}
		\left\Vert e^{it\calH_0} P_{\mathrm{s}}^+ \vec{u}\, \right\Vert_{L_x^\infty \times L_x^\infty} \lesssim \vert t \vert^{-\frac{1}{2}}\Vert \vec{u}\, \Vert_{L_x^1 \times L_x^1},
	\end{equation}	
	and
	\begin{equation}\label{eqn: free weighted estimate}
		\left\Vert \langle x \rangle^{-1} \left( e^{it\calH_0}P_{\mathrm{s}}^+   - F_t^0\right)\vec{u}\,  \right\Vert_{L_x^\infty \times L_x^\infty`} \lesssim \vert t \vert^{-\frac{3}{2}}\Vert \langle x \rangle \vec{u} \, \Vert_{L_x^1 \times L_x^1},
	\end{equation}
	where
	\begin{equation}\label{eqn: formula F_t0}
		F_t^0(x,y) := \frac{e^{it \mu} }{\sqrt{-4\pi i t}}e^{-i\frac{x^2}{4t}}\underline{e}_1e^{-i\frac{y^2}{4t}}\underline{e}_1^\top.
	\end{equation}
\end{proposition}
\begin{proof} 
We first begin by splitting the evolution operator into low and high energy parts\footnote{Symbols like $\chi(\calH_0 - \mu I)$ are only used in a formal way to represent the cut-off $\chi(z^2)$ in the $z$-integrals, where they arise.}:
\begin{equation}\label{eqn: free evolution high low split}
	\begin{split}
		e^{it\calH_0}P_{\mathrm{s}}^+(x,y) &= e^{it\calH_0}\chi(\calH_0 - \mu I)P_{\mathrm{s}}^+(x,y) 		+e^{it\calH_0}(1-\chi(\calH_0 - \mu I))P_{\mathrm{s}}^+(x,y)\\
		& = \frac{e^{it\mu}}{2\pi}\int_\R e^{itz^2+iz\vert x - y \vert} \chi(z^2) \,\ud z \underline{e}_{11} + \frac{e^{it\mu}}{2\pi}\int_\R e^{itz^2+iz\vert x - y \vert} (1-\chi(z^2)) \,\ud z \underline{e}_{11},
	\end{split}
\end{equation}
where $\chi(z)$ is a standard smooth, even, non-negative cut-off function satisfying $\chi(z) = 1$ for $\vert z \vert \leq 1$ and $\chi(z) =0 $ for $ \vert z \vert \geq 2$. 

In the high energy part in \eqref{eqn: free evolution high low split}, following the ideas from \cite{04GoldbergSchlag} \cite{07Goldberg}, we prove the estimate
\begin{equation}\label{eqn: free high energy estimate}
	\left \vert \int_\R e^{itz^2+iz\vert x - y \vert} (1-\chi(z^2))dz \right\vert \lesssim \min\{\vert t \vert^{-\frac{1}{2}},\vert t \vert^{-\frac{3}{2}}\langle x \rangle \langle y \rangle\}.
\end{equation}
For a more rigorous treatment, we instead use a truncated cutoff $\chi_L(z) = (1-\chi(z^2))\chi(z/L)$, where $L \geq 1$, and we prove the uniform estimate
\begin{equation}\label{eqn: free evolution high energy bound}
 \sup_{L \geq 1} \left \vert \int_\R  e^{itz^2 + iz\vert x - y \vert} \chi_L(z) \, \ud z \right\vert \leq C\min\{\vert t \vert^{-\frac{1}{2}},\vert t \vert^{-\frac{3}{2}}\langle x \rangle \langle y \rangle\},
\end{equation}
with a constant $C>0$ independent of $L$. This estimate will imply \eqref{eqn: free high energy estimate}. 
Indeed for any $\vert t \vert >0$, by the Plancherel's identity, we have 
\begin{equation*}
		\sup_{a \in \R}\left\vert \int_\R  e^{itz^2+iaz} \chi_L(z) \,\ud z \right \vert = \sup_{a \in \R} \left \vert \int_{\R} \calF^{-1}{[e^{itz^2+iaz}]} (\xi ) \calF{[\chi_L(z)]}(\xi)  \,\ud \xi \right \vert \leq C\vert t \vert^{-\frac{1}{2}} \Vert \calF{[\chi_L]} \Vert_{L_\xi^1(\R)}.
\end{equation*}
Here, we use that the Fourier transform of the tempered distribution $e^{itz^2+iaz}$ has $\vert t \vert^{-\frac{1}{2}}$ decay. Using the definition of $\chi_L$, the scaling properties of the Fourier transform, and Young's convolution inequality, we obtain
\begin{equation}
	\begin{split}
		\Vert \calF{[\chi_L]} \Vert_{L_\xi^1(\R)} & \leq \Vert \calF[\chi(z/L)]\Vert_{L_\xi^1(\R)} + \Vert \calF[\chi(z/L)]\Vert_{L_\xi^1(\R)}\Vert \calF[\chi(z^2)]\Vert_{L_\xi^1(\R)}\\
		&\leq C \Vert L \calF[\chi](L\xi)\Vert_{L_\xi^1(\R)} =  C \Vert \calF[\chi](\xi)\Vert_{L_\xi^1(\R)} \leq C \Vert \chi \Vert_{W^{2,1}(\R)}\lesssim 1.
	\end{split}
\end{equation}
For the high-energy weighted dispersive estimate, we use integration by parts to find that 
\begin{equation*}
	\begin{split}
		&\left \vert \int_\R e^{itz^2}e^{iz \vert x - y \vert)}\chi_L(z) \,\ud z \right \vert \leq C \vert t \vert^{-1} \left \vert \int_{\R} e^{itz^2}\partial_z\left( e^{iz\vert x - y \vert} z^{-1}\chi_L(z) \right) \,\ud z\right \vert.
	\end{split}
\end{equation*}
When the derivative falls onto $e^{iz \vert x - y \vert}$, the weights $\langle x \rangle\langle y \rangle$ appear, whereas the term $z^{-1}\chi_L(z)$ is smooth since $\chi_L$ is compactly supported away from the interval $[-1,1]$. By following the previous argument, we conclude the $\calO(\vert t \vert^{-\frac{3}{2}}\langle x \rangle \langle y \rangle)$ bound for \eqref{eqn: free evolution high energy bound} in the high-energy regime. 

Next we turn to the low-energy estimates. For the low-energy unweighted estimate, we employ Lemma~\ref{lemma: van der corput estimate} to obtain
\begin{equation}\label{eqn: free low energy unweighted estimate}
	\left \vert \int_{\R} e^{itz^2+iz\vert x -y \vert}\chi(z^2) \,\ud z \right \vert \leq C \vert t \vert^{-\frac{1}{2}} \Vert \partial_z \chi(z^2) \Vert_{L^1(\R)} \leq C \vert t \vert^{-\frac{1}{2}}.
\end{equation}
On the other hand, for the low-energy weighted estimate, we observe that by  Lemma~\ref{lemma: stationary phase fresnel}, 
\begin{equation*}
	\left \vert	\int_{\R} e^{itz^2+iz\vert x -y \vert}\chi(z^2) \,\ud z - \frac{\sqrt{2\pi}}{\sqrt{-2it}} e^{-i\frac{x^2}{4t}}e^{-i\frac{y^2}{4t}}  \right \vert \leq C \vert t \vert^{-\frac{3}{2}} \langle x \rangle \langle y \rangle.
\end{equation*}
Hence, using that $\underline{e}_{11} = \underline{e}_1\underline{e}_1^\top$, we arrive at the kernel estimate
\begin{equation}\label{eqn: free low energy weighted estimate}
	\left \vert e^{it\calH_0}\chi(\calH_0 - \mu)P_\mathrm{s}^+(x,y) - F_t^0(x,y) \right \vert \leq C \vert t \vert^{-\frac{3}{2}}\langle x \rangle \langle y \rangle,
\end{equation}
where $F_t^0$ is given by \eqref{eqn: formula F_t0}. Thus, by combining the high energy bounds \eqref{eqn: free high energy estimate} and the low energy bounds \eqref{eqn: free low energy unweighted estimate} - \eqref{eqn: free low energy weighted estimate}, we conclude the dispersive estimates  \eqref{eqn: free unweighted estimate} and \eqref{eqn: free weighted estimate}.
\end{proof} 
\section{Symmetric resolvent identity}
By assumption (A1), we can factorize the matrix potential
\begin{equation}\label{eqn: def of V,v_1,v_2}
	\calV = -\sigma_3 v v = v_1 v_2,
\end{equation}
with
\begin{equation*}
	v_1 = -\sigma_3 v := \begin{bmatrix}
		-a & -b \\ b & a
	\end{bmatrix}\mand 
	 v_2 = v := \begin{bmatrix}
		a & b \\ b & a
	\end{bmatrix},
\end{equation*}
where 
\begin{equation*}
	a := \frac{1}{2}\Big(\sqrt{V_1+V_2} + \sqrt{V_1 - V_2}\Big) \mand 	b := \frac{1}{2}\Big(\sqrt{V_1+V_2} - \sqrt{V_1 - V_2}\Big).
\end{equation*}
It will be helpful in later sections to keep in mind that 
\begin{equation}\label{identity V_1 V_2}
	V_1 = a^2 + b^2, \quad V_2 = 2ab.
\end{equation}
We denote the resolvent of $\calH = \calH_0 + \calV$ by $(\calH-z)^{-1}$ for $z \in \rho(\calH)$. The resolvent identity states that 
\begin{equation*}
 	(\calH - z)^{-1} = \big(I+(\calH_0 -z)^{-1}\calV\big)^{-1}(\calH_0-z)^{-1}, \quad \forall z \in \rho(\calH_0) \cap \rho(\calH).
\end{equation*}
This identity was used in \cite{06ErdoganSchlag} to establish that there is a limiting absorption principle for the resolvent of $\calH$ on the semi-axes $(-\infty,-\mu)\cup (\mu,\infty)$ in the  weighted $L^2$-spaces $X_{\sigma} \rightarrow  X_{-\sigma}$, $\sigma>\frac{1}{2}$. Note that the lemma below applies in any spatial dimension. 
\begin{lemma}(\cite[Lemma 4-Corollary 6]{06ErdoganSchlag}, see also the proof in \cite[Lemma 6.8]{06KriegerSchlag})\label{lemma: LAP} Suppose assumptions (A1) -- (A4) hold. Then, the following holds.
	\
	\begin{enumerate}
		\item For $\sigma > \frac{1}{2}$, and $\vert \lambda \vert > \mu$, the operator
		\begin{equation}
			\big(\calH_0 - (\lambda \pm i0)\big)^{-1}\calV: X_{-\sigma} \to  X_{-\sigma}
		\end{equation}
		is compact and $I + \big(\calH_0 - (\lambda \pm i0)\big)^{-1}\calV$ is boundedly invertible on $X_{-\sigma}$. 
		\item For $\sigma>\frac{1}{2}$ and $\lambda_0>\mu$ arbitrary, we have
		\begin{equation}
			\sup_{\vert \lambda \vert \geq \lambda_0, \eps > 0} \vert \lambda \vert^{\frac{1}{2}} \left\Vert \big(\calH - (\lambda \pm i\eps)\big)^{-1} \right\Vert_{X_\sigma \to X_{-\sigma}}<\infty. 
		\end{equation}
		\item For $\vert \lambda \vert > \mu$, define
		\begin{equation}\label{eqn: resolvent identity 1}
			\big(\calH - (\lambda \pm i0) \big)^{-1} := \Big(I+ \big(\calH_0 - (\lambda \pm i0)\big)^{-1}\calV\Big)^{-1} \big(\calH_0 -(\lambda \pm i0) \big)^{-1}.
		\end{equation}
		Then, as $\eps \searrow 0$,
		\begin{equation}
			\left\Vert \big(\calH - (\lambda\pm i\eps) \big)^{-1} - \big( \calH - (\lambda \pm i0) \big)^{-1} \right\Vert_{X_\sigma \to X_{-\sigma}} \longrightarrow 0
		\end{equation}
		for any $\sigma > \frac{1}{2}$.
	\end{enumerate}
\end{lemma}
We recall the following spectral representation of $e^{it\calH}$ from \cite{06ErdoganSchlag}.
\begin{lemma} \label{lemma: spectral representation} (\cite[Lemma 12]{06ErdoganSchlag})
	Under assumptions (A1) -- (A6), there is the representation
	\begin{equation}\label{eqn: spectral representation}
		e^{it\calH}= \frac{1}{2\pi i}\int_{\vert \lambda \vert \geq \mu} e^{it\lambda}\left[\big(\calH - (\lambda+i0)\big)^{-1} - \big(\calH -(\lambda - i0)\big)^{-1}\right] \,\ud \lambda + \sum_{j} e^{it\calH}P_{z_j},
	\end{equation}
	where the sum runs over the entire discrete spectrum and $P_{z_j}$ is the Riesz projection corresponding to the eigenvalue $z_j$. The formula \eqref{eqn: spectral representation} and the convergence of the integral are to be understood in the sense that if $\phi,\psi\in [W^{2,2}(\R) \times W^{2,2}(\R)] \cap [\langle x\rangle^{-1-} L^2(\R) \times \langle x\rangle^{-1-} L^2(\R)]$,  then
	\begin{equation}
		\begin{split}
			\langle e^{it\calH}\phi,\psi\rangle &= \lim_{R \to \infty} \frac{1}{2\pi i} \int_{R \geq \vert \lambda \vert \geq \mu} e^{it\lambda}\left\langle \left[\big(\calH-(\lambda+i0)\big)^{-1}-\big(\calH - (\lambda - i0)\big)^{-1}\right]\phi,\psi\right\rangle \,\ud \lambda \\
			&\quad + \sum_{j} \langle e^{it\calH}P_{z_j}\phi,\psi \rangle	,	
		\end{split}
	\end{equation}
	for all $t \in \R$. 
\end{lemma}
We write $P_{\mathrm{s}} = P_{\mathrm{s}}^+ + P_{\mathrm{s}}^-$, where the signs $\pm$ refer to the positive and negative halves of the essential spectrum $(-\infty,-\mu]\cup [\mu,\infty)$. In the following sections, we will focus on the analysis on the positive semi-axis part of the  essential spectrum. We can treat the negative semi-axis of the essential spectrum by taking advantage of the symmetry properties of $\calH$, see Remark~\ref{remark: Ps^-} below. In view of the spectral representation of $e^{it\calH}$ from Lemma~\ref{lemma: spectral representation}, we use the change of variables $\lambda \mapsto \lambda = \mu+z^2$  with $0<z<\infty$ to write 
\begin{equation*}
	e^{it\calH}P_{\mathrm{s}}^+ = \frac{e^{it\mu}}{\pi i} \int_0^\infty e^{itz^2} z \left[\big(\calH - (\mu + z^2 + i0))^{-1} - (\calH - (\mu + z^2 - i0)\big)^{-1}\right] \,\ud z.
\end{equation*}

For the upcoming dispersive estimates, it is convenient to first open up the domain of integration for the above integral to the entire real line by means of analytic continuation for the perturbed resolvent. Following the framework of Section 5 in \cite{06ErdoganSchlag}, we introduce the operator
\begin{equation}
	\begin{split}
	&\calR(z) := (\calH - (\mu + z^2 + i0))^{-1}, \quad \text{for $z>0$},\\		
	&\calR(z) := (\calH - (\mu + z^2 - i0))^{-1} = \overline{(\calH - (\mu + z^2 + i0))^{-1}}, \quad \text{for $z<0$},
	\end{split}
\end{equation}
so that
\begin{equation}\label{eqn: stone's for H}
e^{it\calH}P_{\mathrm{s}}^+ = \frac{e^{it\mu}}{\pi i}\int_\R e^{itz^2} z \calR(z) \,\ud z.
\end{equation}
Here, the integral should be understood in the principal value sense due to the pole associated with the resolvent $\calR(z)$ at the origin. We also set
\begin{equation}
	\begin{split}
		&\calR_0(z) := (\calH_0 - (\mu + z^2 + i0))^{-1}, \quad \text{for $z>0$},\\		
		&\calR_0(z) := \overline{(\calH_0 - (\mu + z^2 + i0))^{-1}}, \quad \text{for $z<0$}.
	\end{split}	
\end{equation}
In particular, with this definition, we have by \eqref{eqn: def resolvent H_0} for all $z \in \R \setminus \{0\}$ that
\begin{equation}\label{eqn: free resolvent}
	\calR_0(z)(x,y) = (\calH_0 - (\mu + z^2 + i0))^{-1}(x,y) = \begin{bmatrix}
		\frac{ ie^{ i z \vert x - y \vert}}{2 z} & 0 \\ 0 & -\frac{e^{-\sqrt{z^2+2\mu}\vert x-y\vert}}{2\sqrt{z^2+2\mu}}
	\end{bmatrix}.
\end{equation}

As in \cite{06ErdoganSchlag}, we employ the symmetric resolvent identity 
\begin{equation}\label{eqn: symmetric resolvent identity}
	\calR(z) = \calR_0(z) - \calR_0(z)v_1(M(z))^{-1}v_2\calR_0(z),
\end{equation}
where
\begin{equation}\label{eqn: def of M}
	M(z) = I + v_2 \calR_0(z)v_1, \quad z \in \R \setminus \{0\}.
\end{equation}
By inserting the above identity, one checks that 
\begin{equation}\label{eqn: stone's formula for P_s^+}
	e^{it\calH}P_{\mathrm{s}}^+ = \frac{e^{it\mu}}{\pi i} \int_{\R} e^{itz^2}z\mathcal{R}_0(z) \,\ud z - \frac{e^{it\mu}}{\pi i} \int_{\R} e^{itz^2}z\mathcal{R}_0(z)v_1 \big(M(z)\big)^{-1}v_2\mathcal{R}_0(z) \,\ud z. 
\end{equation}
In the next section, we will investigate the invertibility of the matrix operator $M(z)$ near the origin. We give the following remark for the evolution operator in the negative part of the essential spectrum.

\begin{remark} \label{remark: Ps^-}
Using the identities
	\begin{equation}\label{eqn: H = -s_1HS_1}
		\calH = -\sigma_1 \calH \sigma_1, \quad \calV= -\sigma_1 \calV \sigma_1,
	\end{equation}
	we infer that
	\begin{equation}\label{eqn: P_s^-}
		e^{it\calH}P_{\mathrm{s}}^{-} = \sigma_1 e^{-it\calH}P_{\mathrm{s}}^{+}\sigma_1.
	\end{equation}
	Furthermore, since these identities also hold for $\calH_0$, the analogue of Proposition~\ref{prop: free estimate} for the weighted estimate of the free evolution $e^{it\calH_0}P_{\mathrm{s}}^{-}$ is given by
	\begin{equation}
		\left \Vert \langle x \rangle^{-1} \left( e^{it\calH_0}P_{\mathrm{s}}^-   - \widetilde{F}_t^0\right)\vecu \,\right\Vert_{L_x^\infty} \leq C \vert t \vert^{-\frac{3}{2}} \left \Vert \langle x \rangle \vecu \,\right\Vert_{L_x^1}, \quad \vert t \vert \geq 1,
	\end{equation}
	where
	\begin{equation}
		\widetilde{F}_t^0(x,y) := \frac{e^{-it \mu}}{\sqrt{4\pi i t}}e^{i\frac{ x ^2}{4t}}\underline{e}_2e^{i\frac{ y^2}{4t}}\underline{e}_2^\top.
	\end{equation}
	Note that $\widetilde{F}_t^0 = \sigma_1 F_{-t}^0 \sigma_1$.
\end{remark}

\section{Laurent expansion of the resolvent near the threshold}

In this section we study asymptotic expansions of the perturbed resolvent operators near the thresholds of the essential spectrum, closely following the framework of the seminal paper \cite{00JensenNenciu} for the scalar Schr\"odinger operators $H = -\partial_x^2 + V$ on the real line. As specified in the introduction, we are interested in the irregular case, where the matrix Schr\"odinger operator $\calH$ exhibits a threshold resonance. See Definition~\ref{def: regular points} for a precise definition. This means that there exist globally bounded non-trivial solutions of $\calH \Psi = \pm \mu \Psi$.  
In this context, we mention that the threshold regularity can also be characterized by the scattering theory introduced by \cite{95BusleavPerelman}; see Lemma~5.20 of \cite{06KriegerSchlag}. We begin with the terminology used in \cite{05Schlag}.

\begin{definition}[Absolutely bounded operators]\label{def: absolutely bounded operators}
	We say an operator $A\colon L^2(\R)\times L^2(\R) \to L^2(\R)\times L^2(\R)$ with an integral kernel $A(x,y)\in \C^{2 \times 2}$ is absolutely bounded if the operator with the kernel $\vert A(x,y) \vert := (\vert A(x,y)_{i,j}\vert)_{i,j=1}^{2} \in \R^{2 \times 2}  $ is bounded from $L^2(\R)\times L^2(\R) \to L^2(\R)\times L^2(\R)$. In particular,  Hilbert-Schmidt and finite rank operators are absolutely bounded. 
\end{definition} 
To investigate the asymptotic expansions of the operator $M(z)$ (c.f.~\eqref{eqn: def of M}), we start with the following Taylor expansions of the free resolvent around the origin $z=0$.  
\begin{lemma}\label{lemma: expansion of R_0} Let $z_0 := \min \{1,\sqrt{2\mu}\}$. For any $0 < \vert z\vert < z_0$, we have the following expansion
	\begin{equation}\label{eqn: expansion of R_0}
\begin{split}
    \calR_0(z)(x,y)     &= \frac{i}{2z}\underline{e}_{11} + \calG_0(x,y) + z\calG_1(x,y) + E(z)(x,y)\\
    \end{split}		
	\end{equation}
	where
	\begin{align}
		\calG_0(x,y) &:= \begin{bmatrix}
			- \frac{\vert x - y \vert}{2} & 0 \\ 0 &	- \frac{e^{-\sqrt{ 2\mu}\vert x - y \vert }}{2\sqrt{2\mu}}
		\end{bmatrix}\label{eqn: G_0 def},\\
		\calG_1(x,y) &:=   \begin{bmatrix}
		\frac{\vert x - y \vert^2}{4i} & 0 \\ 0 & 0
	\end{bmatrix}\label{eqn: G_1 def},
	\end{align}
and $E(z)$ is an error term which satisfies the estimate 
\begin{equation}\label{eqn: error estimate for free resolvent}
  \vert z\vert^{k} \, \vert \partial_z^k E(z)(x,y)\vert \leq C_{\mu,k}\, \vert z \vert^2 \langle x \rangle^{3+k} \langle y \rangle^{3+k}, \quad \forall \ k=0,1,2,
\end{equation}
for any $\vert z \vert < z_0$.
\end{lemma}
\begin{proof}
    Recall from \eqref{eqn: free resolvent} that 
    \begin{equation*}
        \calR_0(z)(x,y) = \begin{bmatrix}
            \frac{ie^{i z \vert x - y \vert}}{2 z} & 0 \\ 0 & -\frac{e^{-\sqrt{z^2+2\mu}\vert x-y\vert}}{2\sqrt{z^2+2\mu}}
        \end{bmatrix}.
    \end{equation*}
	For $0 < \vert z \vert < 1$, we have the Laurent expansion
	\begin{equation}\label{eqn: taylor expand r_1}
		\frac{i e^{i z \vert x -y \vert}}{2 z} = \frac{i}{2 z} + \frac{-\vert x - y \vert}{2} + \frac{ \vert x - y \vert^2}{4i}z  + r_1(z,\vert x - y \vert),
	\end{equation}
	where the remainder term is 
	\begin{equation*}
		r_1(z,\vert x - y\vert) := \frac{i}{2z} \tilr_1(z,\vert x - y\vert), \quad \tilr_1(z,\vert x - y\vert) := \frac{(iz\vert x - y \vert)^3}{2!}\int_0^1 e^{isz \vert x - y \vert} (1-s)^2\,\ud s.
	\end{equation*}
	By direct computation, for any $x, y \in \R$ and for any $\vert z \vert <1$, we have the estimate
	\begin{equation}\label{eqnproof: error estimate for resolvent1}
		\vert z \vert^{k} \, \vert \partial_z^k\, r_1(z,\vert x - y \vert) \vert \lesssim \vert z \vert^2 \langle x\rangle^{3+k} \langle y \rangle^{3+k}, \quad k=0,1,2.
	\end{equation}
	In the lower component of the resolvent kernel, for $ \vert z \vert < 2\mu$, we have the Taylor expansion
	\begin{equation}\label{eqn: taylor expand r_2}
	-\frac{e^{-\sqrt{z^2+2\mu}\vert x-y\vert}}{2\sqrt{z^2+2\mu}} = -\frac{e^{-\sqrt{2\mu}\vert x - y\vert}}{2\sqrt{2\mu}} + r_2(z,\vert x - y \vert),
	\end{equation}
	where we denote  the remainder term by 
	\begin{equation*}
r_2(z,\vert x - y \vert) := \frac{z^2}{2!} \int_0^1 (1-s)  (\partial_z^2\, g_\mu)(sz,\vert x - y \vert) \,\ud s, \quad 		g_\mu(z,\vert x - y \vert) := -\frac{e^{-\sqrt{z^2+2\mu}\vert x-y\vert}}{2\sqrt{z^2+2\mu}}.
	\end{equation*}
	Using the fact that for any $\eta \in \R$, $\langle \eta \rangle := (1+\eta^2)^{\frac{1}{2}}$, one has the bounds
	\begin{equation*}
		\vert \partial_\eta^k \langle \eta \rangle^{-1} \vert \leq C_k \langle \eta \rangle^{-1-k} \qquad  \vert \partial_\eta^k \langle \eta \rangle \vert \leq C_k \langle \eta \rangle^{1-k}, \quad  k =0,1,2,\ldots,
	\end{equation*}
	it follows that all derivatives of $\sqrt{z^2+2\mu}$ and $2(z^2+2\mu)^{-\frac{1}{2}}$ are uniformly bounded in $z$ up to a constant depending only on $\mu$ and the number of derivatives. Therefore, by the Leibniz formula, we have the estimate
	\begin{equation*}
		\sup_{z \in \R} \left\vert \partial_z^k \, g_\mu(z,\vert x- y \vert) \right\vert \leq C_{\mu,k} \langle x \rangle^k \langle y \rangle^k, \quad k=0,1,\ldots,4,
	\end{equation*}
	which in turn implies that
	\begin{equation}\label{eqnproof: error estimate for resolvent2}
		\vert z \vert^{k} \left \vert \partial_z^k \, r_2(z,\vert x - y \vert) \right \vert \lesssim \vert z \vert^2 \langle x \rangle^{2+k} \langle y \rangle^{2+k},\quad k=0,1,2.
	\end{equation}
	Thus, by using \eqref{eqnproof: error estimate for resolvent1} and \eqref{eqnproof: error estimate for resolvent2}, the error term given by 
	\begin{equation}
		E(z)(x,y) := \begin{bmatrix}
			r_1(z, \vert x - y \vert) & 0 \\ 0 & r_2(z,\vert x - y \vert)
		\end{bmatrix}
	\end{equation}
	satisfies \eqref{eqn: error estimate for free resolvent} as claimed.
\end{proof}
We insert the above asymptotic expansion into the operator $M(z) = I + v_2\calR_0(z)v_1$. First, we have the \textit{transfer operator} $T$ on $L^2(\R) \times L^2(\R)$ with a kernel given by 
\begin{equation}\label{eqn: operator T}
	T(x,y) = I + v_2(x) \calG_0(x,y) v_1(y).
\end{equation}
Note that $T$ is self-adjoint because
\begin{equation*}
	(v_2\calG_0v_1)^* = v_1^*\calG_0 v_2 = (-v\sigma_3)\calG_0v = v\calG_0(-\sigma_3v) = v_2\calG_0v_1.
\end{equation*}
Since the potentials $v_1$ and $v_2$ have exponential decay by assumption~(A3), it follows that $v_2\calG_0 v_1$ is a  Hilbert-Schmidt operator on $L^2(\R) \times L^2(\R)$. Hence, $T$ is a compact perturbation of the identity, and therefore the dimension of $\ker(T)$ is finite by the Fredholm alternative. Recalling the formulas for $v_1$ and $v_2$ from \eqref{eqn: def of V,v_1,v_2}, we have the identity
\begin{equation}\label{eqn: v2e11v1}
	v_2 \underline{e}_{11} v_1 = -\begin{bmatrix}
		a & 0 \\ b & 0 
	\end{bmatrix}\begin{bmatrix}
		a & b \\ 0 & 0 
	\end{bmatrix} = - \begin{bmatrix}
	a \\ b
\end{bmatrix}\begin{bmatrix}
a & b
\end{bmatrix}.
\end{equation}
Next, we define the orthogonal projection  onto the span of the vector $(a,b)^\top \in L^2(\R) \times L^2(\R)$ by 
\begin{equation}\label{eqn: operator P}
	P\begin{bmatrix}
		f_1 \\ f_2
	\end{bmatrix}(x) := \frac{\int_{\R}( a(y)f_1(y) + b(y)f_2(y))\, \ud y}{\Vert a^2 + b^2 \Vert_{L^1(\R)}}\begin{bmatrix}
		a(x) \\ b(x) 
	\end{bmatrix} = \frac{1}{\Vert V_1 \Vert_{L^1(\R)}} \langle (a,b)^\top, \vec{f}\, \rangle \begin{bmatrix}a(x) \\ b(x)\end{bmatrix}.
\end{equation}
Note that we use the identity \eqref{identity V_1 V_2} above. From \eqref{eqn: def of M}, the contribution of the singular term $\frac{i}{2z}\underline{e}_{11}$ of $\calR_0(z)$ to $M(z)$ will be associated to the following integral operator with the kernel
\begin{equation}\label{eqn: frak_g}
	\frac{i}{2z}v_2(x)\underline{e}_{11}v_1(y) = \frac{- i}{2z}\begin{bmatrix}
		a(x) \\ b(x)
	\end{bmatrix} \begin{bmatrix}
		a(y) & b(y)
	\end{bmatrix} =: g(z)P(x,y),
\end{equation}
where 
\begin{equation}\label{eqn: def g}
	g(z) := -\frac{i}{2z} \Vert V_1 \Vert_{L^1(\R)}.
\end{equation}
Lastly, we denote the orthogonal projection to the complement of the span of $(a,b)^\top$ by
\begin{equation}
	Q := I - P.
\end{equation}
In summary, we have the following proposition.
\begin{proposition}\label{prop: expansion of M} Suppose $\vert a(x) \vert, \vert b(x) \vert \lesssim \langle x \rangle^{-5.5-}$, and let $z_0 := \min \{1,\sqrt{2\mu}\}$. Then, for any $0<\vert z \vert < z_0$, we have 
	\begin{equation}\label{eqn: expansion for M}
\begin{split}
		M(z) &= g(z)P + T + zM_1 + \calM_2(z),
\end{split}
	\end{equation}
	where $M_1$ and $\calM_2(z)$ are Hilbert-Schmidt operators on $L^2(\R) \times L^2(\R)$ defined by
	\begin{align}
		M_1(x,y) &:= v_2(x)\calG_1(x,y)v_1(y) = \frac{\vert x - y \vert^2}{4i} \begin{bmatrix}
			a(x) \\ b(x)
		\end{bmatrix}\begin{bmatrix}
		a(y) & b(y)
	\end{bmatrix}, \label{eqn: def of M_1}\\
		\calM_2(z)(x,y) &:= v_2(x)E(z)(x,y)v_1(y),\label{eqn: def of calM_2}
	\end{align}
	with $G_1$ and $E(z)$ defined in Lemma~\ref{lemma: expansion of R_0}. Moreover, the error term $\calM_2(z)$ and its derivatives satisfy the absolute bound 
	\begin{equation}
		\vert z \vert^k \left\Vert \vert\partial_z^k  \calM_2(z) \vert \right\Vert_{L^2(\R) \times L^2(\R) \to L^2(\R) \times L^2(\R)} \lesssim \vert z \vert^2, \quad k =0,1,2,
	\end{equation}
	for all $\vert z \vert < z_0$.
\end{proposition}	
\begin{proof}
The identity on the right of \eqref{eqn: def of M_1} follows from  \eqref{eqn: v2e11v1}. We recall that operators of the following type
\begin{equation*}\label{eqn: HS kernel1}
	U(x)\langle x\rangle^k\langle y\rangle^k W(y)
\end{equation*} are Hilbert-Schmidt operators on $L^2(\R)$ whenever $U$ and $W$ are smooth potentials with polynomial decay $\vert U(x) \vert, \vert W(x) \vert \lesssim \langle x \rangle^{-k-\frac{1}{2}-}$, for $k \in \N$. Hence, under the assumptions on $a(x)$ and $b(x)$, and using the fact that
\begin{equation*}
	\vert \calG_1(x,y)\vert \lesssim \vert x -y \vert^{2}\leq \langle x \rangle^{2}\langle y \rangle^{2},
\end{equation*}
it follows that $M_1$ is Hilbert-Schmidt. The same argument can be applied to the error term $\calM_2(z)$ and its derivatives using the remainder estimates in \eqref{eqn: error estimate for free resolvent} and we are done.
\end{proof}
The next definition characterizes the regularity of the endpoint $\mu$ of the essential spectrum. 
\begin{definition} \label{def: regular points}
	\begin{enumerate}
		\item We say that the threshold $\mu$ is a regular point of the spectrum of $\calH$ provided that the operator $QTQ$ is invertible on the subspace $Q(L^2(\R) \times L^2(\R))$.
		\item Suppose $\mu$ is not a regular point. Let $S_1$ be the Riesz projection onto the kernel of $QTQ$, and we define $D_0 = (Q(T+S_1)Q)^{-1}$.  Note that $QD_0Q$ is an absolutely bounded operator on $L^2(\R)\times L^2(\R)$. The proof for this follows from Lemma~8 of \cite{05Schlag} with minor changes. See also \cite[Lemma 2.7]{15GoldbergGreen}. 
	\end{enumerate}
\end{definition}
Note that since we impose symmetry assumptions on the potential $\calV$, the thresholds $\mu$ and $-\mu$ are either both regular or irregular. The invertibility of $QTQ$ is related to the absence of distributional $L^\infty(\R) \times L^\infty(\R)$ solutions to $\calH\Psi = \mu \Psi$. The following lemma establishes the equivalent definitions. See \cite[Lemma 5.4]{00JensenNenciu} for the analogue in the scalar case.  
\begin{lemma} \label{lemma: threshold resonance characterization}
	Suppose assumptions~(A1) -- (A5) hold. Then the following holds.
	\begin{enumerate}
		\item  Let $\Phi \in S_1(L^2(\R) \times L^2(\R)) \setminus \{0\}$. If $\Psi = (\Psi_1,\Psi_2)^\top$ is defined by 
		\begin{equation}\label{eqn: Psi}
			\Psi(x) := -\calG_0[ v_1 \Phi](x) + c_0 \underline{e}_1,
		\end{equation}
		with 
		\begin{equation}\label{eqn: def c_0}
			c_0 = \frac{\langle(a,b)^\top, T\Phi\rangle}{\Vert V_1 \Vert_{L^1(\R)}},
		\end{equation}
		then
		\begin{equation}\label{eqn: Phi = v_2 Psi}
			\Phi = v_2 \Psi,
		\end{equation}
		and $\Psi\in L^\infty(\R) \times L^\infty(\R)$ is a distributional solution to 
		\begin{equation}\label{eqn: resonance eqn}
			\calH\Psi = \mu\Psi.
		\end{equation}
		Furthermore, if additionally assumption~(A6) holds, i.e.,
		\begin{equation}\label{eqn: def of c_2}
			c_{2,\pm} :=  \frac{1}{2\sqrt{2\mu}} \int_{\R} e^{\pm \sqrt{2\mu}y} \big(V_2(y) \Psi_1(y) + V_1(y)\Psi_2(y)\big) \, \ud y = 0,
		\end{equation}
		then 
		\begin{equation}\label{eqn: asymptotics of Psi_1}
			\lim_{x \to \pm \infty} \Psi_1(x) = c_0 \mp c_1,
		\end{equation}
	where 
		\begin{equation}\label{def: c_1}
		c_1 := \frac{1}{2} \langle x(a(x),b(x))^\top,\Phi(x)\rangle = \frac{1}{2}  \int_\R x \big( a(x)\Phi_1(x) + b(x) \Phi_2(x)\big) \,\ud x.
	\end{equation}
		In particular,
		\begin{equation}
		\Psi_1 \notin L^2(\R).
		\end{equation}		
		More precisely, the constants $c_0$ and $c_1$ cannot both be zero. 
		\item Conversely, suppose there exists $\Psi \in L^\infty(\R) \times L^\infty(\R)$ satisfying \eqref{eqn: resonance eqn} in the distributional sense. Then
		\begin{equation}
			\Phi = v_2 \Psi \in S_1 (L^2(\R) \times L^2(\R)).
		\end{equation}	
		\item Suppose assumptions~(A1) -- (A6) hold. Then, $\dim S_1(L^2(\R) \times L^2(\R)) \leq 1$. In the case $\dim S_1(L^2(\R) \times L^2(\R)) =1$, i.e., $S_1(L^2(\R)\times L^2(\R)) = \Span\{\Phi\}$ for some $\Phi = (\Phi_1,\Phi_2)^\top \in L^2(\R) \times L^2(\R) \setminus \{0\}$, we have the following identities
		\begin{align}
			S_1 T P T S_1 &= \vert c_0\vert ^2 \Vert \Phi \Vert_{L^2(\R)\times L^2(\R)}^{-2} \Vert V_1 \Vert_{L^1(\R)} S_1,			\label{eqn: value of STPTS}\\
			PTS_1TP &= \vert c_0\vert^2 \Vert \Phi \Vert_{L^2(\R)\times L^2(\R)}^{-2} \Vert V_1 \Vert_{L^1(\R)}P,\label{eqn: PTS_1TP}\\
			S_1M_1S_1 &= -2i \vert c_1 \vert^2\Vert \Phi \Vert_{L^2(\R)\times L^2(\R)}^{-2} S_1,\label{lemma: tr S_1M_1S_1}
		\end{align}
	where the constants $c_0$ and $c_1$ are given by \eqref{eqn: def c_0} and \eqref{def: c_1} respectively for this $\Phi$.
	\end{enumerate}
	
\end{lemma}
\begin{proof}[Proof of (1)] Let $\Phi = (\Phi_1,\Phi_2) \in S_1(L^2(\R) \times L^2(\R))$ with $\Phi \neq 0$. Since $S_1(L^2(\R) \times L^2(\R))$ is a subspace of $Q(L^2(\R) \times L^2(\R))$, we have $Q\Phi = \Phi$. Using the fact that $\Phi \in \ker(QTQ)$ and the definition of $T$ (c.f \eqref{eqn: operator T}), we obtain
	\begin{equation}\label{eqn: PTphi = Tphi}
		0 = QTQ\Phi = (I- P)T\Phi = (I+v_2 \calG_0 v_1)\Phi - PT\Phi.
	\end{equation}
	Since $(a,b)^\top = v_2 \underline{e}_1$ and $P$ is the orthogonal projection onto the span of $(a,b)^\top$, we have
	\begin{equation}\label{eqn: PTPhi}
		PT\Phi = \frac{\langle (a,b)^\top , T\Phi \rangle  }{\Vert V_1 \Vert_{L^1(\R)}}(a,b)^\top = c_0 v_2 \underline{e}_1,
	\end{equation}
	with $c_0$ defined in \eqref{eqn: def c_0}. It follows that
	\begin{equation*}
		\Phi = -v_2\calG_0v_1 \Phi + c_0 v_2 \underline{e}_1 = v_2(-\calG_0v_1 \Phi + c_0 \underline{e}_1) = v_2 \Psi. 
	\end{equation*}
	This proves \eqref{eqn: Phi = v_2 Psi}.  Next, we show \eqref{eqn: resonance eqn}. Denoting $\Phi = (\Phi_1,\Phi_2)^\top$ and using the definition of $\mathcal{G}_0$ (c.f. \eqref{eqn: G_0 def}), we have
	\begin{equation*}
		(\calH_0 - \mu I)\calG_0 (v_1\Phi) = v_1 \Phi ,
	\end{equation*}
	i.e.,
	\begin{equation*}
		\begin{split}
			\begin{cases}
				(-\partial_x^2)\displaystyle\int_\R \frac{- \vert x - y \vert}{2}\big(-a(y)\Phi_1(y) - b(y)\Phi_2(y)\big) \, \ud y = -a(x)\Phi_1(x) - b(x)\Phi_2(x), \\
				(\partial_x^2 - 2\mu) \displaystyle\int_\R \frac{-e^{-\sqrt{ 2\mu}\vert x - y \vert }}{2\sqrt{2\mu}}\big(b(y)\Phi_1(y) + a(y)\Phi_2(y)\big) \, \ud y = b(x)\Phi_1(x) + a(x)\Phi_2(x).
			\end{cases}
		\end{split}
	\end{equation*}
	This equation is well-defined, since $ v_1\Phi \in \langle x \rangle^{-1-} L^1(\R) \times \langle x \rangle^{-1-} L^1(\R)$. Using \eqref{eqn: Psi}, \eqref{eqn: Phi = v_2 Psi}, and $(H_0 - \mu I)(c_0 \underline{e}_1) = 0$, we have
	\begin{equation*}
		(\calH_0 - \mu I)\Psi = (H_0 - \mu I)[-\calG_0 (v_1 \Phi)+c_0 \underline{e}_1] = - v_1 \Phi = -v_1 v_2 \Psi = -\calV \Psi,
	\end{equation*}
	which implies \eqref{eqn: resonance eqn}. We now show that $\Psi = (\Psi_1,\Psi_2)^\top$ is in $L^\infty(\R) \times L^\infty(\R)$.  
	Noting that 
	\begin{equation*}
		\Psi_1(x) = c_0 + \frac{1}{2}\int_\R \vert x - y \vert \big(a(y)\Phi_1(y) + b(y) \Phi_2(y)\big)  \, \ud y,
	\end{equation*}
	by employing the orthogonality condition $\langle (a,b)^\top,\Phi \rangle = 0$, we have
	\begin{equation*}
		\Psi_1(x) = c_0 + \frac{1}{2}\int_\R (\vert x - y \vert - \vert x \vert) \big(a(y)\Phi_1(y) + b(y) \Phi_2(y)\big)   \, \ud y.
	\end{equation*}
	Using $\big\vert \vert x - y \vert - \vert x \vert\big\vert  \leq \vert y \vert$ and $ \vert a(y) \vert + \vert b(y) \vert \lesssim \langle y \rangle^{-2}$, we have the uniform bound 
	\begin{equation*}
		\sup_{x \in \R} \vert \Psi_1(x) \vert \leq \vert c_0 \vert + \frac{1}{2}\int \vert y \vert \left\vert a(y)\Phi_1(y) + b(y) \Phi_2(y)\right\vert  \, \ud y \lesssim \Vert \Phi \Vert_{L^2(\R)\times L^2(\R)} \lesssim 1.
	\end{equation*}
	Since $(a,b)^\top$ and $\Phi$ are in $L^2(\R) \times L^2(\R)$, we have the uniform bound on $\Psi_2$ by the Cauchy-Schwarz inequality
	\begin{equation*}
		\sup_{x \in \R} \vert \Psi_2(x) \vert \lesssim \int_\R\vert b(y)\Phi_1(y) + a(y)\Phi_2(y) \vert  \, \ud y \leq \Vert b \Vert_{L^2(\R)}\Vert \Phi_1 \Vert_{L^2(\R)} + \Vert a \Vert_{L^2(\R)}\Vert \Phi_2 \Vert_{L^2(\R)} \lesssim 1.
	\end{equation*}
	Thus, we have shown that  $\Psi =(\Psi_1,\Psi_2)^\top \in L^\infty(\R) \times L^\infty(\R)$. 	Finally, we now assume $c_{2,\pm} = 0 $ and show that $\Psi_1$ cannot be in $L^2(\R) \setminus \{0\}$ by a Volterra argument. Using  $\langle (a,b)^\top,\Phi \rangle = 0$,  for $x \geq 0$ large, we write
	\begin{equation}\label{eqn: volterra x>0}
		\begin{split}
			\Psi_1(x) &= c_0 - c_1 + \int_x^\infty (y-x) \big(a(y)\Phi_1(y) + b(y) \Phi_2(y)\big) \, \ud y. 
		\end{split}
	\end{equation}	
Using $c_{2,\pm} = 0$, we insert $-e^{-\sqrt{2\mu}x}c_{2,+} = 0$ to write 
\begin{equation}\label{eqn: volterra x>02}
	\begin{split}
	\Psi_2(x) &= \frac{1}{2\sqrt{2\mu}} \int_x^\infty \big(e^{-\sqrt{2\mu}(y-x)}-e^{-\sqrt{2\mu}(x-y)}\big) \big(V_2(y)\Psi_1(y) + V_1(y)\Psi_2(y)\big) \, \ud  y.
	\end{split}
\end{equation}
Similarly, for $x<0$, using $e^{\sqrt{2\mu x}}c_{2,-}=0$, we have 
\begin{align}
	&\Psi_1(x) = c_0 + c_1 +\int_{-\infty}^x (x-y)(V_1(y) \Psi_1(y) +V_2(y) \Psi_2(y))\, \ud y,\label{eqn: volterra x<0 phi1}\\
	&\Psi_2(x) = \frac{1}{2\sqrt{2\mu}}\int_{-\infty}^x \big(e^{-\sqrt{2\mu}(x-y)}-e^{-\sqrt{2\mu}(y-x)}\big) \big(V_2(y)\Psi_1(y) + V_1(y)\Psi_2(y)\big)  \, \ud y.\label{eqn: volterra x<0}
\end{align}
Suppose now that $c_0 = c_1 = 0$. Owing to the exponential decay of $V_1$, $V_2$ by assumption (A3), we obtain from \eqref{eqn: volterra x>0} and \eqref{eqn: volterra x>0} a homogeneous Volterra equation for $\Psi=(\Psi_1,\Psi_2)^\top$ satisfying
\begin{equation*}
	\Psi(x) = \int_{\R} K(x,y) \Psi(y) \, \ud y, \quad x \geq 0,
\end{equation*}
where $\vert K(x,y) \vert \lesssim e^{-\gamma \vert y \vert}\mathbbm{1}_{y > x}$ for some $0< \gamma < \beta$, which is a quasi-nilpotent operator. By performing a standard contraction on $L^\infty(M,\infty)$, with  $M>0$ sufficiently large, one arrives at a solution $\Psi(x) \equiv 0$ for all $x \geq M$. By the uniqueness theorem for ODEs, this implies that $\Psi \equiv 0$ on $\R$. Then, by the relation $\Phi = v_2 \Psi$ and the fact that $v_2$ is a positive matrix, one finds that $\Phi \equiv 0$, which contradicts the hypothesis $\Phi \neq 0$. Thus, the conclusion is that  $c_0$ and $c_1$ cannot be both zero. In particular, it follows from \eqref{eqn: volterra x>0} and \eqref{eqn: volterra x<0 phi1} that 
\begin{equation*}
	\lim_{x \rightarrow \pm \infty} \Psi_1(x) = c_0 \mp c_1.
\end{equation*}
Since either $c_0+c_1 \neq 0$ or $c_0 - c_1 \neq 0$, we conclude that $\Psi_1 \not \in L^2(\R)$.

	\textit{Proof of (2)}. Define $\Phi = v_2 \Psi$. Since $\Psi$ is a distributional solution to \eqref{eqn: resonance eqn}, using $\calV = v_1v_2$,  we have
	\begin{equation*}
		(\calH_0 - \mu I)\Psi = v_1 \Phi  \Longleftrightarrow \begin{cases}
			\Psi_1'' = a \Phi_1 + b\Phi_2, \\
			\Psi_2'' - 2\mu \Psi_2 = b \Phi_1 + a \Phi_2.
		\end{cases}
	\end{equation*}
	Let $\eta \in C_0^\infty(\R)$ be a non-negative function satisfying $\eta(x) = 1$ for $ \vert x \vert \leq1$ and $\eta(x) = 0$ for $\vert x \vert \geq 2$. Using the first equation from above and integrating by parts, we have for any $\eps>0$,
	\begin{equation*}
		\begin{split}
			&\left\vert \int_\R \big(a(y)\Phi_1(y) + b(y)\Phi_2(y)\big) \eta(\eps y) \,\ud y \right \vert = \left\vert \int_\R \Psi_1''(y) \eta(\eps y)\, \ud y \right \vert \\
			&= \left\vert \int_\R \Psi_1(y) \eps^2 \eta''(\eps y)\, \ud y \right \vert \leq \eps \Vert \Psi_1 \Vert_{L^\infty(\R)}  \int_\R  \left\vert\eta''(x) \right \vert \, \ud x.
		\end{split}
	\end{equation*}
	By taking the limit $\eps \to 0$ and using the Lebesgue dominated convergence theorem, we find that  $\langle (a,b)^\top, \Phi \rangle = 0$. Thus, $P\Phi = 0$, i.e. $\Phi \in Q(L^2(\R) \times L^2(\R))$. Following this fact and using $\Phi = v_2 \Psi$, we have
	\begin{equation}\label{eqn: QTQPhi}
		QTQ\Phi = QT\Phi = Q(I+v_2\calG_0v_1)\Phi = Qv_2\big(\Psi +\calG_0(\calV\Psi)\big).
	\end{equation}
	Now set $u := \Psi + \calG_0(\calV\Psi)$. Since $u = (u_1,u_2)^\top$ is a distributional solution of $(\calH_0 - \mu I)u = 0$, i.e.
	\begin{equation*}
		\begin{split}
			-u_1'' &= 0, \\
			u_2'' - 2\mu u_2 &= 0,
		\end{split}
	\end{equation*}
	we find that 
	\begin{equation*}
		\begin{split}
			&u_1(x) = \kappa_1 + \kappa_2x,\\
			&u_2(x) = \kappa_3e^{-\sqrt{2\mu}x} + \kappa_4e^{\sqrt{2\mu}x},
			\end{split}
	\end{equation*}
	for some $\kappa_i \in \C$, $i \in \{1,\ldots,4\}$. By similar arguments from Item~(1), we obtain that $\calG_0(\calV\Psi) \in L^\infty(\R) \times L^\infty(\R)$. Since $\Psi \in L^\infty(\R) \times L^\infty(\R)$, it follows that $ u \in L^\infty(\R) \times L^\infty(\R)$, which implies that $\kappa_2 = \kappa_3 = \kappa_4 = 0$. Thus, we have $u(x) \equiv (\kappa_1,0)^\top = \kappa_1\underline{e}_1$. Since $Qv_2 \underline{e}_1 = 0$, we conclude from \eqref{eqn: QTQPhi} using the definition of $u(x)$ that $QTQ\Phi = 0$, whence $\Phi \in S_1(L^2(\R) \times L^2(\R))$.
	
\textit{Proof of (3)}. Suppose there are two linearly independent $\Phi,\wtilPhi \in S_1(L^2(\R) \times L^2(\R))$. As in the proof of Item~(1), for $x \geq 0$, we have
	\begin{equation*}
		\begin{split}
			&\Psi_1(x) = c_0 - c_1 + \int_x^\infty (y-x) \big(V_1(y)\Psi_1(y) + V_2(y) \Psi_2(y)\big) \, \ud y,\\
			&\Psi_2(x) = \frac{1}{2\sqrt{2\mu}} \int_x^\infty \big(e^{-\sqrt{2\mu}(y-x)}-e^{-\sqrt{2\mu}(x-y)}\big) \big(V_2(y)\Psi_1(y) + V_1(y)\Psi_2(y)\big) \, \ud y,
		\end{split}
	\end{equation*}
	and
	\begin{equation*}
	\begin{split}
		&\widetilde{\Psi}_1(x) = d_0 - d_1 + \int_x^\infty (y-x) \big(V_1(y)\widetilde{\Psi}_1(y) + V_2(y) \widetilde{\Psi}_2(y)\big) \, \ud y,\\
		&\widetilde{\Psi}_2(x) = \frac{1}{2\sqrt{2\mu}} \int_x^\infty \big(e^{-\sqrt{2\mu}(y-x)}-e^{-\sqrt{2\mu}(x-y)}\big) \big(V_2(y)\widetilde{\Psi}_1(y) + V_1(y)\widetilde{\Psi}_2(y)\big) \, \ud y,
	\end{split}
\end{equation*}
where $d_0$ and $d_1$ are constants defined from $\wtilPhi$ which are analogous to $c_0$ and $c_1$. There is some constant $\theta \in \C$ such that 
\begin{equation*}
	c_0 - c_1 = -\theta (d_0 - d_1),
\end{equation*}
which imply the Volterra integral equation
\begin{equation*}
	\begin{bmatrix}
		\Psi_1 + \theta \widetilde{\Psi}_1 \\ \Psi_2 + \theta \widetilde{\Psi}_2
	\end{bmatrix}(x) = \int_{x}^\infty \begin{bmatrix}
	y - x& 0 \\ 0 & \frac{e^{-\sqrt{2\mu}(y-x)}-e^{-\sqrt{2\mu}(x-y)}}{2\sqrt{2\mu}}
\end{bmatrix} \calV(y) \begin{bmatrix}
\Psi_1(y) + \theta \widetilde{\Psi}_1(y) \\ \Psi_2(y) + \theta \widetilde{\Psi}_2(y)
\end{bmatrix}dy,
\end{equation*}
for any $x \geq 0$. By the same Volterra equation argument used in Item (1), we obtain  $\Psi+ \theta \widetilde{\Psi} \equiv 0$, which implies that $\Phi + \theta \widetilde{\Phi} \equiv 0$, but this contradicts that $\Phi$ and $\widetilde{\Phi}$ are linearly independent. Thus, we have shown that $\dim S_1(L^2(\R) \times L^2(\R)) \leq 1$. Next, we prove \eqref{eqn: value of STPTS}--\eqref{lemma: tr S_1M_1S_1}. Write $S_1 = \Vert \Phi \Vert_{L^2 \times L^2}^{-2}\langle \Phi,\cdot \rangle \Phi$. By \eqref{eqn: PTPhi} and the fact that $P$, $S_1$, and $T$ are self-adjoint, we compute for any $u \in L^2(\R) \times L^2(\R)$ that 
\begin{equation*}
	S_1 T P T S_1 u =  \Vert \Phi \Vert_{L^2 \times L^2}^{-2}\langle \Phi,u \rangle S_1 T P T \Phi =  \Vert \Phi \Vert_{L^2 \times L^2}^{-2}c_0\langle \Phi,u\rangle S_1T \begin{bmatrix}
			a \\ b 
		\end{bmatrix} = \vert c_0 \vert^2  \Vert \Phi \Vert_{L^2 \times L^2}^{-2} \Vert V_1 \Vert_{L^1(\R)}S_1 u.
\end{equation*}
A similar computation reveals
\begin{equation*}
	PTS_1TPu =  \vert c_0\vert^2  \Vert \Phi \Vert_{L^2 \times L^2}^{-2} \Vert V_1 \Vert_{L^1(\R)}Pu.
\end{equation*}
For the third identity \eqref{lemma: tr S_1M_1S_1}, in view of \eqref{eqn: v2e11v1} and \eqref{eqn: def of M_1}, we write
\begin{equation*}
	M_1(x,y) = v_2(x)G_1(x,y)v_1(y) = \frac{i\vert x - y \vert^2}{4}\begin{bmatrix}
		a(x) \\ b(x)
	\end{bmatrix}\begin{bmatrix}
	a(y) & b(y)
\end{bmatrix}.
\end{equation*}
By using the orthogonality  
\begin{equation*}
	\langle \Phi,(a,b)^\top\rangle = \int_\R \big(\Phi_1(x)a(x) + \Phi_2(x)b(x)\big) \, \ud x = 0,
\end{equation*}
and the identity
\begin{equation*}
	\vert x - y \vert^2 = x^2 + y^2 - 2xy,
\end{equation*}
we have
\begin{equation*}
	\begin{split}
		&[S_1M_1S_1](x,y) = \int_{\R^2} S_1(x,x_1)M_1(x_1,y_1)S_1(y_1,y) \,\ud x_1 \,\ud y_1\\
		&\quad = \frac{i}{4}\frac{\Phi(x)}{\Vert \Phi \Vert_{L^2 \times L^2}^2}\int_{\R^2}\left( \vert x_1 - y_1 \vert^2 \Phi^*(x_1)\begin{bmatrix}
			a(x_1) \\ b(x_1)
		\end{bmatrix}\begin{bmatrix}
			a(y_1) & b(y_1)
		\end{bmatrix} \Phi(y_1)\right) \, \ud x_1 \ud y_1 \frac{\Phi^*(y)}{\Vert \Phi\Vert_{L^2 \times L^2}^2}\\
	&\quad = -2i \left(\int_{\R}  \frac{x_1}{2} \Phi^*(x_1)\begin{bmatrix}
		a(x_1) \\ b(x_1)
	\end{bmatrix}\,\ud x_1\right) \left(\int_{\R} \tfrac{y_1}{2} \begin{bmatrix}
		a(y_1) & b(y_1)
	\end{bmatrix} \Phi(y_1) \, \ud y_1\right) \Vert \Phi \Vert_{L^2 \times L^2}^{-2}S_1(x,y)\\
	&\quad = -2i \vert c_1 \vert^2 \Vert \Phi \Vert_{L^2 \times L^2}^{-2} S_1(x,y).
	\end{split}
\end{equation*}
This proves \eqref{lemma: tr S_1M_1S_1} and we are done. 
\end{proof}
\begin{remark}\label{remark: psi_2}
	By direct computation, the conjugation identity $\sigma_3 \calH = \calH^* \sigma_3$ and the identity $v_1 = -\sigma_3 v_2$ imply that the vector $\widetilde{\Psi} := \sigma_3 \Psi$ solves
	\begin{equation}
		\calH^* \widetilde{\Psi} = \mu \widetilde{\Psi},
	\end{equation}
	where $\Psi$ is the distribution solution to \eqref{eqn: resonance eqn}. Moreover, one has the identities
	\begin{equation}\label{eqn: sigma_3Psi}
		\sigma_3 \Psi = \calG_0 (v_2 \Phi) + (c_0,0)^\top, \quad \Phi = v_2 \Psi = -v_1^\top \widetilde{\Psi}
	\end{equation}
	Similarly, using the conjugation identity $\sigma_1 \calH = - \calH\sigma_1$, we note that the vector $\Psi_- = \sigma_1 \Psi $ solves the system 
	\begin{equation}
		\calH\Psi_- = -\mu \Psi_-.
	\end{equation}
\end{remark}
Following the preceding discussion, we assume the threshold $\mu$ is irregular and we derive an expansion for the inverse operator $M(z)^{-1}$ on a small punctured disk near the origin. We employ the inversion lemma due to Jensen and Nenciu \cite[Lemma~2.1]{00JensenNenciu}. 
\begin{lemma}\label{lemma: JN}
	Let $H$ be a Hilbert space, let $A$ be a closed operator and $S$ a projection. Suppose $A+S$ has a bounded inverse. Then $A$ has a bounded inverse if and only if 
	\begin{equation*}
		B = S - S(A+S)^{-1}S
	\end{equation*}
	has a bounded inverse in $SH$, and in this case,
	\begin{equation*}
		A^{-1} = (A+S)^{-1} + (A+S)^{-1}SB^{-1}S(A+S)^{-1}, \quad \text{on $H$}.
	\end{equation*}
\end{lemma}
We will now state the inverse operator of $M(z)$ away from $z=0$. 
\begin{proposition} \label{prop: invert M} Suppose assumptions~(A1) -- (A6) hold. Let $S_1(L^2(\R) \times L^2(\R)) = \Span(\{\Phi\})$ for some $\Phi = (\Phi_1,\Phi_2)^\top \neq \vec0$. Let $\kappa := (2i)^{-1}\Vert V_1 \Vert_{L^1(\R)}$, and let $d$ be the constant defined by 
	\begin{equation}
		d := -2i(\vert c_0 \vert^2 + \vert c_1 \vert^2) \Vert \Phi \Vert_{L^2 \times L^2}^{-2} \neq 0,
	\end{equation}	
	with $c_0$ and $c_1$ defined by \eqref{eqn: def c_0} and \eqref{def: c_1} respectively for this $\Phi$. Then, there exists a positive radius $z_0>0$ such that for all $0 < \vert z\vert < z_0$, $M(z)$ is invertible on $L^2(\R) \times L^2(\R)$ and
	\begin{equation}\label{eqn: M inverse}
		\begin{split}
		M(z)^{-1} &= \frac{1}{d}\left(\frac{1}{z}S_1 - \frac{1}{\kappa} PTS_1 - \frac{1}{\kappa} S_1TP\right) +\left( \frac{1}{\kappa} + \frac{\vert c_0 \vert^2 \Vert \Phi \Vert_{L^2 \times L^2}^{-2} \Vert V_1 \Vert_{L^1(\R)}  }{d\kappa^2} \right)zP \\
		&\qquad + Q \Lambda_0(z) Q + zQ\Lambda_1(z) + z\Lambda_2(z)Q + z^2\Lambda_3(z),
		\end{split}
	\end{equation}
where $\Lambda_j(z)$ are absolutely bounded operators on $L^2(\R) \times L^2(\R)$  satisfying the improved bounds
\begin{equation}\label{eqn: condition on Lambda_j}
	\Vert \vert \partial_z^k \Lambda_j(z) \vert \Vert_{L^2(\R) \times L^2(\R) \to L^2(\R) \times L^2(\R)} \lesssim 1, \quad k=0,1,2,\quad j=0,1,2,3,
\end{equation} 
uniformly in $z$ for $\vert z\vert < z_0$.
\end{proposition}
\begin{proof} Throughout the proof, we will denote by $\calE_j(z)$, for $0 \leq j \leq 3$, as error terms that satisfy the absolute bound
	\begin{equation*}
		\vert z \vert^k \left\Vert \vert \partial_z^k \calE_j(z) \vert   \right\Vert_{L^2(\R) \times L^2(\R) \rightarrow  L^2(\R) \times L^2(\R)} \lesssim \vert z \vert^j, \quad\forall\,  k = 0,1,2, \quad\forall\,  \vert z \vert < z_0,
	\end{equation*}
	for some $z_0>0$ small. This convenient notation will be useful in invoking Neumann series inversion for small values of $z$. Since we only need the expansion of $M(z)^{-1}$ up to a few powers of $z$, the exact expressions of $\calE_j(z)$ are insignificant and we allow it to vary from line to line. By Proposition~\ref{prop: expansion of M}, we rewrite $M(z)$ by setting 
	\begin{equation}\label{eqn: def wtilM}
		\wtilM(z) := \frac{z}{\kappa}M(z) = P + \frac{z}{\kappa}\big (T + zM_1 + \calM_2(z)\big),
	\end{equation}
	where $\calM_2(z)$ is the error term in Proposition~\ref{prop: expansion of M}. Using $I = P + Q$, we write
	\begin{equation}
		\wtilM(z) + Q = I + \frac{z}{\kappa}\big (T + zM_1 + \calM_2(z)\big),
	\end{equation}
	and by choosing $z$ small enough, a Neumann series expansion yields the inverse operator 
	\begin{equation}\label{eqn: Neumann1}
		[\wtilM(z)+Q]^{-1} = \sum_{n \geq 0} (-1)^n \left(\frac{z}{\kappa}\big(T + zM_1 +  \calM_2(z)\big)\right)^n \quad \text{on $L^2(\R)\times L^2(\R)$}.
	\end{equation}
	We collect the terms of power order up to $2$ to obtain
	\begin{equation}\label{eqn: def wtilM+Q}
		[\wtilM(z)+Q]^{-1} = I - \frac{z}{\kappa}T - z^2\left ( \frac{1}{\kappa}M_1 -\frac{1}{\kappa^2}T^2\right ) + \calE_3(z).
	\end{equation}
	Note that $z\calM_2(z)$ is of the form $\calE_3(z)$. Recall by Lemma~\ref{lemma: JN} that the operator  $\wtilM(z)$ is invertible on $L^2(\R) \times L^2(\R)$ if and only if the operator 
	\begin{equation}
		B_1(z) := Q-Q\big[\wtilM(z)+Q\big]^{-1}Q
	\end{equation}
	is invertible on the subspace $QL^2 \equiv Q(L^2(\R) \times L^2(\R))$. Using \eqref{eqn: def wtilM+Q}, we find that
	\begin{equation*}
		B_1(z) = \frac{z}{\kappa}QTQ + z^2\left(\frac{1}{\kappa}QM_1Q - \frac{1}{\kappa^2}QT^2Q\right) + Q\calE_3(z)Q.
	\end{equation*}
	We rewrite $B_1(z)$ by setting
	\begin{equation}
		\wtilB_1(z) := \frac{\kappa}{z}B_1(z) = QTQ + z\left(QM_1Q - \frac{1}{\kappa}QT^2Q\right) + Q\calE_2(z)Q.
	\end{equation}
	Since the threshold $\mu$ is not regular, the operator $QTQ$ is not invertible on $QL^2$ according to Definition~\ref{def: regular points}. By considering the operator
	\begin{equation*}
		\wtilB_1(z) + S_1 = (QTQ + S_1) +  z\left(QM_1Q - \frac{1}{\kappa}QT^2Q\right) + Q\calE_2(z)Q,
	\end{equation*}
	and the fact that we have $QD_0Q = D_0 = (QTQ+S_1)^{-1}$ on $QL^2$, we can pick $z$ small enough such that 
	\begin{equation*}
		\left\Vert z\left(QM_1Q - \frac{1}{\kappa}QT^2Q\right) + Q\calE_2(z)Q \right \Vert_{L^2 \times L^2 \to L^2 \times L^2} < \Vert QD_0Q \Vert_{L^2 \times L^2 \to L^2 \times L^2}^{-1}.
	\end{equation*}
	This allows for the more complicated Neumann series expansion (c.f. Lemma~\ref{lemma: Neumann series}) on $QL^2$:
	\begin{equation}\label{eqn: Neumann2}
		(\wtilB_1(z) + S_1)^{-1} = D_0\sum_{n\geq0}  (-1)^n\Big( \big(z (QM_1Q - \kappa^{-1} QT^2Q ) + Q\calE_2(z)Q\big)D_0\Big)^n \quad \text{on $QL^2$}.
	\end{equation}
	We collect the leading order terms in this expansion and write 
	\begin{equation}\label{eqn: tilB_1 + S_1 inv}
		\begin{split}
			(\wtilB_1(z) + S_1)^{-1} = D_0 - zD_0\left(QM_1Q - \kappa^{-1}QT^2Q\right)D_0 + Q\calE_2(z)Q.
		\end{split}
	\end{equation}
	At this step, it is crucial that the operator $D_0$ is absolutely bounded to ensure that the remainder term $Q\calE_2(z)Q$ and its derivatives are absolutely bounded.	Next, we set 
	\begin{equation}
		B_2(z) := S_1 - S_1(\wtilB_1(z) + S_1)^{-1}S_1, \quad \text{on $S_1L^2 \equiv S_1(L^2(\R)\times L^2(\R))$}.
	\end{equation}
	Using the orthogonality conditions
	\begin{equation}\label{eqn: various conditions}
		\begin{split}
			&S_1D_0 = D_0 S_1 = S_1,\\
			&S_1Q= QS_1 = S_1,\\
			& QTS_1 = S_1TQ = 0,
		\end{split}
	\end{equation}
	we obtain
	\begin{equation*}
		B_2(z) = z S_1(M_1 - \kappa^{-1} T^2)S_1 + S_1\calE_2(z)S_1.
	\end{equation*}
	By Lemma~\ref{lemma: threshold resonance characterization}, we note that $S_1L^2$ is spanned by $\Phi(x)$ and that $PT\Phi = T\Phi$ holds (c.f. \eqref{eqn: PTphi = Tphi}),  whence $S_1T^2S_1 = S_1TPTS_1$. Using Lemma~\ref{lemma: threshold resonance characterization} (c.f. \eqref{eqn: value of STPTS}, \eqref{lemma: tr S_1M_1S_1}), we obtain that
	\begin{equation*}
		d := \Tr (S_1(M_1 - \kappa^{-1}T^2)S_1) = \Tr(S_1M_1S_1)-\kappa^{-1}\Tr(S_1TPTS_1) =-2i(\vert c_0 \vert^2 + \vert c_1 \vert^2)\Vert \Phi \Vert_{L^2 \times L^2}^{-2} \neq 0.
	\end{equation*}
	Hence, we apply another Neumann series expansion to invert the operator $B_2(z)$ on $S_1L^2$ for small $z$ and write
	\begin{equation}\label{eqn: B_2inv}
		B_2(z)^{-1} = \frac{1}{dz}S_1 + S_1\calE_0(z)S_1 \quad \text{on $S_1L^2$}.
	\end{equation}
	Moreover, by Lemma~\ref{lemma: JN}, we have 
	\begin{equation*}
		\wtilB_1(z)^{-1} = \big(\wtilB_1(z)+S_1 \big)^{-1} + \big(\wtilB_1(z)+S_1\big)^{-1}S_1B_2(z)^{-1}S_1\big(\wtilB_1(z)+S_1\big)^{-1} \quad \text{on $QL^2$}.
	\end{equation*}
	Using \eqref{eqn: tilB_1 + S_1 inv}, \eqref{eqn: various conditions}, and \eqref{eqn: B_2inv}, we find that 
	\begin{equation*}
		\begin{split}
			\wtilB_1(z)^{-1} &= \frac{1}{dz}S_1 + Q\calE_0(z)Q\quad \text{on $QL^2$}.
		\end{split}
	\end{equation*}
	Hence, 
	\begin{equation*}
		B_1(z)^{-1} = \frac{\kappa}{z}\wtilB_1(z)^{-1} = \frac{\kappa}{d z^2}S_1 + \frac{\kappa}{z}Q\calE_0(z)Q \quad \text{on $QL^2$}.		
	\end{equation*}
We return to the expansion of $\wtilM(z)^{-1}$ by using Lemma~\ref{lemma: JN} with \eqref{eqn: def wtilM+Q} to obtain that
	\begin{equation*}
		\begin{split}
			\wtilM(z)^{-1} &= \big(\wtilM(z)+Q\big)^{-1} +  \big(\wtilM(z)+Q\big)^{-1}QB_1(z)^{-1}Q\big(\wtilM(z)+Q\big)^{-1}\\
			&= \big(I - \frac{z}{\kappa}T\big) + \frac{\kappa}{d z^2}S_1 -\frac{1}{dz}TS_1 - \frac{1}{dz}S_1T + \frac{1}{d \kappa}TS_1T \\
			&\quad + \frac{\kappa}{z}\left(Q\calE_0(z)Q + \calE_1(z)Q + Q\calE_1(z) + \calE_2(z)\right).
		\end{split}
	\end{equation*}
	Here, we used the identity $Q = IQ = QI$. By reverting back to $M(z) = \frac{\kappa}{z}\wtilM(z)$, we have  
	\begin{equation*}
		\begin{split}
		M(z)^{-1} = \frac{z}{\kappa}\wtilM(z)^{-1} &= \frac{z}{\kappa}I + \frac{1}{d z}S_1 - \frac{1}{d\kappa}TS_1 - \frac{1}{d\kappa}S_1T + \frac{z}{d\kappa^2}TS_1T \\
&\qquad + Q\calE_0(z)Q + \calE_1(z)Q + Q\calE_1(z) + \calE_2(z).			
		\end{split}
	\end{equation*}
	Note that we absorb the $\frac{z^2}{\kappa^2}T$ term into the error $\calE_2(z)$ above.
	By using the identities $I = Q + P$, $QTS_1 = S_1TQ = 0$, and by factoring the powers of $z$ from the error terms $\calE_j(z)$, we obtain the expansion of $M(z)^{-1}$ on $L^2$: for $0 < \vert z \vert < z_0$,
	\begin{equation*}
\begin{split}
	M(z)^{-1} &= \frac{z}{\kappa}P + \frac{1}{d}\left(\frac{1}{z}S_1 - \frac{1}{\kappa} PTS_1 - \frac{1}{\kappa} S_1TP + \frac{1}{\kappa^2}PTS_1TP \right)\\
	&\qquad  + Q \Lambda_0(z) Q + zQ\Lambda_1(z) + z\Lambda_2(z)Q + z^2\Lambda_3(z),
\end{split}
	\end{equation*}
	where the operators $\Lambda_j(z)$, $j=0,\ldots,3$,  satisfy \eqref{eqn: condition on Lambda_j}. Here, we choose $z_0>0$ sufficiently small such that the expansion $\eqref{eqn: def wtilM}$ and the Neumann series inversions \eqref{eqn: Neumann1}, \eqref{eqn: Neumann2}, \eqref{eqn: B_2inv} are valid for all $0<\vert z \vert < z_0$. Finally, by Lemma~\ref{lemma: threshold resonance characterization} (c.f. \eqref{eqn: PTS_1TP}), the term $PTS_1TP$ can be simplified to $\vert c_0 \vert^2 \Vert \Phi \Vert_{L^2 \times L^2}^{-2} \Vert V_1 \Vert_{L^1(\R)}P$, which finishes the proof.
\end{proof}
\begin{remark}\label{remark: M inverse}
We appeal to the reader that each leading term in the expansion \eqref{eqn: M inverse} plays an important role in revealing the cancellations among the finite rank operators that arise in the local decay estimate \eqref{eqn: theorem local decay estimate}. Such a precise expression was also obtained for the one-dimensional Dirac operators in \cite{22ErdoganGreen}, even though the proof we give here is different. See Remark~3.7 in that paper. For the low-energy unweighted dispersive estimates, it is sufficient to work with the simpler expression 
\begin{equation}\label{eqn: M inverse weak}
	M(z)^{-1} = \frac{1}{z}Q\widetilde{\Lambda}_0(z)Q + Q\widetilde{\Lambda}_1(z) + \widetilde{\Lambda}_2(z)Q + z\widetilde{\Lambda}_3(z),
\end{equation}
where we absorb the operators $S_1,S_1TP,PTS_1,P$ in \eqref{eqn: M inverse} into the operators $Q\widetilde{\Lambda}_0(z)Q$, $Q\widetilde{\Lambda}_1(z)$, $ \widetilde{\Lambda}_2(z)Q$, $\widetilde{\Lambda}_3(z)$ respectively. The operators $\widetilde{\Lambda}_j(z)$, for $j=0,\ldots,3$, satisfy the same estimates as \eqref{eqn: condition on Lambda_j}. 
\end{remark}

\section{Low energy estimates}
In this section, we prove the low energy bounds for the perturbed evolution, following the ideas in Section~4 of \cite{22ErdoganGreen}. We will frequently exploit the crucial orthogonality condition
\begin{equation}\label{eqn: orthogonality condition}
	\int_\R \underline{e}_{11}v_1(x) Q(x,y) \, \ud x = \int_\R Q(x,y) v_2(y) \underline{e}_{11} \, \ud y= \mathbf{0}_{2\times 2}.
\end{equation}
The following calculus lemma will be helpful for dealing with the lower entry of the free resolvent kernel.
\begin{lemma}\label{lemma: bound on 2,2 resolvent}
	For any $m>0$ and $r \geq 0$, we define
	\begin{equation}
		g_m(x) := \frac{e^{-r\sqrt{x^2+m^2}}}{\sqrt{x^2+m^2}}.
	\end{equation}
	Then, there exists $C_m > 0$ (independent of $r$) such that
	\begin{equation}\label{eqn: bound on g_m}
		\Vert \partial_x^k \, g_m \Vert_{L^\infty(\R)}\leq C_m \lesssim 1, \quad \forall\ k=0,1,2.
	\end{equation}
\end{lemma}
\begin{proof}
	First, by rescaling, we set $g_m(x) = \frac{1}{m}\wtilg(x/m)$ where
	\begin{equation}
		\wtilg(x) := \frac{e^{-rm \sqrt{ x^2+1}}}{\sqrt{x^2+1}} = \frac{1}{e^{\tilr \langle x \rangle}\langle x \rangle}, \quad \tilr := rm.
	\end{equation}
	Hence, it sufficient to prove the same estimate \eqref{eqn: bound on g_m} for $\wtilg(x)$. For $k=0$, it is clear that $\vert \wtilg(x) \vert \leq 1$ for all $x \in \R$. For $k=1,2$, direct computation shows that
	\begin{equation}\label{eqnproof: partial_x tilg 1}
		\partial_x \, \wtilg(x) = -\frac{ x(1+\tilr\langle x \rangle)}{e^{\tilr\langle x \rangle }\langle x \rangle^3},
	\end{equation}
	and
	\begin{equation}\label{eqnproof: partial_x tilg 2}
		\partial_x^2\, \wtilg(x) = \frac{3x^2 + 3\tilr x^2\langle x \rangle-\langle x \rangle^2 + \tilr^2 x^2\langle x \rangle^2 - \tilr\langle x \rangle^4}{e^{\tilr\langle x \rangle}\langle x \rangle^5}.
	\end{equation}
	Since $e^{-\tilr\langle x \rangle}\max\{1,\tilr,\tilr^2\} \leq 1$, it follows from \eqref{eqnproof: partial_x tilg 1}, \eqref{eqnproof: partial_x tilg 2} that the estimate \eqref{eqn: bound on g_m} holds for $\wtilg$ and thus for $g(x)$ too.
\end{proof}
The next proposition establishes the dispersive estimates for the evolution semigroup $e^{it\calH}P_\mathrm{s}^+$ for small energies  close to the threshold $\mu$. 
\begin{proposition}\label{prop: low energy bounds} Let the assumptions of Theorem~\ref{theorem: local decay estimate} hold. Let $\chi_0(z)$ be a smooth, even, non-negative cut-off function satisfying $\chi_0(z) = 1$ for $\vert z \vert \leq \frac{z_0}{2}$ and $\chi_0(z) = 0$ for $\vert z \vert \geq z_0$, where $z_0>0$ is  given by Proposition~\ref{prop: invert M}. Then, for any $\vert t \vert \geq 1$, and $\vec{u} = (u_1,u_2) \in \calS(\R) \times \calS(\R)$, we have 
	\begin{equation}\label{eqn: low energy unweighted estimate}
		\Vert e^{it\calH}\chi_0(\calH - \mu I)P_{\mathrm{s}}^+ \vec{u}\Vert_{L^\infty(\R)\times L^\infty(\R)} \lesssim \vert t \vert^{-\frac{1}{2}}\Vert \vec{u} \Vert_{L^1(\R) \times L^1(\R)},
	\end{equation}	
	and
	\begin{equation}\label{eqn: low energy weighted estimate}
		\Vert \langle x \rangle^{-2}(e^{it\calH}\chi_0(\calH - \mu I)P_{\mathrm{s}}^+ - F_t^+ )\vec{u}\Vert_{L^\infty(\R)\times L^\infty(\R)} \lesssim \vert t \vert^{-\frac{3}{2}}\Vert\langle x \rangle^{2} \vec{u} \Vert_{L^1(\R) \times L^1(\R)},
	\end{equation}	
where $F_t^+$ is defined by 
\begin{equation}\label{eqn: formula F_t^+}
	F_t^+(x,y) = \frac{e^{it\mu}}{\sqrt{-4 \pi i t}} \vec\Psi(x) [\sigma_3 \vec\Psi(y)]^\top.
\end{equation}
\end{proposition}
We begin with the proof of the  dispersive decay estimate \eqref{eqn: low energy unweighted estimate}.
\begin{proof}[Proof of \eqref{eqn: low energy unweighted estimate}]
	 We recall the spectral representation from \eqref{eqn: stone's formula for P_s^+}:
	\begin{equation*}
		e^{it\calH}P_{\mathrm{s}}^+ = \frac{e^{it\mu}}{\pi i} \int_{\R} e^{itz^2}z\mathcal{R}_0(z) \,\ud z - \frac{e^{it\mu}}{\pi i} \int_{\R} e^{itz^2}z\mathcal{R}_0(z)v_1(M(z))^{-1}v_2\mathcal{R}_0(z)\,\ud z.
	\end{equation*}
	Note that the first term on the right is the spectral representation for the free evolution $e^{it\calH_0}P_{\mathrm{s}}^+$ and it satisfies the same estimate as \eqref{eqn: low energy unweighted estimate} thanks to Proposition~\ref{prop: free estimate}. We insert the weaker expansion \eqref{eqn: M inverse weak} for $M(z)^{-1}$ following Remark~\ref{remark: M inverse}, and write
	\begin{equation*}
		\begin{split}
			&\int_{\R} e^{itz^2}z\chi_0(z^2)\mathcal{R}_0(z)v_1(M(z))^{-1}v_2\mathcal{R}_0(z)\,\ud z\\
			&=\int_{\R} e^{itz^2}\chi_0(z^2)\mathcal{R}_0(z)v_1 Q\widetilde{\Lambda}_0(z)Q v_2\mathcal{R}_0(z)\,\ud z + \int_{\R} e^{itz^2}z\chi_0(z^2)\mathcal{R}_0(z)v_1 Q\widetilde{\Lambda}_1(z) v_2\mathcal{R}_0(z)\,\ud z \\
			&\quad +\int_{\R} e^{itz^2}z\chi_0(z^2)\mathcal{R}_0(z)v_1 \widetilde{\Lambda}_2(z)Q v_2\mathcal{R}_0(z)\,\ud z+\int_{\R} e^{itz^2}z^2\chi_0(z^2)\mathcal{R}_0(z)v_1 \widetilde{\Lambda}_3(z) v_2\mathcal{R}_0(z)\,\ud z\\
			&=: J_1 + J_2 + J_3 + J_4.
		\end{split}
	\end{equation*}
	It remains to show that
	\begin{equation}\label{eqn: bound on J_k}
		\left \Vert J_k\right \Vert_{L^1\to L^\infty} \leq C \vert t \vert^{-\frac{1}{2}}, \quad \forall \ k=1,\ldots,4 .
	\end{equation}
We treat the case for $J_1$ since the other cases follow similarly. First, we recall the kernel of $\calR_0(z)$ from \eqref{eqn: free resolvent} and write
\begin{equation}\label{eqn: def calR_1 calR_2}
	\calR_0(z)(x,y) := \calR_1(z)(x,y) + \calR_2(z)(x,y) := \frac{ie^{iz \vert x - y \vert}}{2z} \underline{e}_{11} + \frac{-e^{-\sqrt{z^2 + 2\mu}\vert x - y \vert}}{2\sqrt{z^2 + 2\mu}}   \underline{e}_{22},
\end{equation}
and we further decompose the integral $J_1$ as
\begin{equation*}
	J_1 = J_{1}^{(1,1)} + J_{1}^{(1,2)} + J_{1}^{(2,1)} + J_{1}^{(2,2)},
\end{equation*}
where
\begin{equation*}
	J_{1}^{(i,j)}(x,y) := \int_\R e^{itz^2}\chi_0(z^2)[\mathcal{R}_i(z)v_1 Q\widetilde{\Lambda}_0(z)Q v_2\mathcal{R}_j(z)](x,y)\,\ud z, \quad i,j \in \{1,2\}.
\end{equation*}
We begin with the most singular term
\begin{equation}
	J_{1}^{(1,1)}(x,y) = \int_{\R^3}  e^{itz^2 + iz(\vert x - x_1 \vert + \vert y - y_1 \vert)} \frac{\chi_0(z^2)}{(2iz)^2}  [\underline{e}_{11}v_1Q\widetilde{\Lambda}_0(z)Qv_2\underline{e}_{11}](x_1,y_1) \,\ud z \,\ud x_1 \,\ud y_1 .
\end{equation}
The orthogonality conditions \eqref{eqn: orthogonality condition} imply that 
\begin{equation}\label{eqn: Q orthogonality trick}
\int_\R e^{iz \vert x \vert} \underline{e}_{11} v_1(x_1)Q(x_1,x_2) \,\ud x_1 = 	\int_\R e^{iz \vert y \vert} Q(y_2,y_1)v_2(y_1)\underline{e}_{11} \,\ud y_1 = \mathbf{0}.
\end{equation}
Hence, writing
\begin{equation}\label{eqn: FTC for e^izr}
	e^{iz\vert x - x_1\vert} - e^{iz\vert x \vert} = iz \int_{\vert x \vert}^{\vert x - x_1 \vert}e^{izs_1} \,\ud s_1 \quad \text{and} \quad e^{iz\vert y - y_1\vert} - e^{iz\vert y \vert} = iz \int_{\vert y \vert}^{\vert y - y_1 \vert}e^{izs_2} \,\ud s_2,
\end{equation}
we obtain
\begin{equation*}
	J_{1}^{(1,1)}(x,y) = \frac{1}{4}\int_{\R^2} \int_{\vert x \vert}^{\vert x - x_1 \vert}\int_{\vert y \vert}^{\vert y - y_1 \vert}\int_{\R} e^{itz^2 + iz(s_1+s_2)} A(z,x_1,y_1) \,\ud s_1 \,\ud s_2 \,\ud x_1 \,\ud y_1 \,\ud z,
\end{equation*}
where $A(z,x_1,y_1) = \chi_0(z^2)[\underline{e}_{11}v_1Q\widetilde{\Lambda}_0(z)Qv_2\underline{e}_{11}](x_1,y_1)$, and note that $A$ is differentiable and compactly supported in $z$ due to Proposition~\ref{prop: invert M} and the compact support of $\chi_0(z^2)$. We obtain by Lemma~\ref{lemma: van der corput estimate} and the Fubini theorem that 
\begin{equation*}
	\left \vert J_{1}^{(1,1)}(x,y) \right \vert \leq C \vert t\vert^{-\frac{1}{2}}\int_{\R^2} \int_{\vert x \vert}^{\vert x - x_1 \vert}\int_{\vert y \vert}^{\vert y - y_1 \vert}\int_{\R} \left \vert \partial_z A(z,x_1,x_2)\right \vert \,\ud z \,\ud s_1 \,\ud s_2 \,\ud x_1 \,\ud y_1.
\end{equation*}
Using
\begin{equation}
	\int_{\vert x \vert}^{\vert x - x_1 \vert}\int_{\vert y \vert}^{\vert y - y_1 \vert} 1 \,\ud s_1 \,\ud s_2 \leq \vert \vert x - x_1 \vert - \vert x \vert \vert \cdot \vert \vert y - y_1 \vert - \vert y \vert \vert \lesssim \langle x_1 \rangle \langle y_1 \rangle,
\end{equation}
as well as
\begin{equation}
	\partial_z A(z,x_1,y_1) = [\underline{e}_{11}v_1Q \partial_z(\chi_0(z^2)\widetilde{\Lambda_0}(z))Qv_2\underline{e}_{11}](x_1,y_1),
\end{equation}
along with the bound \eqref{eqn: condition on Lambda_j} on $\widetilde{\Lambda}_0$, we deduce that
\begin{equation}\label{eqn: 6.16}
	\begin{split}
		&\int_{\R^2} \int_{\vert x \vert}^{\vert x - x_1 \vert}\int_{\vert y \vert}^{\vert y - y_1 \vert}\int_{\R} \left \vert \partial_z A(z,x_1,x_2)\right \vert \,\ud z \,\ud s_1 \,\ud s_2 \,\ud x_1 \,\ud y_1\\
		&\leq C\Vert Q \Vert_{L^2 \to L^2}^2 \Vert \langle x_1 \rangle v_1(x_1)\Vert_{L^2(\R)}  \Vert \langle y_1 \rangle v_2(y_1)\Vert_{L^2(\R)}  \\
		&\qquad \cdot \int_{[-z_0,z_0]}  (\Vert \vert \widetilde{\Lambda}_0(z)\vert \Vert_{L^2\times L^2 \to L^2\times L^2} + \Vert \vert \partial_z \widetilde{\Lambda}_0(z)\vert \Vert_{L^2\times L^2 \to L^2\times L^2}) \,\ud z \\
		&\lesssim 1.
	\end{split}
\end{equation}
Hence, 
\begin{equation*}
	\Vert J_{1}^{(1,1)}\Vert_{L^1 \times L^1 \rightarrow L^\infty \times L^\infty} \leq C \vert t \vert^{-\frac{1}{2}}.
\end{equation*}
 Next, we consider the least singular term 
\begin{equation}
	J_1^{(2,2)}(x,y) = \int_{\R^3} e^{itz^2}B(z,x,y,x_1,y_1) \,\ud x_1  \,\ud y_1  \,\ud z,
\end{equation}
where 
\begin{equation}
	B(z,x,y,x_1,y_1) := e^{-\sqrt{z^2+2\mu}(\vert x - x_1 \vert + \vert y - y_1 \vert)} \frac{ \chi_0(z^2) }{4(z^2+2\mu)} [\underline{e}_{22}v_1Q\widetilde{\Lambda}_0(z)Qv_2\underline{e}_{22}](x_1,y_1).
\end{equation}
By Lemma~\ref{lemma: van der corput estimate}, we have
\begin{equation}\label{eqn: bound on J_1_22}
	\vert J_{1}^{(2,2)}(x,y)\vert \leq C \vert t \vert^{-\frac{1}{2}},
\end{equation}
if we can show the uniform estimate 
\begin{equation*}
	\sup_{x,y \in \R} \int_{\R^3} \vert \partial_z B(z,x,y,x_1,y_1)\vert \,\ud z \,\ud x_1 \,\ud y_1 \lesssim 1.
\end{equation*}
By Lemma~\ref{lemma: bound on 2,2 resolvent}, we have
\begin{equation*}
	\sup_{z\in \R} \left \vert \partial_z^k \left( \frac{ e^{-\sqrt{z^2+2\mu}(\vert x - x_1 \vert + \vert y - y_1 \vert)} }{4(z^2+2\mu)}\right) \right \vert \leq C_\mu \lesssim 1, \quad k=0,1,
\end{equation*} 
uniformly in the $x,y,x_1,y_1$ variables. Hence, using the Cauchy-Schwarz inequality in the $x_1,y_1$ variables and the bound \eqref{eqn: condition on Lambda_j} on $\widetilde{\Lambda}_0$, we have 
\begin{equation*}
	\begin{split}
		&\int_{\R^3} \vert \partial_z B(z,x,y,x_1,y_1)\vert \,\ud z \,\ud x_1 \,\ud y_1 \\
		&\leq C_\mu \int_{\R^3} \left \vert (1+\partial_z)\chi_0(z^2)[\underline{e}_{22}v_1Q\widetilde{\Lambda}_0(z)Qv_2\underline{e}_{22}](x_1,y_1) \right \vert \,\ud z \,\ud x_1 \,\ud y_1\\
		&\lesssim \Vert Q \Vert_{L^2 \times L^2 \to L^2 \times L^2}^2 \Vert v_1\Vert_{L^2(\R)}  \Vert   v_2\Vert_{L^2(\R)}  \\
		&\qquad  \int_{[-z_0,z_0]}  \big(\Vert \vert \widetilde{\Lambda}_0(z)\vert \Vert_{L^2 \times L^2 \to L^2 \times L^2} + \Vert \vert \partial_z \widetilde{\Lambda}_0(z)\vert \Vert_{L^2 \times L^2 \to L^2 \times L^2}\big) \,\ud z \\
		&\lesssim 1.
	\end{split}
\end{equation*}
Hence, the bound \eqref{eqn: bound on J_1_22} is proven. The remaining terms $J_1^{(1,2)}$ and $J_1^{(2,1)}$ can be treated similarly with the same techniques, while for the remaining cases $J_2,J_3$, and $J_4$, we use the additional powers of $z$ in place of the missing $Q$ operators to obtain the same bounds \eqref{eqn: bound on J_k} as the term $J_1$. This finishes the proof of \eqref{eqn: low energy unweighted estimate}.
\end{proof}
Next, we turn to the proof of the low-energy weighted estimate \eqref{eqn: low energy weighted estimate}. 

\begin{proof}[Proof of \eqref{eqn: low energy weighted estimate}]
Recall that the threshold resonance function $\Psi = (\Psi_1,\Psi_2)^\top$ has been normalized in Theorem~\ref{theorem: local decay estimate}, which means that we need to carefully treat the constants relating to the function $\Phi$ where $\Phi := v_2\Psi$. By Lemma~\ref{lemma: threshold resonance characterization}, note that $\Phi$ spans the subspace $S_1(L^2(\R)\times L^2(\R))$. We define 
\begin{equation}
	\eta := \Vert \Phi \Vert_{L^2(\R) \times L^2(\R)}^{-2} \neq 0,
\end{equation}
so that $S_1(x,y) =  \eta\Phi(x)\Phi^*(y)$, and we fix the constants $c_0$ and $c_1$ defined by \eqref{eqn: def c_0} and \eqref{def: c_1} respectively for this $\Phi$. By Lemma~\ref{lemma: threshold resonance characterization}, one finds the relation
\begin{equation}
	2 = \lim_{x \to \infty}(\vert \Psi_1(x)\vert^2 + \vert \Psi_1(-x)\vert^2) = 2(\vert c_0 \vert^2 + \vert c_1 \vert^2),
\end{equation}
by the polarization identity (c.f.~\eqref{eqn: asymptotics of Psi_1}). Thus, the precise expansion \eqref{eqn: M inverse} of $M(z)^{-1}$ from Proposition~\ref{prop: invert M} simplifies to 
\begin{equation}
	\begin{split}
	M(z)^{-1} &= \frac{i}{2\eta z}S_1 + \frac{1}{\eta \Vert V_1 \Vert_{L^1(\R)}} PTS_1 + \frac{1}{\eta \Vert V_1 \Vert_{L^1(\R)}} S_1TP + \left(\frac{2i}{\Vert V_1 \Vert_{L^1(\R)}} + \frac{2 \vert c_0 \vert^2}{i\Vert V_1 \Vert_{L^1(\R) }} \right)zP		\\
	&\quad +  Q \Lambda_0(z) Q + zQ\Lambda_1(z) + z\Lambda_2(z)Q + z^2\Lambda_3(z), \qquad 0 < \vert z \vert < z_0.
	\end{split}
\end{equation}
We insert the above expression into the spectral representation of $e^{it\calH}\chi_0(\calH - \mu I)P_{\mathrm{s}}^+$, and obtain that
\begin{equation} \label{eqn: spectral representation low energy}
	\begin{split}
		&e^{it\calH}\chi_0(\calH - \mu I)P_{\mathrm{s}}^+  \\
		&= \frac{e^{it\mu}}{\pi i}\int_{\R}e^{itz^2}z\chi_0(z^2)\calR_0(z) \, \ud z - \frac{e^{it\mu}}{\pi i}\int_{\R}e^{itz^2}z\chi_0(z^2)\calR_0(z)v_1(M(z))^{-1}v_2\calR_0(z) \, \ud z   \\
		&= \frac{e^{it\mu}}{\pi i}I_1 \\
		&\quad -\frac{e^{it\mu}}{\pi i}\left( \frac{i}{2\eta}I_{2,1} + \frac{1}{\eta\Vert V_1 \Vert_{L^1(\R)}}I_{2,2} + \frac{1}{\eta\Vert V_1 \Vert_{L^1(\R)}} I_{2,3} + \left(\frac{2i}{\Vert V_1 \Vert_{L^1(\R)}} + \frac{2 \vert c_0 \vert^2}{i\Vert V_1 \Vert_{L^1(\R) }} \right)I_{2,4}\right)\\
		&\quad -\frac{e^{it\mu}}{\pi i}\left(I_{3,1} + I_{3,2} + I_{3,3} + I_{3,4}\right),
	\end{split}
\end{equation}
where
\begin{align}
&I_1 := \int_\R e^{itz^2}z \chi_0(z^2) \calR_0(z)  \,\ud z,\label{eqn: I_1 def}\\
&I_{2,1} := \int_\R e^{itz^2} \chi_0(z^2) [\calR_0(z)v_1 S_1 v_2\calR_0(z)] \,\ud z,\label{eqn: I_2 def}\\
&I_{2,2} := \int_\R e^{itz^2} z\chi_0(z^2) [\calR_0(z)v_1 S_1 T P  v_2\calR_0(z)] \,\ud z,\\
&I_{2,3} := \int_\R e^{itz^2} z\chi_0(z^2) [\calR_0(z)v_1 P T S_1  v_2\calR_0(z)] \,\ud z,\\
&I_{2,4} := \int_\R e^{itz^2} z^2\chi_0(z^2) [\calR_0(z)v_1   P  v_2\calR_0(z)] \,\ud z,
\end{align}
and
\begin{align}
&I_{3,1} := \int_\R e^{itz^2} z\chi_0(z^2) [\calR_0(z)v_1  Q \Lambda_0 (z) Q  v_2\calR_0(z)] \,\ud z,\label{eqn: I_{4,1} def}\\
&I_{3,2} := \int_\R e^{itz^2} z^2\chi_0(z^2) [\calR_0(z)v_1 Q \Lambda_1(z) v_2\calR_0(z)] \,\ud z,\\
&I_{3,3} := \int_\R e^{itz^2} z^2\chi_0(z^2) [\calR_0(z)v_1  \Lambda_2(z)Q  v_2\calR_0(z)] \,\ud z,\\
&I_{3,4} := \int_\R e^{itz^2} z^3\chi_0(z^2) [\calR_0(z)v_1  \Lambda_3(z)  v_2\calR_0(z)] \,\ud z\label{eqn: I_{4,4} def}.
\end{align}
Now we study the local decay of the terms $I_1$, $I_{2,j}$, $I_{3,\ell}$, for $j,\ell \in \{1,\ldots,4\}$ and we will observe in the following propositions that the terms  $I_1,I_{2,1},\ldots,I_{2,4}$ contribute to the leading order for the local decay estimate while the remainder terms $I_{3,1}, \ldots, I_{3,4}$ satisfy the stronger local decay estimate  $\calO(\vert t \vert^{-\frac{3}{2}}\langle x \rangle \langle y \rangle)$. We first handle these remainder terms by Lemma~\ref{lemma: van der corput estimate} in a similar spirit to the proof for the (unweighted) dispersive bound \eqref{eqn: low energy unweighted estimate}, exploiting the additional power of $z$.
\begin{proposition}\label{prop: I_4 estimate}
	For $i\in \{1,2,\ldots,4\}$ and $\vert t\vert \geq 1$, we have
	\begin{equation}
		\vert I_{3,i}(x,y) \vert \leq C \vert t\vert^{-\frac{3}{2}}\langle x \rangle \langle y \rangle. 
	\end{equation}
\end{proposition}
\begin{proof}
	We treat the case for $I_{3,1}$ as the other cases follow similarly by using the additional powers of $z$ in place of the missing operators $Q$. As before, we consider the decomposition 
	\begin{equation*}
		I_{3,1}  = I_{3,1}^{(1,1)} + I_{3,1}^{(1,2)} + I_{3,1}^{(2,1)} + I_{3,1}^{(2,2)},
	\end{equation*}
	where 
	\begin{equation*}
		I_{3,1}^{(i,j)} := \int_{\R} e^{itz^2}z\chi_0(z^2)[\calR_i(z)v_1Q\Lambda_0(z)Qv_2\calR_j(z)] \,\ud z, \quad i,j\in \{1,2\},
	\end{equation*}
	with $\calR_{1}$ and $\calR_2$ defined in \eqref{eqn: def calR_1 calR_2}. We begin with the term
	\begin{equation*}
		I_{3,1}^{(1,1)}(x,y) = \int_{\R^3}  e^{itz^2 + iz(\vert x - x_1 \vert + \vert y - y_1 \vert)} \frac{z\chi_0(z^2)}{(2iz)^2}  [\underline{e}_{11}v_1Q  {\Lambda}_0(z)Qv_2\underline{e}_{11}](x_1,y_1) \,\ud z \,\ud x_1 \,\ud y_1.
	\end{equation*}
	 Using the orthogonality condition \eqref{eqn: orthogonality condition} like in \eqref{eqn: FTC for e^izr}, we obtain
	\begin{equation*}
		\begin{split}
			I_{3,1}^{(1,1)}(x,y) &= \frac{1}{4}\int_{\R^2}\int_{\vert x \vert}^{\vert x - x_1 \vert} \int_{\vert y \vert}^{\vert y - y_1\vert} \int_{\R}e^{itz^2+iz(s_1+s_2)}z A(z,x_1,y_1) \,\ud z \,\ud s_1 \,\ud s_2 \,\ud x_1 \,\ud y_1,
		\end{split}
	\end{equation*}
	where $A(z,x_1,y_1) :=  \chi_0(z^2)[\underline{e}_{11}v_1Q\Lambda_0(z)v_2Q\underline{e}_{11}](x_1,y_1)$.  By Lemma~\ref{lemma: van der corput estimate}, we obtain that
	\begin{equation}\label{eqnproof: I_41 1}
		\begin{split}
		&\left\vert I_{3,1}^{(1,1)}(x,y) \right\vert \\
		&\lesssim \vert t \vert^{-\frac{3}{2}} \int_{\R^2}\int_{\vert x \vert}^{\vert x - x_1 \vert} \int_{\vert y \vert}^{\vert y - y_1\vert}\int_{[-z_0,z_0]}  \big(\vert \partial_z^2 A  \vert + (s_1+s_2)\vert \partial_z A \vert + \vert A \vert \big) \,\ud z \,\ud s_1 \,\ud s_2 \,\ud x_1 \,\ud y_1.
		\end{split}
	\end{equation}
	Using the bounds
	\begin{equation}\label{eqn: bound on s_1,s_2}
		\begin{split}
			&\int_{\vert x \vert}^{\vert x - x_1 \vert} \int_{\vert y \vert}^{\vert y - y_1\vert} 1 \,\ud s_1 \,\ud s_2  \lesssim \langle x_1 \rangle \langle y_1 \rangle, \\ 
			&\int_{\vert x \vert}^{\vert x - x_1 \vert} \int_{\vert y \vert}^{\vert y - y_1\vert} (s_1+s_2) \,\ud s_1 \,\ud s_2  \lesssim \langle x_1 \rangle^2 \langle y_1 \rangle^2 \langle x \rangle \langle y \rangle, 
		\end{split}
	\end{equation}
	we have
	\begin{equation*}
		\left\vert I_{3,1}^{(1,1)} (x,y) \right\vert \lesssim \vert t \vert^{-\frac{3}{2}} \int_{\R^2} \int_{[\vert z \vert \leq z_0]} \langle x_1 \rangle \langle y_1 \rangle (\vert \partial_z^2 A  \vert + \langle x_1 \rangle \langle y_1 \rangle \langle x \rangle \langle y \rangle\vert \partial_z A \vert + \vert A \vert ) \,\ud z \,\ud x_1 \,\ud y_1.
	\end{equation*}
	Noting that $\langle x\rangle v_1(x_1)$ and $\langle y_1 \rangle v_2(y_1)$ are in $L^2$ and that $\Lambda_0$ satisfies the bound \eqref{eqn: condition on Lambda_j}, we apply Cauchy-Schwarz inequality in $x_1$ and $y_1$ variables to obtain the bound
	\begin{equation}\label{eqnproof: I_41 4}
		\begin{split}
			\vert I_{3,1}^{(1,1)}(x,y)  \vert &\lesssim \vert t \vert^{-\frac{3}{2}} \Vert Q \Vert_{L^2 \to L^2}^2 \Vert \langle x_1 \rangle v_1 \Vert_{L_{x_1}^2(\R) }\Vert \langle y_1 \rangle v_2 \Vert_{L_{y_1}^2(\R) }\\
			&\qquad \cdot \int_{[\vert z \vert \leq z_0]} (\Vert \vert \partial_z^2 \Lambda_0(z) \vert \Vert_{L^2 \times L^2 \to L^2 \times L^2} + \Vert \vert \Lambda_0(z) \vert \Vert_{L^2 \times L^2\to L^2\times L^2} ) \,\ud z \\
			&\quad +\vert t \vert^{-\frac{3}{2}}\langle x \rangle \langle y \rangle \Vert Q \Vert_{L^2 \to L^2}^2\Vert \langle x_1 \rangle v_1 \Vert_{L_{x_1}^2(\R) }\Vert \langle y_1 \rangle v_2 \Vert_{L_{y_1}^2(\R) }  \\
			&\qquad \cdot \int_{[\vert z \vert \leq z_0]} \Vert \vert \partial_z \Lambda_0(z) \vert \Vert_{L^2 \times L^2 \to L^2\times L^2 } \,\ud z 		\\
			&\lesssim \vert t \vert^{-\frac{3}{2}}\langle x \rangle \langle y \rangle .
		\end{split}
	\end{equation}
	Next, we consider the term 
	\begin{equation*}
		I_{3,1}^{(1,2)}(x,y) = \int_{\R^3}  e^{itz^2 + iz\vert x - x_1 \vert - \sqrt{z^2+2\mu}\vert y - y_1 \vert} \frac{\chi_0(z^2)}{4i\sqrt{z^2+2\mu}}  [\underline{e}_{11}v_1Q  {\Lambda}_0(z)Qv_2\underline{e}_{22}](x_1,y_1) \,\ud z \,\ud x_1 \,\ud y_1.
	\end{equation*}
	By using the $Q$ orthogonality (c.f. \eqref{eqn: orthogonality condition}) condition, we write 
	\begin{equation}
		\begin{split}
		I_{3,1}^{(1,2)}(x,y) = \int_{\R^3}\int_{\vert x \vert}^{\vert x - x_1 \vert }  e^{itz^2 + izs_1} zB(z,x_1,y_1,x,y) \,\ud s_1 \,\ud z \,\ud x_1 \,\ud y_1,
		\end{split}
	\end{equation}
	where
	\begin{equation}
		B(z,x_1,y_1,x,y) := \frac{e^{- \sqrt{z^2+2\mu}\vert y - y_1 \vert}}{4i\sqrt{z^2+2\mu}} \chi_0(z^2) [\underline{e}_{11}v_1Q  {\Lambda}_0(z)Qv_2\underline{e}_{22}](x_1,y_1) .
	\end{equation}
	Since $B$ is compactly supported in $z$, we can exchange the order of integration and we use Lemma~\ref{lemma: van der corput estimate} to obtain 
	\begin{equation*}
		\left\vert I_{3,1}^{(1,2)}(x,y) \right\vert \leq C \vert t \vert^{-\frac{3}{2}} \int_{\R^2}\int_{\vert x \vert}^{\vert x - x_1\vert} \int_{\R} \vert [\partial_z^2 + is_1 \partial_z] B(z,x_1,y_1,x,y) \vert \,\ud z \,\ud s_1 \,\ud x_1 \,\ud y_1.
	\end{equation*}
	By Lemma~\ref{lemma: bound on 2,2 resolvent}, we have
	\begin{equation}
		\sup_{z \in \R} \left \vert \partial_z^k  \left(\tfrac{e^{- \sqrt{z^2+2\mu}\vert y - y_1 \vert}}{4i\sqrt{z^2+2\mu}}\right) \right \vert \leq C_\mu \lesssim 1, \quad \forall k=0,1,2,
	\end{equation}
	which implies by H\"older's inequality and Leibniz rule that 
	\begin{equation*}
		\begin{split}
			&\int_{\R} \left \vert [\partial_z^2 + is_1 \partial_z] B(z,x_1,y_1,x,y) \right \vert \,\ud z  \\
			&\quad \leq C \langle s_1 \rangle \int_{\R}  \left\vert \underline{e}_{11}v_1Q  [1+\partial_z + \partial_z^2](\chi_0(z^2){\Lambda}_0(z))Qv_2\underline{e}_{22} \right\vert \,\ud z.
		\end{split}
	\end{equation*}
	Repeating the arguments from \eqref{eqnproof: I_41 1}--\eqref{eqnproof: I_41 4}, we obtain 
	\begin{equation*}
		\left\vert I_{3,1}^{(1,2)}(x,y) \right\vert \leq C \vert t \vert^{-\frac{3}{2}}\langle x \rangle.
	\end{equation*}
	Similarly, one has the bounds
	\begin{equation*}
		\left\vert I_{3,1}^{(2,1)}(x,y) \right\vert \leq C \vert t \vert^{-\frac{3}{2}}\langle y \rangle, \qquad \left\vert I_{3,1}^{(2,2)}(x,y) \right\vert \leq C \vert t \vert^{-\frac{3}{2}},
	\end{equation*}
	and we are done.
\end{proof}

\begin{proposition}\label{prop: I_2 estimate} For all $ \vert t \vert \geq 1$, we have 
\begin{equation}\label{eqn: I_{2,1} estimate}
	\left\vert I_{2,1}(x,y) - F_t^{1}(x,y) \right\vert \leq C \vert t \vert^{-\frac{3}{2}}\langle x\rangle^2 \langle y \rangle^2,
\end{equation}	
where 
\begin{equation}\label{eqn: def F_t1}
	F_t^1(x,y) := \frac{\eta\sqrt{\pi}}{{\sqrt{-it}}} [c_0 \underline{e}_1 - \Psi(x)][\sigma_3 \Psi(y) - {c_0} \underline{e}_1]^*.
\end{equation}
\end{proposition}
\begin{proof} As in the previous propositions, we decompose $I_{2,1}$ into the sum
\begin{equation*}
	I_{2,1} = I_{2,1}^{(1,1)}+I_{2,1}^{(1,2)}+I_{2,1}^{(2,1)}+I_{2,1}^{(2,2)},
\end{equation*}
with 
\begin{equation*}
	I_{2,1}^{(i,j)} := \int_\R e^{itz^2} \chi_0(z^2) [\calR_i(z)v_1 S_1 v_2\calR_j(z)] \,\ud z, \quad i,j \in \{1,2\}.
\end{equation*}
We start with the most singular term 
\begin{equation*}
	I_{2,1}^{(1,1)}(x,y) = \int_{\R^3} e^{itz^2 + iz(\vert x - x_1 \vert + \vert y - y_1 \vert)} \frac{\chi_0(z^2)}{(2iz)^2}  [\underline{e}_{11}v_1S_1v_2\underline{e}_{11}](x_1,y_1) \,\ud x_1 \,\ud y_1 \,\ud z.
\end{equation*}
Noting that $S_1L^2 \subset QL^2$, the orthogonality conditions \eqref{eqn: orthogonality condition} imply that 
\begin{equation}\label{eqn: S_1 orthogonality trick}
	\int_\R e^{iz \vert x \vert} \underline{e}_{11} v_1(x_1)S_1(x_1,x_2)\,\ud x_1 = 	\int_\R e^{iz \vert y \vert} S_1(y_2,y_1)v_2(y_1)\underline{e}_{11}\,\ud y_1 = \mathbf{0}_{2\times 2}, \quad \forall x, y \in \R.
\end{equation}
Hence, by the Fubini theorem,
\begin{equation*}
	\begin{split}
		I_{2,1}^{(1,1)} (x,y)
		&= \frac{1}{4} \int_{\R^2} \int_{\vert x \vert}^{\vert x - x_1 \vert}\int_{\vert y \vert}^{\vert y - y_1 \vert} \int_{\R} e^{itz^2 + iz(s_1 + s_2)} \chi_0(z^2) [\underline{e}_{11}v_1S_1v_2\underline{e}_{11}](x_1,y_1) \,\ud z \,\ud s_1 \,\ud s_2 \,\ud x_1 \,\ud y_1\\
		&= \frac{1}{4}  \int_{\vert x \vert}^{\vert x - x_1 \vert}\int_{\vert y \vert}^{\vert y - y_1 \vert} G_t(s_1+s_2) \,\ud s_1 \,\ud s_2 \int_{\R^2} [\underline{e}_{11}v_1S_1v_2\underline{e}_{11}](x_1,y_1)  \,\ud x_1 \,\ud y_1,
	\end{split}
\end{equation*}
where $G_t(\cdot)$ is the function defined in Lemma \ref{lemma: stationary phase fresnel}, which satisfies the estimate
\begin{equation}\label{eqn: G_t(s_1+s_2)}
	\left \vert G_t(s_1+s_2) - \frac{\sqrt{\pi}}{{\sqrt{-it}}}  e^{-i\frac{s_1^2}{4t}}e^{-i\frac{s_2^2}{4t}} \right\vert \leq C \vert t \vert^{-\frac{3}{2}} \langle s_1 \rangle \langle s_2 \rangle.
\end{equation}
Using the bound
\begin{equation}\label{eqn: s_1 s_2 angle}
	\int_{\vert x \vert}^{\vert x - x_1 \vert}\int_{\vert y \vert}^{\vert y - y_1 \vert}  \langle s_1 \rangle \langle s_2 \rangle \,\ud s_1 \,\ud s_2 \lesssim \langle x_1\rangle^2\langle y_1\rangle^2\langle x \rangle \langle y \rangle, 
\end{equation}
the decay assumptions on $v_1,v_2$, and the estimate \eqref{eqn: G_t(s_1+s_2)}, we have
\begin{equation*}
	\begin{split}
		&\left \vert I_{2,1}^{(1,1)}(x,y) -  \frac{\sqrt{\pi}}{{4\sqrt{-it}}}  e^{i \frac{\pi}{4}} \int_{\R^2}H_t(x_1,x)[\underline{e}_{11}v_1S_1v_2\underline{e}_{11}](x_1,y_1)H_t(y_1,y)\,\ud x_1 \,\ud y_1\right \vert\\
		&\quad \leq C \vert t \vert^{-\frac{3}{2}}\langle x \rangle \langle y \rangle \Vert S_1 \Vert_{L^2 \times L^2 \to L^2 \times L^2 } \Vert \langle x_1 \rangle^2 v_1(x_1) \Vert_{L^2} \Vert \langle y_1 \rangle^2 v_2(y_2) \Vert_{L^2} \leq C\vert t \vert^{-\frac{3}{2}}\langle x \rangle \langle y \rangle,
	\end{split}
\end{equation*}
where we set
\begin{equation}
	H_t(x_1,x) := \int_{\vert x \vert}^{\vert x_1 - x \vert} e^{-i\frac{s^2}{4t}} \,\ud s.
\end{equation}
Since $S_1(x,y) = \eta\Phi(x)\Phi^*(y)$, the orthogonality conditions \eqref{eqn: orthogonality condition} imply that
\begin{equation*}
	\begin{split}
		\int_{\R} \vert x \vert \underline{e}_{11} v_1(x_1) S_1(x_1,y_1) \, \ud x_1 &= \eta \int_\R \vert x \vert \underline{e}_{11}v_1(x_1)\Phi(x_1) \,\ud x_1\Phi^*(y_1) = \bm{0}_{2\times 2}, \quad \forall y \in \R, \\
		\int_{\R} \vert y \vert  S_1(x_1,y_1) v_2(y_1)\underline{e}_{11} \, \ud y_1&= 	\eta \Phi(x_1)\int_\R \vert y \vert \Phi^*(y_1)v_2(y_1)\underline{e}_{11} \,\ud y_1 = \bm{0}_{2 \times 2}, \quad \forall x \in \R.
	\end{split}
\end{equation*}
Hence, using the bound
\begin{equation}
	\begin{split}
		\vert H_t(x_1,x) - (\vert x - x_1 \vert - \vert x \vert )\vert \leq C \vert t \vert^{-1}\langle x \rangle^2 \langle x_1 \rangle^3,
	\end{split}
\end{equation}
and the exponential decay of $v_1,v_2$, we conclude the estimate 
\begin{equation}\label{eqn: I_2^{(1,1)}}
	\begin{split}
		\left \vert I_{2,1}^{(1,1)}(x,y) - \frac{\eta\sqrt{\pi}}{{\sqrt{-it}}}  [G_0(\underline{e}_{11}v_1\Phi)(x)][G_0(\underline{e}_{11}v_2\Phi)(y)]^* \right \vert \leq C\vert t \vert^{-\frac{3}{2}}\langle x \rangle^2 \langle y \rangle^2,
	\end{split}
\end{equation}
where \begin{equation}\label{eqn: def of G_0}
	G_0(x,y) := -\frac{1}{2}\vert x - y\vert,
\end{equation} 
and
\begin{equation*}
	\begin{split}
		&[G_0(\underline{e}_{11}v_1\Phi)(x)] := -\frac{1}{2}\int_{\R} \vert x - x_1 \vert \underline{e}_{11}v_1(x_1)\Phi(x_1)\,\ud x_1,\\
		&[G_0(\underline{e}_{11}v_2\Phi)(y)]^* :=  -\frac{1}{2}\int_{\R} \vert y - y_1 \vert \Phi^*(y_1)v_2(y_1) \underline{e}_{11} \,\ud y_1.
	\end{split}
\end{equation*}
In the preceding definition, we used the identity $v_2^* = v_2$. Next, we treat the term 
\begin{equation*}
	I_{2,1}^{(2,2)}(x,y) = \int_{\R^3} e^{itz^2}\chi_0(z^2)\frac{e^{-\sqrt{z^2+2\mu}\vert x - x_1\vert }}{-2\sqrt{z^2+2\mu}} [\underline{e}_{22}v_1S_1v_2\underline{e}_{22}](x_1,y_1)\frac{e^{-\sqrt{z^2+2\mu}\vert y - y_1\vert }}{-2\sqrt{z^2+2\mu}}\,\ud x_1 \,\ud y_1 \,\ud z.
\end{equation*}
By Taylor expansion, we have 
\begin{equation}\label{eqn: I_{2,1}{2,2}}
	\begin{split}
		I_{2,1}^{(2,2)}(x,y) &=\int_{\R^3} e^{itz^2}\chi_0(z^2)\frac{e^{-\sqrt{2\mu}\vert x - x_1\vert }}{-2\sqrt{2\mu}} [\underline{e}_{22}v_1S_1v_2\underline{e}_{22}](x_1,y_1)\frac{e^{-\sqrt{2\mu}\vert y - y_1\vert }}{-2\sqrt{2\mu}}\,\ud x_1 \,\ud y_1 \,\ud z \\
		&\quad + \int_{\R^3} e^{itz^2}z^2\chi_0(z^2)[\underline{e}_{22}v_1S_1v_2\underline{e}_{22}](x_1,y_1)\kappa(x,x_1)\kappa(y,y_1)\,\ud x_1 \,\ud y_1 \,\ud z \\
		&= \eta \int_\R e^{itz^2}\chi_0(z^2) \,\ud z [G_2(\underline{e}_{22}v_1\Phi)(x)][G_2(\underline{e}_{22}v_2\Phi)(y)]^*\\
		&\quad  + \int_{\R^3} e^{itz^2}z^2\chi_0(z^2)[\underline{e}_{22}v_1S_1v_2\underline{e}_{22}](x_1,y_1)\kappa(x,x_1)\kappa(y,y_1)\,\ud x_1 \,\ud y_1\,\ud z,
	\end{split}
\end{equation}
where we set
\begin{equation}\label{eqn: G_2(x,y)}
	G_2(x,y) := \frac{e^{-\sqrt{2\mu}\vert x - y\vert }}{-2\sqrt{2\mu}},
\end{equation}
and where $\kappa(x,x_1)\kappa(y,y_1)$ is an error term bounded by $C\langle x \rangle \langle x_1 \rangle \langle y \rangle \langle y_1 \rangle e^{-c(\vert x - x_1 \vert + \vert y - y_1 \vert)}$, for some $C,c >0$, (c.f. \eqref{eqn: taylor expand r_2}). The definitions for $G_2(\underline{e}_{22}v_1\Phi)(x)$ and $G_2(\underline{e}_{22}v_2\Phi)(y)$ are defined analogously to the ones for $G_0(\underline{e}_{11}v_1\Phi)(x)$ and $G_0(\underline{e}_{11}v_2\Phi)(y)$. By non-stationary phase, one has the uniform estimate
\begin{equation}
	\left \vert \int_\R e^{itz^2}z^2\chi_0(z^2) \,\ud z \right \vert \leq C\vert t \vert^{-\frac{3}{2}}.
\end{equation}
Hence, we can control the remainder term in $I_{2,1}^{(2,2)}$ by 
\begin{equation}\label{eqn: bound on error I_{2,1}^{2,2}}
	\left \vert \int_{\R^3} e^{itz^2}z^2\chi_0(z^2)[\underline{e}_{22}v_1S_1v_2\underline{e}_{22}](x_1,y_1)\kappa(x,x_1)\kappa(y,y_1)\,\ud x_1 \,\ud y_1 \,\ud z \right \vert \leq C \vert t\vert^{-\frac{3}{2}}\langle x \rangle \langle y \rangle.
\end{equation}
On the other hand, by Lemma~\ref{lemma: stationary phase fresnel}, one has
\begin{equation*}
	\int_{\R} e^{itz^2}\chi_0(z^2)\,\ud z = \frac{\sqrt{\pi}}{{\sqrt{-it}}}  + R_t, \quad \vert R_t \vert \leq C\vert t \vert^{-\frac{3}{2}}.
\end{equation*}
Hence, the leading contribution of $I_{2,1}^{(2,2)}$ can be written as 
\begin{equation*}
	\left \vert \int_\R e^{itz^2}\chi_0(z^2) \,\ud z [G_2(\underline{e}_{22}v_1\Phi)(x)][G_2(\underline{e}_{22}v_2\Phi)(y)]^* - \frac{\eta\sqrt{\pi}}{{\sqrt{-it}}} [G_2(\underline{e}_{22}v_1\Phi)(x)][G_2(\underline{e}_{22}v_2\Phi)(y)]^*\right \vert \leq C\vert t \vert^{-\frac{3}{2}}.
\end{equation*}
Thus, one concludes the estimate for  $I_{2,1}^{(2,2)}$:
\begin{equation}\label{eqn: I_2^{(2,2)}}
	\left \vert I_{2,1}^{(2,2)} - \frac{\eta\sqrt{\pi}}{{\sqrt{-it}} } [G_2(\underline{e}_{22}v_1\Phi)(x)][G_2(\underline{e}_{22}v_2\Phi)(y)]^* \right \vert \leq C \vert t \vert^{-\frac{3}{2}}\langle x \rangle \langle y \rangle.
\end{equation}
Finally, we note that a similar analysis holds for the terms $I_{2,1}^{(1,2)}$ and $I_{2,1}^{(2,1)}$ yielding the contributions
\begin{equation}\label{eqn: I_2^{(1,2)} and I_{2,1}^{(1,2)}}
	\begin{split}
		&\left \vert 	I_{2,1}^{(1,2)} - \frac{\eta\sqrt{\pi}}{{\sqrt{-it}} }[G_0(\underline{e}_{11}v_1\Phi)(x)][G_2(\underline{e}_{22}v_2\Phi)(y)]^* \right \vert \leq C\vert t \vert^{-\frac{3}{2}}\langle x \rangle^2 \langle y \rangle,\\
		&\left \vert I_{2,1}^{(2,1)} -\frac{\eta\sqrt{\pi}}{{\sqrt{-it}} }[G_2(\underline{e}_{22}v_1\Phi)(x)][G_0(\underline{e}_{11}v_2\Phi)(y)]^*\right \vert \leq C\vert t \vert^{-\frac{3}{2}}\langle x \rangle \langle y \rangle^2.
	\end{split}
\end{equation}
By adding all leading order contributions, we obtain
\begin{equation*}
	F_{t}^1(x,y) =\frac{\eta \sqrt{\pi}}{{\sqrt{-it}} }[(G_0\underline{e}_{11} + G_2\underline{e}_{22})v_1\Phi](x)[(G_0\underline{e}_{11} + G_2\underline{e}_{22})v_2\Phi]^*(y).
\end{equation*}
Recalling that $\calG_0 = G_0 \underline{e}_{11} + G_2\underline{e}_{22}$ from Lemma \ref{eqn: expansion of R_0}, that $\calG_0(v_1 \Phi ) = c_0\underline{e}_1 - \Psi$ from Lemma \ref{lemma: threshold resonance characterization}, and that $\calG_0(v_2 \Phi) = \sigma_3 \Psi - c_0 \underline{e}_1$ from Remark \ref{remark: psi_2} (c.f.~\eqref{eqn: sigma_3Psi}), we arrive at  
\begin{equation*}
	F_t^1(x,y) = \frac{\eta\sqrt{\pi}}{{\sqrt{-it}} }[c_0 \underline{e}_1 - \Psi(x)][\sigma_3 \Psi(y) - {c_0} \underline{e}_1]^*,
\end{equation*}
as claimed
\end{proof}
We continue the analysis for the terms involving the operators $S_1TP$ and $PTS_1$.
\begin{proposition}\label{prop: I_{2,3}}
For all $\vert t \vert \geq 1$, we have
\begin{equation}\label{eqn: I_{2,2}}
	\vert I_{2,2}(x,y) - F_t^{2}(x,y) \vert \leq C \vert t\vert^{-\frac{3}{2}}\langle x \rangle^2 \langle y \rangle^2,
\end{equation}
\begin{equation}\label{eqn: I_{2,3}}
	\vert I_{2,3}(x,y) - F_t^{3}(x,y) \vert \leq C \vert t\vert^{-\frac{3}{2}}\langle x \rangle^2 \langle y \rangle^2,
\end{equation}
where
\begin{equation}\label{eqn: F_t^2}
	F_{t}^{2}(x,y) := \frac{i\eta\Vert V_1 \Vert_{L^1(\R)}}{2}\frac{\sqrt{\pi}}{{\sqrt{-it}}}[c_0 \underline{e}_1 - \Psi(x)][e^{i\frac{y^2}{4t}} {c_0}\underline{e}_1]^*,
\end{equation}
\begin{equation}\label{eqn: F_t^3}
	F_{t}^{3}(x,y) := -\frac{i\eta\Vert V_1 \Vert_{L^1(\R)}}{2}\frac{\sqrt{\pi}}{{\sqrt{-it}}} [e^{-i\frac{x^2}{4t}}c_0 \underline{e}_{1}][\sigma_3 \Psi(y) -  {c_0} \underline{e}_1]^*.
\end{equation}
	
\end{proposition}
\begin{proof}
As in the proof of Proposition~\ref{prop: I_2 estimate}, we decompose $I_{2,2}$ into
\begin{equation*}
	I_{2,2} = I_{2,2}^{(1,1)} + I_{2,2}^{(1,2)} + I_{2,2}^{(2,1)} + I_{2,2}^{(2,2)},
\end{equation*}
with
\begin{equation*}
	I_{2,2}^{(i,j)} := \int_{\R} e^{itz^2} z \chi_0(z^2) [\calR_i(z) v_1 S_1TP v_2 \calR_{j}(z)] \,\ud z, \quad i,j \in \{1,2\},
\end{equation*}
where $\calR_1$ and $\calR_{2}$ were defined in \eqref{eqn: def calR_1 calR_2}. We start with 
\begin{equation*}
	I_{2,2}^{(1,1)}(x,y) = \int_{\R^3} e^{itz^2}z\chi_0(z^2) \frac{e^{iz \vert x - x_1 \vert}}{-2iz}[\underline{e}_{11}v_1S_1TPv_2\underline{e}_{11}](x_1,y_1)\frac{e^{iz \vert y - y_1 \vert}}{-2iz}  \,\ud x_1 \,\ud y_1 \,\ud z.
\end{equation*}
Using the orthogonality \eqref{eqn: S_1 orthogonality trick}, we have
\begin{equation*}
	\begin{split}
	I_{2,2}^{(1,1)}(x,y) &= \frac{1}{4} \int_{\R^3}  \int_{\vert x \vert}^{\vert x - x_1 \vert}\int_{\vert y \vert}^{\vert y - y_1 \vert} e^{itz^2+iz(s_1+s_2)}z\chi_0(z^2) [\underline{e}_{11}v_1S_1TPv_2\underline{e}_{11}](x_1,y_1) \,\ud s_1 \,\ud s_2 \,\ud x_1 \,\ud y_1 \,\ud z\\
	&\quad + \frac{1}{4i} \int_{\R^3}  \int_{\vert x \vert}^{\vert x - x_1 \vert}e^{itz^2+izs_1}\chi_0(z^2) [\underline{e}_{11}v_1S_1TPv_2\underline{e}_{11}](x_1,y_1)e^{iz \vert y \vert}\,\ud s_1 \,\ud x_1 \,\ud y_1 \,\ud z\\
	&=: I_{2,2;1}^{(1,1)} + I_{2,2;2}^{(1,1)}.
	\end{split}
\end{equation*}
By Lemma~\ref{lemma: van der corput estimate}, we have 
\begin{equation*}
	\left \vert \int_{\R}e^{itz^2+iz(s_1+s_2)}z\chi_0(z^2) \,\ud z \right \vert \leq C\vert t \vert^{-\frac{3}{2}}\langle s_1 \rangle \langle s_2 \rangle.
\end{equation*}
Using this estimate, the bound 
\begin{equation*}
	\int_{\vert x \vert}^{\vert x - x_1 \vert}\int_{\vert y \vert}^{\vert y - y_1 \vert} \langle s_1 \rangle \langle s_2 \rangle \,\ud s_1 \,\ud s_2 \lesssim \langle x_1 \rangle^2 \langle y_2 \rangle^2\langle x \rangle \langle y \rangle,
\end{equation*}
the absolute boundedness of $S_1TP$, and the exponential decay of $v_1,v_2$, we deduce that 
\begin{equation}
\begin{split}
	\left \vert I_{2,2;1}^{(1,1)}(x,y) \right \vert \lesssim \vert t \vert^{-\frac{3}{2}}\langle x \rangle \langle y \rangle \int_{\R^2}\vert \langle x_1 \rangle^2 \langle y_2 \rangle^2 [\underline{e}_{11}v_1S_1TPv_2\underline{e}_{11}](x_1,y_1)\vert \,\ud x_1 \,\ud y_1 \lesssim \vert t \vert^{-\frac{3}{2}}\langle x \rangle \langle y \rangle.
\end{split}
\end{equation}
By Lemma~\ref{lemma: threshold resonance characterization} and direct computation,
\begin{equation}\label{eqn: S_1TPv_2}
	\int_{\R} S_1TP(x_1,y_1)v_2(y_1)\underline{e}_{11} \,\ud y_1 = \eta \Vert V_1 \Vert_{L^1(\R)}\Phi(x_1)[c_0\underline{e}_1]^*.
\end{equation}
Hence, integrating in $y_1$, we have
\begin{equation*}
	\begin{split}
		I_{2,2;2}^{(1,1)}(x,y)  &= \frac{\eta \Vert V_1 \Vert_{L^1(\R)}}{4i} \left(\int_{\R}\int_{\vert x \vert}^{\vert x - x_1 \vert} \int_{\R} e^{itz^2+iz(s_1 + \vert y \vert)} \chi_0(z^2)  \underline{e}_{11} v_1(x_1)\Phi(x_1) \,\ud z \,\ud s_1 \,\ud x_1\right)[c_0\underline{e}_1]^*\\
		&= \frac{\eta \Vert V_1 \Vert_{L^1(\R)}}{4i}\left(\int_{\R}\int_{\vert x \vert}^{\vert x - x_1 \vert} G_t(s_1 + \vert y \vert)\,\ud s_1 \underline{e}_{11} v_1(x_1)\Phi(x_1)\,\ud x_1\right)[c_0\underline{e}_1]^*,
	\end{split}
\end{equation*}
where $G_t$ is the function defined in Lemma~\ref{lemma: stationary phase fresnel}. By Lemma~\ref{lemma: stationary phase fresnel} (c.f.~\eqref{eqn: G_t(s_1+s_2)}--\eqref{eqn: I_2^{(1,1)}} for similar computations), we have
\begin{equation*}
	\left \vert I_{2,2;2}^{(1,1)}(x,y)  - \frac{i \eta \Vert V_1 \Vert_{L^1(\R)}}{2}    [G_0(\underline{e}_{11}v_1\Phi)(x)][e^{i\frac{y^2}{4t}} {c_0}\underline{e}_1]^* \right \vert \leq C \vert t \vert^{-\frac{3}{2}}\langle x \rangle^2 \langle y \rangle^2,
\end{equation*}
where $G_0$ is the operator defined in \eqref{eqn: def of G_0}. This completes the analysis of the term $I_{2,2}^{(1,1)}$. Next, we treat the term
\begin{equation}
	\begin{split}
		I_{2,2}^{(2,1)}(x,y) = \int_{\R^3} e^{itz^2}z \chi_0(z^2) \frac{e^{-\sqrt{z^2+2\mu}\vert x - x_1 \vert}}{-2\sqrt{z^2+2\mu}}[\underline{e}_{22}v_1S_1TPv_2\underline{e}_{11}](x_1,y_1)\frac{e^{iz \vert y - y_1 \vert}}{-2iz} \,\ud x_1 \,\ud y_1\,\ud z.
	\end{split}
\end{equation}
By inserting $e^{iz \vert y \vert}$, we write
\begin{equation*}
	\begin{split}
		I_{2,2}^{(2,1)}(x,y) &= -\frac{1}{2}\int_{\R^3} e^{itz^2}z \chi_0(z^2) \frac{e^{-\sqrt{z^2+2\mu}\vert x - x_1 \vert}}{-2\sqrt{z^2+2\mu}}[\underline{e}_{22}v_1S_1TPv_2\underline{e}_{11}](x_1,y_1) \int_{\vert y \vert}^{\vert y - y_1 \vert}e^{izs_2} \,\ud s_2 \,\ud x_1 \,\ud y_1\,\ud z\\
		&\quad + \int_{\R^3} e^{itz^2}z \chi_0(z^2) \frac{e^{-\sqrt{z^2+2\mu}\vert x - x_1 \vert}}{-2\sqrt{z^2+2\mu}}[\underline{e}_{22}v_1S_1TPv_2\underline{e}_{11}](x_1,y_1)\frac{e^{iz \vert y \vert}}{-2iz} \,\ud x_1 \,\ud y_1\,\ud z\\
		&=: I_{2,2;1}^{(2,1)}(x,y) + I_{2,2;2}^{(2,1)}(x,y),
	\end{split}
\end{equation*}
where $I_{2,2;2}^{(2,1)}$ is the leading term. By Lemma~\ref{lemma: van der corput estimate} and Lemma~\ref{lemma: bound on 2,2 resolvent},
\begin{equation}\label{eqn: R_0R_2 terms}
\left \vert 	\int_{\R} e^{itz^2+izs_2} z \chi_0(z^2) \frac{e^{-\sqrt{z^2+2\mu}\vert x - x_1 \vert}}{-2\sqrt{z^2+2\mu}}\,\ud z \right \vert \leq C \vert t \vert^{-\frac{3}{2}}\langle s_2 \rangle.
\end{equation}
Hence, using the absolute boundedness of $S_1TP$ and the bound \eqref{eqn: s_1 s_2 angle}, we have
\begin{equation*}
	\begin{split}
		\left \vert I_{2,2;1}^{(2,1)}(x,y) \right\vert \lesssim \vert t \vert^{-\frac{3}{2}} \int_{\R^2} \langle y_1 \rangle^2 \langle y \rangle[\underline{e}_{22}v_1S_1TPv_2\underline{e}_{11}](x_1,y_1)\,\ud x_1 \,\ud y_1 \lesssim \vert t \vert^{-\frac{3}{2}} \langle y \rangle.
	\end{split}
\end{equation*}
On the other hand, we treat $I_{2,2;1}^{(2,1)}$ similarly as in \eqref{eqn: I_{2,1}{2,2}} - \eqref{eqn: bound on error I_{2,1}^{2,2}} and find that 
\begin{equation*}
	\begin{split}
	\left \vert I_{2,2;2}^{(2,1)}(x,y) - \frac{i}{2}\int_{\R^3} e^{itz^2 + iz \vert y \vert}\chi_0(z^2)G_2(x,x_1)[\underline{e}_{22}v_1S_1TPv_2\underline{e}_{11}](x_1,y_1) \,\ud x_1 \,\ud y_1 \,\ud z \right \vert \leq C\vert t \vert^{-\frac{3}{2}} \langle x \rangle \langle y \rangle,		
	\end{split}
\end{equation*}
where $G_2$ is defined in \eqref{eqn: G_2(x,y)}. Hence, by Lemma~\ref{lemma: stationary phase fresnel} and \eqref{eqn: S_1TPv_2}, we conclude that 
\begin{equation}
	\left \vert I_{2,2}^{(2,1)}(x,y) - \frac{i\eta \Vert V_1 \Vert_{L^1(\R) }}{2} \frac{\sqrt{\pi}}{{\sqrt{-it}}}  [G_2(\underline{e}_{22}v_1\Phi)(x)][e^{i\frac{y^2}{4t}} {c_0}\underline{e}_1]^* \right \vert \leq C\vert t \vert^{-\frac{3}{2}} \langle x \rangle \langle y \rangle.
\end{equation}
Finally, we show that the terms $I_{2,2}^{(1,2)}$ and $I_{2,2}^{(2,2)}$ satisfy the better decay rates of $\calO(\vert t \vert^{-\frac{3}{2}}\langle x \rangle \langle y \rangle)$. By orthogonality (c.f.~\eqref{eqn: S_1 orthogonality trick}),
\begin{equation*}
	\begin{split}
	&I_{2,2}^{(1,2)}(x,y) \\
	& =	\frac{1}{-2}\int_{\R^3}e^{itz^2}z\chi_0(z^2) \int_{\vert x \vert}^{\vert x - x_1\vert} e^{izs_1}\,\ud s_1[\underline{e}_{11}v_1S_1TPv_2\underline{e}_{22}](x_1,y_1) \frac{e^{-\sqrt{z^2+2\mu}\vert y - y_1 \vert}}{-2\sqrt{z^2+2\mu}} \,\ud x_1\,\ud y_1\,\ud z.
	\end{split}
\end{equation*}
By Lemma~\ref{lemma: van der corput estimate} and Lemma~\ref{lemma: bound on 2,2 resolvent}, we note that the $z$-integral satisfy the bound
\begin{equation*}
	\left \vert  \int_{\R} e^{itz^2+izs_1} z \chi_0(z^2)\frac{e^{-\sqrt{z^2+2\mu}\vert y - y_1 \vert}}{-2\sqrt{z^2+2\mu}}\,\ud z \right \vert  \leq C\vert t \vert^{-\frac{3}{2}}\langle s_1 \rangle. 
\end{equation*}
Hence, by the absolute boundedness of $S_1TP$ and decay of $v_1$, $v_2$, we conclude that
\begin{equation*}
	\left \vert I_{2,2}^{(1,2)}(x,y) \right\vert \leq C\vert t \vert^{-\frac{3}{2}}\langle x \rangle.
\end{equation*}
The analysis of $I_{2,2}^{(2,2)}$ is analogous to the preceeding one, yielding the bound
\begin{equation*}
	\left \vert I_{2,2}^{(2,2)}(x,y) \right \vert \leq C \vert t \vert^{-\frac{3}{2}} \langle y \rangle.
\end{equation*}
Thus, using $\calG_0 = G_0 \underline{e}_{11} + G_2\underline{e}_{22}$, and $\calG_0(v_1\Phi) = c_0\underline{e}_1 - \Psi$ from Lemma \ref{lemma: threshold resonance characterization}, we conclude \eqref{eqn: I_{2,2}} and \eqref{eqn: F_t^2}. For the estimate \eqref{eqn: I_{2,3}} involving $I_{2,3}$, one should instead use the identity
	\begin{equation}
		\int_{\R} \underline{e}_{11}v_1(x_1)PTS_1(x_1,y_1) \,\ud x_1 = -\eta \Vert V_1 \Vert_{L^1(\R)} c_0\underline{e}_1 \Phi(y_1)^*,
	\end{equation}
and we leave the remaining details to the reader. 
\end{proof}
Next, we remark that the analysis for $I_{2,4}$ involving the operator $P$ leads to a similar estimate as the free evolution in Proposition~\ref{prop: free estimate}.
\begin{proposition}\label{prop: I_{2,4} estimate}
For all $\vert t \vert \geq 1$, we have
\begin{equation}\label{eqn: I_{2,4}}
	\left \vert I_{2,4}(x,y) - F_t^4(x,y) \right \vert \leq C \vert t\vert^{-\frac{3}{2}}\langle x \rangle^2 \langle y \rangle^2,
\end{equation} 
where
\begin{equation}\label{eqn: F_t^4}
	F_{t}^{4}(x,y) := \frac{\Vert V_1 \Vert_{L^1(\R)}}{4}\frac{\sqrt{\pi}}{{\sqrt{-it}}}e^{-i\frac{x^2}{4t}}\underline{e}_1 e^{-i\frac{y^2}{4t}}\underline{e}_1^\top.
\end{equation}
\end{proposition}
\begin{proof} As before, we write
	\begin{equation*}
		I_{2,4} = I_{2,4}^{(1,1)} + I_{2,4}^{(1,2)} + I_{2,4}^{(2,1)} + I_{2,4}^{(2,2)},
	\end{equation*}
with
\begin{equation*}
	I_{2,4}^{(i,j)} := \int_{\R} e^{itz^2} z^2 \chi_0(z^2) [\calR_i(z) v_1 P v_2 \calR_{j}(z)] \,\ud z, \quad i,j \in \{1,2\},
\end{equation*}
where $\calR_1$ and $\calR_{2}$ were defined in \eqref{eqn: def calR_1 calR_2}. We first treat the leading term
	\begin{equation}
		I_{2,4}^{(1,1)}(x,y) = \int_{\R} e^{itz^2} z^2 \chi_0(z^2) \frac{e^{iz \vert x - x_1 \vert}}{2iz}[\underline{e}_{11}v_1Pv_2\underline{e}_{11}](x_1,y_1)\frac{e^{iz \vert y - y_1 \vert}}{2iz} \,\ud x_1 \,\ud y_1\,\ud z.
	\end{equation}
	By adding and subtracting $e^{iz \vert x \vert}$ and $e^{iz \vert y \vert}$ twice, we further consider
	\begin{equation*}
		\begin{split}
			I_{2,4}^{(1,1)}(x,y) &= \int_{\R^3} e^{itz^2} z^2 \chi_0(z^2) \frac{e^{iz \vert x \vert}}{2iz}[\underline{e}_{11}v_1Pv_2\underline{e}_{11}](x_1,y_1)\frac{e^{iz \vert y  \vert}}{2iz} \,\ud x_1 \,\ud y_1\,\ud z \\
			&\quad + \frac{1}{2}\int_{\R^3} e^{itz^2} z^2 \chi_0(z^2) \frac{e^{iz \vert x \vert}}{2iz}[\underline{e}_{11}v_1Pv_2\underline{e}_{11}](x_1,y_1)\int_{\vert y \vert}^{\vert y - y_1 \vert}e^{izs_2} \,\ud s_2\,\ud x_1 \,\ud y_1\,\ud z\\
			&\quad + \frac{1}{2}\int_{\R^3} e^{itz^2} z^2 \chi_0(z^2) \int_{\vert x \vert}^{\vert x - x_1 \vert}e^{izs_1}ds_1[\underline{e}_{11}v_1Pv_2\underline{e}_{11}](x_1,y_1)\frac{e^{iz \vert y  \vert}}{2iz} \,\ud x_1 \,\ud y_1\,\ud z\\
			&\quad + \frac{1}{4}\int_{\R^3} e^{itz^2} z^2 \chi_0(z^2) \int_{\vert x \vert}^{\vert x - x_1 \vert}e^{izs_1}ds_1[\underline{e}_{11}v_1Pv_2\underline{e}_{11}](x_1,y_1)\int_{\vert y \vert}^{\vert y - y_1 \vert}e^{izs_2} \,\ud s_2 \,\ud x_1\,\ud y_1 \,\ud z\\
			&=: I_{2,4;1}^{(1,1)}(x,y) +I_{2,4;2}^{(1,1)}(x,y) +I_{2,4;3}^{(1,1)}(x,y) +I_{2,4;4}^{(1,1)}(x,y).
		\end{split}
	\end{equation*}
	By direct computation,
	\begin{equation*}
		\int_{\R^2} [\underline{e}_{11}v_1Pv_2\underline{e}_{11}](x_1,y_1)\,\ud x_1 \,\ud y_1 = - \Vert V_1  \Vert_{L^1(\R)}\underline{e}_1\underline{e}_1^\top.
	\end{equation*}
	Hence, 	by Lemma~\ref{lemma: stationary phase fresnel},
	\begin{equation*}
		\left \vert I_{2,4;1}^{(1,1)}(x,y) - \frac{\Vert V_1 \Vert_{L^1(\R)}}{4} \frac{\sqrt{\pi}}{{\sqrt{-it}}} e^{-i\frac{x^2}{4t}}\underline{e}_1e^{-i\frac{y^2}{4t}}\underline{e}_1^\top \right \vert \leq C\vert t \vert^{-\frac{3}{2}}\langle x \rangle \langle y \rangle.
	\end{equation*}
	For the terms $I_{2,4;2}^{(1,1)}$, $I_{2,4;3}^{(1,1)}$, the additional factor of $z$ allows to invoke Lemma~\ref{lemma: van der corput estimate},
	\begin{equation*}
		\begin{split}
		\left \vert \int_{\R} e^{itz^2+iz (\vert x \vert + s_2)}z\chi_0(z^2) \,\ud z \right \vert \leq C\vert t \vert^{-\frac{3}{2}}\langle x \rangle\langle s_2\rangle,\\
		\left \vert \int_{\R} e^{itz^2+iz (s_1 + \vert y \vert)}z\chi_0(z^2) \,\ud z \right \vert \leq C\vert t \vert^{-\frac{3}{2}}\langle  y \rangle \langle s_1\rangle .			
		\end{split}
	\end{equation*}
	Thus, we infer from the exponential decay of $v_1$ and $v_2$ that 
	\begin{equation}
		\left\vert I_{2,4;2}^{(1,1)}(x,y) \right \vert + \left\vert I_{2,4;3}^{(1,1)}(x,y) \right\vert \leq C\vert t \vert^{-\frac{3}{2}}\langle x \rangle \langle y \rangle.
	\end{equation}
	For the term $I_{2,4;4}^{(1,1)}$, we can use non-stationary phase to conclude the same bound. Hence, we have
	\begin{equation*}
		\left \vert I_{2,4}^{(1,1)}(x,y) -  \frac{\Vert V_1 \Vert_{L^1(\R)}}{4} \frac{\sqrt{\pi}}{{\sqrt{-it}}} e^{-i\frac{x^2}{4t}}\underline{e}_1e^{-i\frac{y^2}{4t}}\underline{e}_1^\top \right \vert \leq C\vert t \vert^{-\frac{3}{2}}\langle x \rangle \langle y \rangle.
	\end{equation*}
	Thus, it remains to prove that the other terms $I_{2,4}^{(1,2)}$, $I_{2,4}^{(2,1)}$, $I_{2,4}^{(2,2)}$ have the better $\calO(\vert t \vert^{-\frac{3}{2}}\langle x \rangle \langle y \rangle)$ weighted decay estimate to finish the proposition. We first treat the term
	\begin{equation}
		\begin{split}
	I_{2,4}^{(1,2)}(x,y) &=\frac{1}{2i} \int_{\R^3} e^{itz^2} z \chi_0(z^2) e^{iz \vert x - x_1 \vert} [\underline{e}_{11}v_1Pv_2\underline{e}_{22}](x_1,y_1)\frac{e^{-\sqrt{z^2+2\mu}\vert y - y_1 \vert}}{-2\sqrt{z^2+2\mu}} \,\ud x_1 \,\ud y_1 \,\ud z.\\
		\end{split}
	\end{equation}
	By Lemma~\ref{lemma: van der corput estimate} and Lemma~\ref{lemma: bound on 2,2 resolvent},
	\begin{equation*}
		\left \vert \int_{\R} e^{itz^2+iz(\vert x-x_1\vert)}z\chi_0(z^2) \frac{e^{-\sqrt{z^2+2\mu}\vert y - y_1 \vert}}{-2\sqrt{z^2+2\mu}}\,\ud z \right \vert \leq C \vert t \vert^{-\frac{3}{2}}\langle x \rangle \langle x_1 \rangle.
	\end{equation*}
	Hence, using the decay assumptions on $v_1$ and $v_2$, we conclude that 
	\begin{equation*}
		\left\vert I_{2,4}^{(1,2)}(x,y) \right\vert \leq C\vert t \vert^{-\frac{3}{2}}\langle x \rangle \langle y \rangle.
	\end{equation*}
	The same bound holds for the term $I_{2,4}^{(2,1)}$ and we will skip the details. Finally, we are left with
	\begin{equation*}
		I_{2,4}^{(2,2)}(x,y) = \int_{\R^3} e^{itz^2} z^2 \chi_0(z^2) \frac{e^{-\sqrt{z^2+2\mu}\vert x - x_1 \vert}}{-2\sqrt{z^2+2\mu}}[\underline{e}_{22}v_1Pv_2\underline{e}_{22}](x_1,y_1)\frac{e^{-\sqrt{z^2+2\mu}\vert y - y_1 \vert}}{-2\sqrt{z^2+2\mu}} \,\ud x_1 \,\ud y_1 \,\ud z.
	\end{equation*}
	By direct computation using \eqref{identity V_1 V_2},
	\begin{equation*}
		[\underline{e}_{22}v_1Pv_2\underline{e}_{22}](x_1,y_1) = \frac{1}{\Vert V_1 \Vert_{L^1(\R)}}[V_2\underline{e}_2](x_1)[V_2\underline{e}_{2}]^\top(y_1),
	\end{equation*}
and by Lemma~\ref{lemma: van der corput estimate} and Lemma~\ref{lemma: bound on 2,2 resolvent}, we have the uniform estimate
	\begin{equation*}
		\left \vert \int_{\R}e^{itz^2} z^2 \chi_0(z^2) \frac{e^{-\sqrt{z^2+2\mu}\vert x - x_1 \vert}}{-2\sqrt{z^2+2\mu}}\frac{e^{-\sqrt{z^2+2\mu}\vert y - y_1 \vert}}{-2\sqrt{z^2+2\mu}}dz \right \vert \leq C_\mu \vert t \vert^{-\frac{3}{2}}.
	\end{equation*}
	Hence, by exchanging the order of integration, we conclude that 
	\begin{equation}
		\left\vert I_{2,4}^{(2,2)}(x,y) \right\vert \leq C \vert t \vert^{-\frac{3}{2}}.
	\end{equation}
	Thus, we conclude \eqref{eqn: I_{2,4}} by summing over the four terms.
\end{proof}
Finally, we are ready to complete the proof of the local decay estimate \eqref{eqn: low energy weighted estimate}. 
We sum the leading contributions of the spectral representation of $e^{it\calH}\chi_0(\calH - \mu I)P_{\mathrm{s}}^+$ in \eqref{eqn: spectral representation low energy} by invoking Proposition~\ref{prop: free estimate}, Proposition~\ref{prop: I_2 estimate}, Proposition~\ref{prop: I_{2,3}}, and Proposition~\ref{prop: I_{2,4} estimate} to obtain
\begin{equation*}
	\begin{split}
	&F_t^0 - \frac{e^{it\mu}}{\pi i}\left(\frac{i}{2\eta } F_t^1  + \frac{1}{\eta\Vert V_1 \Vert_{L^1(\R)}} F_t^2 + \frac{1}{\eta\Vert V_1 \Vert_{L^1(\R)}} F_t^3 + \left(\frac{2i}{\Vert V_1 \Vert_{L^1(\R)}}+\frac{2  \vert c_0 \vert^2 }{i\Vert V_1 \Vert_{L^1(\R)} }\right) F_t^4 \right)\\
	&=	\frac{e^{it\mu} }{\sqrt{-4 \pi i t}} \left(  - [c_0 \underline{e}_1 - \Psi(x)][\sigma_3 \Psi(y) -  {c_0} \underline{e}_1]^* - [c_0 \underline{e}_1 - \Psi(x)][e^{ i\frac{y^2}{4t}} {c_0}\underline{e}_1]^*\right.\\
	&\left.\hspace{5cm} + [e^{-i\frac{x^2}{4t}}c_0 \underline{e}_{1}][\sigma_3 \Psi(y) -  {c_0} \underline{e}_1]^* + \vert c_0 \vert^2 e^{-i\frac{x^2}{4t}} e^{-i\frac{y^2}{4t}}\underline{e}_1\underline{e}_1^\top   \right)\\
	&= \frac{e^{it\mu} }{\sqrt{-4 \pi i t}} \left(  \Psi(x) [\sigma_3 \Psi(y)]^* + (e^{-i \frac{x^2}{4t}}-1)c_0 [\sigma_3 \Psi(y)]^* + (e^{-i \frac{y^2}{4t}}-1) \Psi(x) [ {c_0}\underline{e}_1]^*  \right. \\
	&\left. \hspace{5cm} + (1-e^{-i \frac{x^2}{4t}} - e^{-i \frac{y^2}{4t}} + e^{-i\frac{x^2}{4t}} e^{-i\frac{y^2}{4t}})\vert c_0 \vert^2   \underline{e}_1\underline{e}_1^\top \right),
	\end{split}
\end{equation*}
where we use the cancellation $F_t^0 - \frac{e^{it\mu}}{\pi i} \frac{2i}{\Vert V_1 \Vert_{L^1(\R)}} F_t^4 = 0$ in the first equality. We note that the first term gives us the finite rank operator 
\begin{equation}
	F_t^+(x,y) = \frac{e^{it\mu} }{\sqrt{-4 \pi i t}} \Psi(x) [\sigma_3 \Psi(y)]^*,
\end{equation}
and we show that the last three terms satisfy the better decay rate. Using, 
\begin{equation}\label{eqn: x^2/t bound}
	\vert 1 - e^{-i \frac{x^2}{4t}} \vert \leq \frac{x^2}{4\vert t\vert},
\end{equation}
and the fact that $\Psi \in L^{\infty}(\R) \times L^\infty(\R)$, we have
\begin{equation*}
\left \vert 	\frac{e^{it\mu} e^{i\frac{\pi}{4}} }{2\sqrt{\pi} \sqrt{t}}(e^{-i \frac{x^2}{4t}} - 1)c_0 \underline{e}_1[\sigma_3 \Psi(y)]^* \right \vert \lesssim \vert t \vert^{-\frac{3}{2}} \langle x \rangle^2,
\end{equation*}
and similarly
\begin{equation*}
	\left \vert \frac{e^{it\mu} e^{i\frac{\pi}{4}} }{2\sqrt{\pi} \sqrt{t}  }(e^{-i \frac{y^2}{4t}}-1) \overline{c_0}\Psi(x) \underline{e}_1^\top \right \vert \lesssim \vert t \vert^{-\frac{3}{2}}  \langle y \rangle^2.
\end{equation*}
For the last term, we have
\begin{equation}
\left\vert 1-e^{-i \frac{x^2}{4t}} - e^{-i \frac{y^2}{4t}} + e^{-i\frac{x^2}{4t}} e^{-i\frac{y^2}{4t}} \right\vert = \left\vert 1 - e^{-i \frac{x^2}{4t}} \right\vert \left\vert 1 - e^{-i \frac{y^2}{4t}} \right\vert \lesssim \vert t \vert^{-2}\langle x \rangle^2 \langle y \rangle^2.
\end{equation}
Thus, the leading contribution to $e^{it\calH}\chi_0(\calH - \mu I)P_{\mathrm{s}}^+$ is $F_t^+$. 
\end{proof}

\section{Intermediate and high energy estimates}	
In order to complete the proof of Theorem~\ref{theorem: local decay estimate}, we also need to prove the dispersive estimates when the spectral variable is bounded away from the thresholds $\pm \mu$. As usual, we focus on the positive semi-axis $[\mu,\infty)$ of the essential spectrum and prove the dispersive estimates for energies $\lambda > \mu$. The negative semi-axis $(-\infty,-\mu]$ can be treated by symmetry of $\calH$. We recall from Section 2 that the kernel of the limiting resolvent operator for $\calH_0$ has the formula
\begin{equation}\label{eqn: free resolvent 2}
	\calR_0^\pm(z)(x,y) := (\calH_0-(z^2+\mu\pm i0))^{-1} = \begin{bmatrix}
		\pm \frac{ie^{\pm i z \vert x -y \vert} }{2 z} & 0 \\ 0 & -\frac{e^{-\sqrt{z^2+2\mu}\vert x - y \vert }}{2 \sqrt{z^2 + 2\mu}}
	\end{bmatrix}, \quad\forall\  0 < z <\infty.
\end{equation}
From this, we have the following bound
\begin{equation*}
	\Vert \calR_0^{\pm}(z) \Vert_{L^1 \times L^1 \to L^\infty \times L^\infty} \leq C \vert z \vert^{-1}.
\end{equation*}
Hence, for sufficiently large $z$, the perturbed resolvent $\calR^\pm(z)$ can be expanded into the infinite Born series
\begin{equation}\label{eqn: Born series}
	\calR^\pm(z) = \sum_{n=0}^\infty \calR_0^\pm(z)(-\calV\calR_0^\pm(z))^n.
\end{equation}
More precisely, since the operator norm  $L^1 \times L^1 \to L^\infty \times L^\infty$ in the $n$-th summand above is bounded by $C \vert z\vert ^{-1} (C\Vert \calV \Vert_1  \vert z\vert^{-1})^{n}$, the Born series converges in the operator norm  whenever $\vert z\vert  > z_1 := 2C\Vert \calV \Vert_{L^1 \times L^1}$. We define the high-energy cut-off by
\begin{equation}
	\chi_{\mathrm{h}}(z) := 1-\chi(z),	
\end{equation}
 where $\chi(z)$ is a standard smooth even cut-off supported on $[-z_1,z_1]$ satisfying $\chi(z) = 1$ for $\vert z \vert \leq \frac{z_1}{2}$ and $\chi(z) = 0$ for $\vert z \vert \geq z_1$. We insert the cut-off and the Born series expansion into the spectral representation $e^{it\calH} \chi_{\mathrm{h}}(\calH-\mu I) P_{\mathrm{s}}^+$ and look to bound the following 
\begin{equation}
	\begin{split}
	\vert \langle e^{it\calH} \chi_{\mathrm{h}}(\calH-\mu I) P_{\mathrm{s}}^+ \vec{u},\vec{v} \rangle \vert  &= \left \vert \int_0^\infty e^{itz^2}z \chi_{\mathrm{h}}(z^2) \langle [\calR^+(z) - \calR^-(z)]\vec{u},\vec{v}\rangle \,\ud z \right \vert \\
	&\leq C \sum_\pm  \sum_{n=0}^\infty \left \vert \int_0^\infty e^{itz^2}z \chi_{\mathrm{h}}(z^2)  \langle \calR_0^{\pm}(z)(\calV \calR_0^\pm(z))^{n}\vec{u},\vec{v}\rangle \,\ud z\right \vert,
	\end{split}
\end{equation}
where $\vec{u},\vec{v}\in \calS(\R) \times \calS(\R)$. From \cite{06KriegerSchlag}, we have the following dispersive estimates:
\begin{proposition}\label{prop: high energy estimate} Under the same hypothesis as Theorem~\ref{theorem: local decay estimate}, we have 
\begin{equation}\label{eqn: unweighted high energy}
	\left\Vert e^{it\calH} \chi_{\mathrm{h}}(\calH-\mu I)P_{\mathrm{s}}^+ \vec{u}\, \right\Vert_{L^\infty(\R)\times L^\infty(\R)} \lesssim \vert t \vert^{-\frac{1}{2}} \left\Vert \vec{u}\,  \right\Vert_{L^1(\R) \times L^1(\R)},
\end{equation}	
and
\begin{equation}\label{eqn: weighted high energy}
	\left\Vert \langle x \rangle^{-1}e^{it\calH} \chi_{\mathrm{h}}(\calH-\mu I) P_{\mathrm{s}}^+\vec{u} \, \right\Vert_{L^\infty(\R)\times L^\infty(\R)} \lesssim \vert t \vert^{-\frac{3}{2}} \left\Vert\langle x \rangle \vec{u} \, \right\Vert_{L^1(\R) \times L^1(\R)},
\end{equation}	
for any $\vert t \vert \geq 1$.
\end{proposition}
\begin{proof}
	For \eqref{eqn: unweighted high energy}, see the proof of \cite[Proposition 7.1]{06KriegerSchlag}, and for \eqref{eqn: weighted high energy}, see the proof of \cite[Proposition 8.1]{06KriegerSchlag}. Note that the high-energy dispersive estimate holds irrespective of the regularity of the thresholds $\pm \mu$. 
\end{proof}
Let $z_0>0$ be the constant from Proposition~\ref{prop: invert M}. It may happen that $z_1$ is strictly larger than $z_0$. In this case, we need to derive estimates analogous to the above proposition in the remaining intermediate energy regime $[-z_1,-z_0]\cup[z_0,z_1]$. To this end, we set $\chi_{\mathrm{m}}(z)$ to be the intermediate energy cut-off given by 
\begin{equation}
	\chi_{\mathrm{m}}(z) := 1 - \chi_0(z) - \chi_{\mathrm{h}}(z),
\end{equation}
where $\chi_0(z)$ was the cut-off defined in the previous section in Proposition~\ref{prop: low energy bounds}. 
\begin{proposition}\label{prop: inter energy estimates}  For any $\vert t \vert \geq 1$, we have 
	\begin{equation}\label{eqn: unweighted inter energy}
		\left\Vert e^{it\calH} \chi_{\mathrm{m}}(\calH-\mu I)P_{\mathrm{s}}^+ \vec{u}\, \right\Vert_{L_x^\infty(\R)\times L_x^\infty(\R)} \lesssim \vert t \vert^{-\frac{1}{2}} \left\Vert \vec{u}\,  \right\Vert_{L_x^1(\R) \times L_x^1(\R)},
	\end{equation}	
	and
	\begin{equation}\label{eqn: weighted inter energy}
		\left\Vert \langle x \rangle^{-1}e^{it\calH} \chi_{\mathrm{m}}(\calH-\mu I) P_{\mathrm{s}}^+\vec{u}\, \right\Vert_{L_x^\infty(\R)\times L_x^\infty(\R)} \lesssim \vert t \vert^{-\frac{3}{2}}\left\Vert\langle x \rangle \vec{u}\, \right\Vert_{L_x^1(\R) \times L_x^1(\R)}.
	\end{equation}	
\end{proposition}
Before proving the above proposition, we need the following lemmas for pointwise bounds and operator norm bounds on the resolvent operators and its derivatives. The first lemma follows immediately from the expression \eqref{eqn: free resolvent 2} and the triangle inequality $\vert \vert x - x_1 \vert - \vert x \vert \vert \leq \vert x_1 \vert$.
\begin{lemma}\label{lemma: bound on free resolvent} Let $\gamma_0 > 0$. For every $ z > \gamma_0$, and $k \in \{0,1,2\}$, we have
	\begin{equation}
		\left\vert \partial_z^k \calR_0^\pm(z)(x,y) \right\vert \leq C \gamma_0^{-1-k}\langle x - y \rangle^k, 
	\end{equation}
	and hence
	\begin{equation}
		\left\Vert \partial_z^k \calR_0^\pm(z)(x,\cdot) \right\Vert_{X_{-(\frac{1}{2}+k)-}} \leq C \gamma_0^{-1-k}\langle x \rangle^{k}.
	\end{equation}
Moreover, define 
\begin{equation}\label{eqn: def of calGpm}
	\calG_\pm(z)(x,x_1) = \begin{bmatrix}
		e^{ \mp i  z \vert x \vert} & 0 \\ 0 & 1
	\end{bmatrix}\calR_0^\pm(z)(x,x_1)=\begin{bmatrix}
	\pm \frac{ie^{\pm i z (\vert x - x_1 \vert - \vert x \vert)} }{2 z} & 0 \\ 0 & -\frac{e^{-\sqrt{z^2+2\mu}\vert x - x_1 \vert }}{2 \sqrt{z^2 + 2\mu}}
\end{bmatrix}.
\end{equation}
Then, for any $k \geq 0$,
\begin{equation}\label{eqn: bound on calGpm}
	\sup_{x \in \R} \left\vert \partial_z^k \calG^\pm(z)(x,x_1)\right \vert \leq C \gamma_0^{-1-k}\vert x_1 \vert.
\end{equation}
\end{lemma}
With these bounds, we are able to give operator norm bounds on the perturbed resolvent via the resolvent identity. 
\begin{lemma}\label{lemma: bound on perturbed resolvent} Let $\gamma_0 > 0$. We have
	\begin{equation}\label{eqn: derivative perturbed resolvent}
		\sup_{\vert z \vert > \gamma_0} \left\Vert \partial_z \calR^\pm(z) \right\Vert_{X_{\frac{3}{2}+}  \to X_{-\frac{3}{2}-}}   \lesssim 1,
	\end{equation}	
	\begin{equation}\label{eqn: bound on 2nd derivative perturbed resolvent}
				\sup_{\vert z \vert > \gamma_0} \left\Vert \partial_z^2 \calR^\pm(z) \right\Vert_{X_{\frac{5}{2}+} \to X_{-\frac{5}{2}-}}  \lesssim 1.
	\end{equation}
\end{lemma}
\begin{proof}
	By Lemma \ref{lemma: LAP}, for any $\vert z \vert > \gamma_0$, we have
	\begin{equation}
		\calR^\pm(z) = (I+\calR_0^\pm(z)\calV)^{-1}\calR_0^\pm(z) =: S^\pm(z)^{-1}\calR_0^\pm(z),
	\end{equation}
as a bounded operator from $X_{\frac{1}{2}+}$ to $X_{-\frac{1}{2}-}$. Note that $S^\pm(z)$ is boundedly invertible on $X_{-\sigma}$ for any $\sigma>0$. By differentiation, we have
\begin{equation}
	\partial_z \calR^\pm(z) = -S^\pm(z)^{-1}\partial_z\calR_0^\pm(z) \calV S^\pm(z)^{-1} \calR_0^\pm(z) + S^\pm(z)^{-1}\partial_z \calR_0^\pm(z).
\end{equation}
Moreover, as a multiplication operator, $\calV:X_{-\sigma} \to X_{\sigma}$ is bounded for any $\sigma>0$ due to the exponential decay of $\calV$. By Lemma~\ref{lemma: bound on free resolvent}, $\partial_zR_0^\pm(z):X_{\frac{3}{2}+} \to X_{-\frac{3}{2}-}$ is bounded and since the embedding $X_{-\frac{1}{2}-}\subset X_{-\frac{3}{2}-}$ is continuous, we infer the bound \eqref{eqn: derivative perturbed resolvent} by taking composition. By a similar argument, 
\begin{equation}
	\Vert \partial_z^2 \calR^\pm(z) \Vert_{X_{\frac{5}{2}+} \to X_{-\frac{5}{2}-}} \lesssim 1.
\end{equation}
\end{proof}
\begin{proof}[Proof of Proposition~\ref{prop: inter energy estimates}]
By iterating the second resolvent identity, we write the perturbed resolvent as a finite sum
\begin{equation}\label{eqn: twice expanded resolvent}
	\calR^\pm(z) = \calR_0^\pm(z) - \calR_0^\pm(z)\calV\calR_0^\pm(z) + \calR_0^\pm(z)\calV\calR^\pm(z)\calV\calR_0^\pm(z),
\end{equation}
and we write
	\begin{equation}
		e^{it\calH}\chi_{\mathrm{m}}(\calH - \mu I)P_{\mathrm{s}}^+(x,y) = \sum_{j=1}^3 \int_0^\infty e^{itz^2}z\chi_{\mathrm{m}}(z^2)(-1)^{j+1}(\calE_{j}^+(z) - \calE_{j}^-(z))(x,y)dz,
	\end{equation}
with
\begin{equation}
\calE_1^\pm(z) = \calR_0^\pm(z),\quad  \calE_2^\pm(z) = \calR_0^\pm(z)\calV\calR_0^\pm(z),\quad \calE_3^\pm(z) = \calR_0^\pm(z)\calV\calR^\pm(z)\calV\calR_0^\pm(z).
\end{equation}
Hence, to prove \eqref{eqn: unweighted inter energy} and \eqref{eqn: weighted inter energy}, it is sufficient to establish the estimates 
\begin{equation}\label{eqn: estimate for calE_j}
	\sup_{\pm}\sup_{j=1,2,3}\left\vert \int_0^\infty e^{itz^2}z\chi_{\mathrm{m}}(z^2)\calE_{j}^\pm(z)(x,y) \,\ud z \right \vert \lesssim \min\{\vert t \vert^{-\frac{1}{2}},\vert t \vert^{-\frac{3}{2}}\langle x \rangle \langle y \rangle\}.
\end{equation}
The term involving $\calE_1^\pm$ is handled by the earlier Proposition~\ref{prop: free estimate}, while the second term involving $\calE_2^\pm$ can be treated analogously as in Proposition~\ref{prop: high energy estimate}. We refer the reader to \cite[Lemma 3]{04GoldbergSchlag} and \cite[Proposition 3]{07Goldberg} for similar computations. For the term involving $\calE_3^\pm$, we first write
\begin{equation*}
	\calR_0^\pm(z)(s_1,s_2) = \begin{bmatrix}
		e^{\pm i z\vert s_1 \vert} & 0 \\ 0 & 1
	\end{bmatrix}\calG_\pm(z)(s_1,s_2),
\end{equation*}
where the operator $\calG_\pm(z)$ was defined in \eqref{eqn: def of calGpm}. Then, using that the kernel $\calR_0^\pm(z)(x,y)$ is symmetric in $x$ and $y$ variables, and using the matrix identity
\begin{equation}
	\underline{e}_{jj} \begin{bmatrix}
		a_{11} & a_{12} \\ a_{21} & a_{22}
	\end{bmatrix}\underline{e}_{kk} = a_{jk}\underline{e}_j\underline{e}_k^\top, \quad j,k \in \{1,2\},
\end{equation}
we compute the following kernel identity 
\begin{equation*}
	\begin{split}
	\calE_3^\pm(z)(x,y) &= \int_{\R^2} \calR_0^\pm(x,x_1)[\calV\calR^\pm(z)\calV](x_1,y_1)\calR_0^\pm(y,y_1)\,\ud x_1 \,\ud y_1\\
	&= \begin{bmatrix}
		e^{\pm iz \vert x \vert} & 0 \\ 0 & 1
	\end{bmatrix}\int_{\R^2} \calG^\pm(x,x_1)[\calV\calR^\pm(z)\calV](x_1,y_1)\calG^\pm(y,y_1)\,\ud x_1 \,\ud y_1\begin{bmatrix}
	e^{\pm iz \vert y \vert} & 0 \\ 0 & 1
\end{bmatrix}\\
&= e^{\pm iz (\vert x \vert + \vert y \vert)}\langle (\calG^\pm)^*(z)(x,\cdot)\underline{e}_1,\calV \calR^\pm(z)\calV \calG^\pm(z)(y,\cdot)\underline{e}_1\rangle\, \underline{e}_{1}\underline{e}_1^\top \\
&\quad +  e^{\pm iz \vert x \vert}\langle (\calG^\pm)^*(z)(x,\cdot)\underline{e}_2,\calV \calR^\pm(z)\calV \calG^\pm(z)(y,\cdot)\underline{e}_1\rangle\, \underline{e}_{1}\underline{e}_2^\top \\
&\quad +  e^{\pm iz \vert y \vert}\langle (\calG^\pm)^*(z)(x,\cdot)\underline{e}_1,\calV \calR^\pm(z)\calV \calG^\pm(z)(y,\cdot)\underline{e}_2\rangle\, \underline{e}_{2}\underline{e}_1^\top \\
&\quad + \langle (\calG^\pm)^*(z)(x,\cdot)\underline{e}_2,\calV \calR^\pm(z)\calV \calG^\pm(z)(y,\cdot)\underline{e}_2\rangle\, \underline{e}_{2}\underline{e}_2^\top\\
&=:   e^{\pm iz (\vert x \vert + \vert y \vert)} A_{1}^\pm(z,x,y) + e^{\pm iz \vert x \vert} A_{2}^\pm(z,x,y) + e^{\pm iz \vert y \vert}A_3^\pm(z,x,y) + A_4^\pm(z,x,y).
	\end{split}
\end{equation*}
We plug this identity into the left hand side of \eqref{eqn: estimate for calE_j}, and hence it will be sufficient to provide the bounds
\begin{equation}\label{eqn: bound on A_k}
	\left\vert  \int_{0}^\infty e^{itz^2 \pm iz r }z\chi_{\mathrm{m}}(z^2) A_{k}^\pm(z,x,y)\,\ud z \right \vert \lesssim \min\{\vert t \vert^{-\frac{1}{2}},\vert t \vert^{-\frac{3}{2}}\langle r \rangle\}, \quad k\in \{1,\ldots,4\},
\end{equation}
where $r$ can represent $0$ or $\vert x\vert $, $\vert y\vert$, or the sum of both variables. For the case $k=1$, by Lemma~\ref{lemma: van der corput estimate}, we have that
\begin{equation*}
	\left\vert  \int_{0}^\infty e^{itz^2 \pm  iz (\vert x \vert + \vert y \vert) }z\chi_{\mathrm{m}}(z^2) A_{1}^\pm(z,x,y)\,\ud z \right \vert  \leq C  \vert t \vert^{-\frac{1}{2}}\Vert \partial_z \left(z\chi_{\mathrm{m}}(z^2) A_{1}^\pm(z,x,y) \right)\Vert_{L_z^1(\R)}.
\end{equation*}
Since the term $z\chi_{\mathrm{m}}(z^2)$ is smooth and has compact support, we only need to track the derivatives when they fall onto either $\calG^\pm(z)$ or $\calR^\pm(z)$. In any case, thanks to the exponential decay of $\calV$, and the bounds \eqref{eqn: bound on calGpm}, \eqref{eqn: derivative perturbed resolvent} from the previous lemmas, we have the following uniform bound
\begin{equation}\label{eqn: bound on calG}
	\begin{split}
		&\sup_{\pm}\sup_{z \in \supp (\chi_{\mathrm{m}})}\sup_{j,k =1,2} \vert \partial_z \langle (\calG^\pm)^*(z)(y,\cdot)\underline{e}_j, \calV \calR^\pm(z)\calV \calG^\pm(z)(x,\cdot)\underline{e}_k\rangle\vert \\
		&\lesssim\sup_{\pm}\sup_{z \in \supp (\chi_{\mathrm{m}})}\sup_{j,k =1,2} \left\Vert \sqrt{\vert \calV\vert}(x_1) \left(\vert \calR^\pm(z)(x_1,x_2) \vert + \vert \partial_z \calR^\pm(z)(x_1,x_2) \vert\right) \sqrt{\vert \calV\vert}(x_2) \right\Vert_{L_{x_2}^2 \to L_{x_1}^2} \\
		&\quad \cdot \Vert\sqrt{\vert \calV\vert}(x_1) \left(\vert \calG^\pm(z)(x,x_1)\vert + \vert \partial_z\calG^\pm(z)(x,x_1)\vert\right) \underline{e}_j\Vert_{L_{x_1}^2} \\
		&\quad \cdot \Vert \sqrt{\vert \calV\vert}(x_2)\left(\vert \calG^\pm(z)(x_2,y)\vert + \vert \partial_z\calG^\pm(z)(x_2,y)\vert\right) \underline{e}_k\Vert_{L_{x_2}^2}\\
		&\lesssim 1,
	\end{split}
\end{equation}
for all $x,y \in \R$. 

To prove the weighted dispersive estimate, we invoke the stronger estimate in Lemma~\ref{lemma: van der corput estimate}:
\begin{equation*}
	\left\vert  \int_{0}^\infty e^{itz^2 \pm  iz (\vert x \vert + \vert y \vert) }z\chi_{\mathrm{m}}(z^2) A_{1}^\pm(z,x,y)\,\ud z \right \vert  \leq C  \vert t \vert^{-\frac{3}{2}} \left\Vert [\partial_z^2 \pm i(\vert x \vert + \vert y \vert)\partial_z]  \left(\chi_{\mathrm{m}}(z^2) A_{1}^\pm(z,x,y) \right)\right\Vert_{L_z^1(\R)}
\end{equation*}
Here, we can apply the same argument as in \eqref{eqn: bound on calG} for the two derivatives bound on $A_1^\pm$ using the estimates \eqref{eqn: bound on calGpm} and \eqref{eqn: bound on 2nd derivative perturbed resolvent}, whereas the bound on one derivative for $A_1^\pm$ leads to the weights $\langle x \rangle \langle y \rangle$. Thus, we prove \eqref{eqn: bound on A_k} for $k=1$. The other cases follow by the same argument and we are done. 
\end{proof}
Finally, we conclude with the proof of Theorem~\ref{theorem: local decay estimate}.
\begin{proof}[Proof of Theorem~\ref{theorem: local decay estimate}]
	By combining the estimates from Proposition~\ref{prop: low energy bounds}, Proposition~\ref{prop: high energy estimate}, and Proposition~\ref{prop: inter energy estimates}, we have established the bounds 
	\begin{equation*}
		\left\Vert e^{it\calH}P_\mathrm{s}^+ \vecu \, \right \Vert_{L_x^\infty(\R)\times L_x^\infty(\R)} \lesssim \vert t \vert^{-\frac{1}{2}} \Vert \vecu \, \Vert_{L_x^1(\R) \times L_x^1(\R)},
	\end{equation*}
	as well as
	\begin{equation*}
		\left\Vert \langle x \rangle^{-2} (e^{it\calH}P_\mathrm{s}^+  - F_t^+)\vecu \, \right \Vert_{L_x^\infty(\R)\times L_x^\infty(\R)} \lesssim \vert t \vert^{-\frac{3}{2}} \Vert \vecu \, \Vert_{L_x^1(\R) \times L_x^1(\R)},
	\end{equation*}
	for any $\vecu := (u_1,u_2)^\top \in \calS(\R) \times \calS(\R)$ and $\vert t \vert \geq 1$, with $F_t^+$ given by \eqref{eqn: formula F_t^+}. By  Remark~\ref{remark: Ps^-}, we can similarly deduce that the unweighted dispersive estimate for the evolution $e^{it\calH}P_\mathrm{s}^-$ using the identity \eqref{eqn: P_s^-}. On the other hand, for the weighted estimate, we find that the leading contribution to $e^{it\calH}P_{\mathrm{s}}^{-}$ is given by 
	\begin{equation}
		F_t^{-}(x,y) = \sigma_1 F_{-t}^{+}(x,y)\sigma_1 =  -\frac{e^{-it\mu} }{\sqrt{4 \pi i t}}[\sigma_1\Psi(x)][\sigma_3\sigma_1\Psi(y)]^*,
	\end{equation}
	where we used the anti-commutation identity $\sigma_3 \sigma_1 = - \sigma_1 \sigma_3$. Thus, we conclude the local decay estimate \eqref{eqn: theorem local decay estimate} and the formula \eqref{eqn: def of F_t} by setting $F_t := F_t^+ + F_t^-$.	
\end{proof}

	\begin{appendix}
	\section{Neumann series}
\begin{lemma} \label{lemma: Neumann series}
	Let $A$ be an invertible operator and $B$ be a bounded operator satisfying $\Vert B \Vert < \Vert A^{-1} \Vert^{-1}$. Then, $A-B$ is invertible with
	\begin{equation}\label{eqn: appendix neumann1}
		(A-B)^{-1} =A^{-1}\sum_{n=0}^\infty (BA^{-1})^{n} =  A^{-1} + A^{-1}BA^{-1} + A^{-1}BA^{-1}BA^{-1} + \cdots,
	\end{equation}
	and
	\begin{equation}\label{eqn: appendix neumann2}
		\Vert (A-B)^{-1} \Vert \leq (\Vert A^{-1}\Vert^{-1} - \Vert B \Vert)^{-1}.
	\end{equation}
\begin{proof}
By the hypothesis $\Vert B \Vert < \Vert A^{-1}\Vert^{-1}$, we have $\Vert A^{-1}B\Vert <1$. 
Consider the identity
\begin{equation*}
	(A-B)^{-1} = (I-A^{-1}B)^{-1}A^{-1}.
\end{equation*}
The term on the right hand side can be written in the usual Neumann series
\begin{equation*}
	(I-A^{-1}B)^{-1} = \sum_{n=0}^\infty (A^{-1}B)^{n}.
\end{equation*}
Thus, by multiplying $A^{-1}$, we deduce \eqref{eqn: appendix neumann1}. Note that the argument also holds true for $(A-B)^{-1} = A^{-1}(I-BA^{-1})^{-1}$. Now, since we have the estimate
\begin{equation*}
	\Vert (I- A^{-1}B)^{-1}\Vert \leq (1 - \Vert A^{-1}B\Vert)^{-1},
\end{equation*}
we deduce \eqref{eqn: appendix neumann2} by the sub-multiplicative property for operator norms.
\end{proof}
\end{lemma}		 
	\end{appendix}

\bibliographystyle{alpha}
	\bibliography{references}

\end{document}